\documentclass[a4paper,10pt,reqno]{amsart}

\usepackage{color}
\usepackage{cite}
\usepackage{hyperref}
\usepackage{amsmath,amssymb,amsthm}
\usepackage{mleftright}
\mleftright
\hypersetup{linktocpage,colorlinks,linkcolor={red},citecolor={blue}}
\usepackage[english]{babel}
\usepackage[T1]{fontenc}
\usepackage[utf8]{inputenc}
\usepackage{dsfont}
\usepackage{esint}
\usepackage{enumitem}
\usepackage[a4paper,left=2.6cm,right=2.6cm,top=3cm,bottom=3.5cm]{geometry}

\newcommand{\RomanNumeralCaps}[1]{\MakeUppercase{\romannumeral #1}}

\usepackage{lmodern}
\numberwithin{equation}{section}

\setlist{noitemsep}

\newtheorem{theorem}{Theorem}

\newtheorem{proposition}{Proposition}[section]
\newtheorem{lemma}[proposition]{Lemma}
\newtheorem{corollary}{Corollary}

\newtheorem{remark}{Remark}[section]%
\newtheorem{assumption}[remark]{Assumption}

\newcommand{\dQ}{\mathbb{Q}}
\newcommand{\dR}{\mathbb{R}}
\newcommand{\dT}{\mathbb{T}}
\newcommand{\ve}{\varepsilon}

\newcommand{\dP}{\mathbb{P}}
\newcommand{\dE}{\mathbb{E}}
\newcommand{\Var}{\operatorname{Var}}
\newcommand{\Cov}{\operatorname{Cov}}

\newcommand{\Error}{\mathrm{Error}}
\newcommand{\sgn}{\mathrm{sgn}}
\newcommand{\mc}{\mathcal}

\newcommand{\dive}{\mathrm{div}}
\newcommand{\dd}{\mathop{}\!\mathrm{d}}

\newcommand{\Gap}{\mathrm{Gap}}
\newcommand{\Hc}{\mathcal{H}}
\newcommand{\fluct}{\mathrm{fluct}}
\newcommand{\Fluct}{\mathrm{Fluct}}
\newcommand{\dGi}{\dP_{N,\beta}}
\newcommand{\dGiQ}{\dQ_{N,\beta}}
\newcommand{\dV}{\mathrm{V}}
\newcommand{\dW}{\mathrm{W}}
\newcommand{\reg}{\mathrm{reg}}
\newcommand{\FF}{\mathrm{F}}
\newcommand{\dB}{\mathrm{B}}
\newcommand{\Sine}{\mathrm{Sine}}
\newcommand{\Loop}{\mathrm{Loop}}
\newcommand{\wN}[1]{w_N^{(#1)}} 
\newcommand{\wo}{\wN{1}}  
\newcommand{\wt}{\wN{2}}  
\newcommand{\gap}{\mathrm{gap}}
\newcommand{\Anch}{\mathrm{Anch}}
\newcommand{\dBL}{\mathrm{d}_{\mathrm{BL}}}

\begin{document}
\title{Optimal local laws and central limit theorem for the circular Riesz gas}
\author{Jeanne Boursier}
\address{Department of Mathematics, Massachusetts Institute of Technology, Cambridge, MA 02139, USA}
\email{boursier@mit.edu}

\begin{abstract}
We study the circular Riesz gas, a one-dimensional system of particles with long-range power-law repulsion. We establish near-optimal rigidity estimates for the gaps at all scales. We then prove a quantitative central limit theorem for linear statistics, combining Stein's method with a mean-field transport. The main difficulty is to control the fluctuations of the local error terms left by the transport, which we do using the rigidity estimates, convexity and monotonicity. In particular, a comparison principle for the Helffer--Sjöstrand equation allows us to treat test functions with low regularity, including indicators of intervals.
\end{abstract}

\maketitle

\setcounter{tocdepth}{1}
\tableofcontents

\section{Introduction}

\subsection{The circular Riesz gas}
In this paper, we study the one-dimensional Riesz gas on the circle. We write $\dT:=\dR/\mathbb{Z}$. For a parameter $s\in (0,1)$, let us consider the Riesz $s$-kernel on $\dT$, defined by
\begin{equation}\label{eq:formula claim}
    g(x)=\lim_{n\to\infty}\left(\sum_{k=-n}^{n}\frac{1}{|k+x|^s}-\frac{2}{1-s}n^{1-s}\right)=\zeta(s,x)+\zeta(s,1-x),
\end{equation}
where the right-hand side is extended periodically from $x\in(0,1)$ and $\zeta(s,x)$ stands for the Hurwitz zeta function~\cite{berndt}. Note that $g$ is the fundamental solution of the fractional Laplace equation on the circle
\begin{equation}\label{eq:frac laplace}
\begin{cases}
    (-\Delta)^{(1-s)/2}g=c_s(\delta_0-1),\\
    \int g=0,
\end{cases}
\end{equation}
with $c_s$ given by
\begin{equation}\label{eq:defcs}
    c_s=\frac{\Gamma((1-s)/2)}{\Gamma(s/2)}2^{1-s}\sqrt{\pi}.
\end{equation}
We work on the set of ordered configurations
\[
D_N=\{X_N=(x_1,\ldots,x_N)\in \dT^N:\text{the points admit real lifts }x_1<x_2<\cdots<x_N<x_1+1\}.
\]
We use these cyclic lifts throughout and extend them by $x_{i+N}=x_i+1$.
On $D_N$, let us consider the energy
\begin{equation*}
    \Hc_N:X_N\in D_N\mapsto N^{-s}\sum_{i\neq j}g(x_i-x_j),
\end{equation*}
where $g$ is given by~\eqref{eq:formula claim}. The circular Riesz gas, at the inverse temperature $\beta>0$, is defined by the Gibbs measure
\begin{equation*}
    \dd\dP_{N,\beta}=\frac{1}{Z_{N,\beta}}\exp(-\beta \Hc_N(X_N) )\mathds{1}_{D_N}(X_N)\dd X_N,
\end{equation*}
where $Z_{N,\beta}$ is the normalization constant, called the \emph{partition function}, given by
\begin{equation*}
    Z_{N,\beta}=\int_{D_N}e^{-\beta \Hc_N(X_N)}\dd X_N.
\end{equation*}
Throughout the paper, $s$ is a fixed parameter in $(0,1)$.

The model belongs to the family of \emph{Riesz gases}, with kernel $|x|^{-s}$ on $\dR^d$. The logarithmic kernel, obtained as $s\downarrow0$, gives the $\beta$-ensembles of random matrix theory~\cite{MR148397}, and the case $s=d-2$ gives the Coulomb gas. For $0<s<d$, the interaction is long-range~\cite{lewin2022coulomb}. The classical theory of stable, tempered interactions~\cite{ruelle1974statistical} does not cover long-range Riesz gases.

\subsection{Fluctuations and the regularity of test functions}
As $N\to\infty$, the empirical measure $\mu_N:=\frac{1}{N}\sum_{i=1}^N \delta_{x_i}$ converges in probability to the uniform measure on $\dT$. We ask how much the linear statistics
\[
    \Fluct_N[\xi(\ell_N^{-1}\cdot )]:=\sum_{i=1}^N \xi(\ell_N^{-1}x_i)-N\ell_N \int_{\ell_N^{-1}\dT}\xi
\]
fluctuate, where $\ell_N\in(0,1]$ is a scale and $\xi$ is a test function on the circle $\ell_N^{-1}\dT=\dR/\ell_N^{-1}\mathbb Z$. The focus of this work is on test functions with as little regularity as possible, such as the indicator of an interval, which counts the particles in an arc.

For log-gases, CLTs for smooth test functions go back to Johansson~\cite{johansson1998fluctuations} and were extended in~\cite{borot2013asymptotic,shcherbina2013fluctuations,bekerman2018clt,lambert2019quantitative,lambert2021mesoscopic,peilen2022local}. These results require several derivatives of the test function, for instance $\mathcal C^4$ in~\cite{bekerman2018clt} and $\mathcal C^9$ with compact support in~\cite{lambert2019quantitative}. The limiting variance needs only half a derivative. Indeed, the logarithmic kernel is, up to a constant, the fundamental solution of $(-\Delta)^{1/2}$, and the variance of a linear statistic converges to a multiple of $\int\xi(-\Delta)^{1/2}\xi$, the squared $H^{1/2}$ seminorm of the test function~\cite{jiang2015moments}. Singular test functions are much harder to treat. Indicators of intervals lie just outside $H^{1/2}$, and interval counts have variance of order $\log N$. Their Gaussian fluctuations were long known only for ensembles with a determinantal structure or a random matrix model~\cite{costin1995gaussian,gustavsson2005gaussian,killip2008gaussian}. For general $\beta$ and one-cut analytic potentials, Bourgade, Mody and Pain~\cite{bourgade2022optimal} proved them as a consequence of an optimal local law valid down to the microscopic scale.

\subsection{A partial differential equation for the fluctuations}
The fluctuations of a statistic are tied to the response of the gas to a perturbation, as in the fluctuation--dissipation theorem of statistical physics. For a Gibbs measure, this relation expresses the variance of a statistic through the solution of an elliptic equation in dimension $N$, the Helffer--Sjöstrand equation~\cite{sjostrandpotential,sjostrandpot2,helffer1994correlation}. We study this equation with three types of tools.
\begin{itemize}[leftmargin=*]
\item \emph{Mean-field approximation.} For a linear statistic, the solution is close to a transport that moves each particle according to its position, as in the change of variables of Johansson~\cite{johansson1998fluctuations}. At the level of variances, this transport becomes an integration by parts, known as the loop equations or Schwinger--Dyson equations~\cite{bourgade2014universality}. It gives the leading order of the variance and leaves an error term, the loop equation term.
\item \emph{Brascamp--Lieb inequality and Gaussian concentration.} Gibbs measures with a convex energy satisfy the Brascamp--Lieb inequality~\cite{Brascamp2002,HELFFER1998571}, and those with a uniformly convex energy satisfy Gaussian concentration~\cite{bakryemery,ledoux1999concentration}. Our energy has a special form of convexity, since every pair of particles contributes to the curvature. We will see that, as in~\cite{bourgade2014universality}, this gives near-optimal concentration for the gaps. For a linear statistic, the Brascamp--Lieb inequality gives, up to the common rotation $x_i\mapsto x_i+\theta$ of all particles, which leaves the energy unchanged,
\[
\Var_{\dGi}\bigl[\Fluct_N[\xi]\bigr]\leq CN^s\int_{\dT}|\xi'|^2.
\]
This bound has the correct order for smooth test functions, but it is infinite for indicators of intervals.
\item \emph{Monotonicity.} The equation satisfies a maximum principle, as in classical correlation inequalities~\cite{cartier1974inegalites,helffer1994correlation}. It compares the fluctuations of a singular statistic with those of a linear function of the positions, which rigidity controls. This is what allows test functions as singular as indicators of intervals.
\end{itemize}

\subsection{Main results}

\subsubsection{Rigidity of gaps and local particle numbers}

Our first theorem gives a counterpart for Riesz gaps and interval counts of the log-gas rigidity estimate in~\cite[Theorem~3.1]{bourgade2014universality}.

\begin{theorem}[Near-optimal rigidity]\label{theorem:almost optimal rigidity}
For $\ve>0$, let $\delta(\ve)=\ve/(4(s+2))$. Then there exists \(\ve_{0}>0\) such that, for every \(0<\ve<\ve_{0}\), there exist constants $c>0$ and $C>0$ depending on $\beta$, $s$ and $\ve$ such that:
\begin{enumerate}
\item (Gap fluctuations) For every \(1\le i\le N\) and every $1\le k\le\lfloor N/2\rfloor$,
\begin{equation}\label{statement:gaps}
\dGi\bigl(|N(x_{i+k}-x_{i})-k|\ge k^{s/2+\ve}\bigr)\leq
Ce^{-ck^{\delta(\ve)}}.
\end{equation}
\item (Discrepancy in an interval) For every \(x\in\dT\) and every \(0<\ell_{N}\le1/4\),
\begin{equation}\label{statement:dis}
\dGi\left(\Bigl|\sum_{i=1}^{N}\mathds{1}_{(x-\ell_{N},\,x+\ell_{N})}(x_{i})-2N\ell_{N}\Bigr|
            \ge(N\ell_{N})^{s/2+\ve}\right)
\leq Ce^{-c(N\ell_{N})^{\delta(\ve)}}.
\end{equation}
\end{enumerate}
\end{theorem}

\subsubsection{Singular linear statistics: variance expansion}

Next, we prove a variance expansion for possibly singular test functions, at scales satisfying $N\ell_N\geq1$, under the following assumption:

\begin{assumption}[Admissible test function]\label{assumption:test func}
Let $\ell_N\in (0,1]$ and let \(\xi\in L^2(\ell_N^{-1}\dT)\). Let $\psi$ be the transport map of $\xi$, given by
\[
\psi'=-\frac{1}{2\beta c_s}(-\Delta)^{(1-s)/2}\xi\quad \text{and}\quad \int \psi=0.
\]

\begin{enumerate}[label=(A\arabic*)]

\item\label{item:A1} \emph{Global Hölder regularity.}
There exists $\ve\in(0,s)$ such that
\[
\xi-\ell_N\int_{\ell_N^{-1}\dT}\xi=\Xi'
\]
for some $\Xi\in\mathcal{C}^{1-s+\ve}(\ell_N^{-1}\dT)$, where the derivative is understood in the sense of distributions.
\item\label{item:A2} \emph{Piecewise smoothness.}
For some integer $p\geq1$ and distinct marked points $a_1,\ldots,a_p$ of $\ell_N^{-1}\dT$, the function $\psi$ is $\mathcal C^2$ on $\ell_N^{-1}\dT\setminus\{a_1,\ldots,a_p\}$.

\item\label{item:A3} \emph{Controlled algebraic singularities.}
There exist a constant $C>0$ and an exponent \(\alpha\in[s-1,s/2)\) such that
\begin{equation}\label{eqdef:singularity psi''}
|\psi''(x)|\le
C\left(\sum_{j=1}^{p}\frac{1}{|x-a_{j}|^{2-s+\alpha}}\mathds{1}_{|x|\leq 1}+\frac{1}{|x|^{3-s}}\mathds{1}_{|x|\geq 1} \right),
\qquad x\in \ell_N^{-1}\dT\setminus\{a_{1},\dots,a_{p}\}.
\end{equation}

\item\label{item:A4} \emph{Support.}
Either \(\ell_{N}\equiv1\) and $\xi$ does not depend on $N$,
or \(\ell_{N}\to0\) and, in this case, \(\xi\) is the periodization on $\ell_N^{-1}\dT$ of a fixed function in $L^2(\dR)$ supported in a closed subinterval of \((-1/2,1/2)\).
We call these two alternatives the macroscopic case and the mesoscopic case, respectively.

\end{enumerate}
The data $\ve$, the real numbers $a_j$, the constant $C$ in~\eqref{eqdef:singularity psi''} and a bound on the $\mathcal C^{1-s+\ve}$ norm of $\Xi$ do not depend on $N$. Constants said below to depend on $\xi$ may depend on these data, but not on $(\ell_N)$ or $N$.
\end{assumption}

\begin{remark}\label{remark:assumption xi}
For $\alpha>0$, the periodic function $x\mapsto|x|^{-\alpha}$ satisfies Assumption~\ref{assumption:test func} if and only if $\alpha<s/2$. Here $|x|$ is the circle distance to $0$, and the antipodal point may be included among the marked points in~\ref{item:A2}. The indicator of a proper arc also satisfies the assumption. Such an indicator belongs to $H^{(1-s)/2}(\dT)$ for $s>0$, but not to $H^{1/2}(\dT)$.
\end{remark}

Let $r>0$. The fractional Sobolev seminorm of a function $f\in L^2(\dT)$ with Fourier coefficients $\hat{f}(k)$ is given by
\begin{equation*}
    \Vert f\Vert_{H^{r}(\dT) }^2:=\sum_{k\in\mathbb{Z}}|k|^{2r}|\hat{f}|^2(k),
\end{equation*}
whenever the sum is finite.
Similarly, the fractional Sobolev seminorm of a function $f\in L^2(\dR)$, denoted $\Vert f\Vert_{H^{r}(\dR)}$, is defined by
\[
   \Vert f\Vert_{H^{r}(\dR) }^2:=\int_{\dR} |\eta|^{2r}|\hat{f}|^2(\eta)\dd \eta,
\]
where $\hat f(\eta)=\int_{\dR}f(x)e^{-2\pi i\eta x}\dd x$ for $f\in L^1(\dR)\cap L^2(\dR)$, extended to $L^2(\dR)$ by continuity. For a function $\xi$ on $\ell_N^{-1}\dT$, we write $\xi\in H^r(\ell_N^{-1}\dT)$ when $\xi(\ell_N^{-1}\cdot)\in H^r(\dT)$.

Let $(\ell_N)$ be a sequence taking values in $(0,1]$. For every function $\xi\in H^{(1-s)/2}(\ell_N^{-1}\dT)$, define
\begin{equation}\label{def:sigmaN}
    \sigma_{\ell_N}^2(\xi):=\frac{(2\pi)^{1-s}}{2 \beta c_s } \ell_N^{-s}\Vert\xi(\ell_N^{-1}\cdot)\Vert_{H^{(1-s)/2}(\dT)}^2.
\end{equation}
If $\xi\in H^{(1-s)/2}(\dT)$ is supported on $(-1/2,1/2)$, let $\xi_0:\dR\to\dR$ be given for all $x\in \dR$ by
\begin{equation}\label{eq:defxi0}
    \xi_0(x)=\begin{cases}
        \xi(x) & \text{if $|x|\leq 1/2$,}\\
        0 & \text{if $|x|>1/2$,}
    \end{cases}
\end{equation}
and
\[
    \sigma_0^2(\xi):=\frac{(2\pi)^{1-s}}{2\beta c_s}\Vert \xi_0\Vert_{H^{(1-s)/2}(\dR)}^2.
\]

\begin{theorem}[Variance of singular linear statistics]\label{theorem:quantitative variance}
Let $\xi\in L^2(\ell_N^{-1}\dT)$ and $(\ell_N)$ satisfy Assumption~\ref{assumption:test func}, with $N\ell_N\geq1$. Let $\alpha\in [s-1,s/2)$ be such that~\eqref{eqdef:singularity psi''} holds.

For all $\ve>0$, there exists a constant $C>0$ depending on $\beta$, $s$, $\ve$ and $\xi$ such that
\begin{equation}\label{eq:variance expansion theorem}
   \Bigl| \Var_{\dGi}[ \Fluct_N[\xi(\ell_N^{-1}\cdot)]]-(N\ell_N)^s\sigma_{\ell_N}^2(\xi)\Bigr|\leq  C(N\ell_N)^{\max(2s-1,2\alpha)+\ve }.
\end{equation}
If $\xi$ is supported in $(-1/2,1/2)$, then there exists a constant $C>0$ depending on $\beta$, $s$ and $\xi$ such that
\begin{equation*}
 | \sigma_{\ell_N}^2(\xi)- \sigma_{0}^2(\xi)|\leq C\ell_N^{2-s}.
\end{equation*}
\end{theorem}

Since $\alpha<s/2$, we have $\xi\in H^{(1-s)/2}(\ell_N^{-1}\dT)$.

Theorem~\ref{theorem:quantitative variance} implies that the count in a fixed proper arc has variance of order $N^s$, the same order as every nonconstant smooth linear statistic at a fixed macroscopic scale. The order $N^s$ lies between those of two classical point processes:
\begin{itemize}[leftmargin=*]
\item the circular unitary ensemble, where nonconstant smooth statistics have variance of order $1$ and counts in a fixed proper interval have variance of order $\log N$~\cite[Equation~(1.2)]{jiang2015moments}, and
\item a Poisson point process of intensity $N$ on $\dT$, whose count in an arc has variance $N$ times the arc length. As $s\uparrow1$, the order $N^s$ approaches this linear growth.
\end{itemize}
More generally, for arcs of length $L_N\to0$ with $NL_N\to\infty$, the number variance is asymptotic to a positive constant times $(NL_N)^s$. This growth is sublinear in $NL_N$, as for hyperuniform point processes~\cite{Torquato_2016}.

\subsubsection{Central limit theorem}

Let $\dBL$ be the bounded-Lipschitz distance between probability measures on $\dR$, defined for all $\mu$ and $\nu$ by
\begin{equation*}
\dBL(\mu,\nu):=\sup_{f\in \mathcal{C}^1(\dR)}\left\{ \int f\dd (\mu-\nu):\Vert f\Vert_{L^\infty(\dR)}\leq 1,\Vert f'\Vert_{L^\infty(\dR)}\leq 1\right\}.
\end{equation*}
We establish the following result:

\begin{theorem}[CLT for singular linear statistics]\label{theorem:CLT}
Let $\xi\in L^2(\ell_N^{-1}\dT)$ and $(\ell_N)$ satisfy Assumption~\ref{assumption:test func}. Let $\alpha\in [s-1,s/2)$ be such that~\eqref{eqdef:singularity psi''} holds. Assume in addition that \(\ell_{N}\gg N^{-1}\). Let
\(Z_{N}\sim\mc N\bigl(0,\sigma_{\ell_{N}}^{2}(\xi)\bigr)\). Let $\ell:=\lim_{N\to\infty}\ell_N\in\{0,1\}$.
\begin{enumerate}
\item The random variables
\((N\ell_{N})^{-s/2}\Fluct_{N}\bigl[\xi(\ell_{N}^{-1}\cdot)\bigr]\)
converge in distribution to \(\mc N\bigl(0,\sigma_\ell^2(\xi)\bigr)\).
\item Let $r=(2-s)(s-2\alpha)/(2(4+s))$. For every \(\ve>0\), there exists a constant $C>0$ depending on $\beta$, $s$, $\ve$ and $\xi$ such that, whenever $N\ell_N\geq1$,
\begin{align*}
\dBL\bigl(\mathrm{Law}((N\ell_{N})^{-s/2}\Fluct_{N}[\xi(\ell_N^{-1}\cdot)] ),\mathrm{Law}(Z_{N})\bigr)
&\leq C(N\ell_{N})^{\ve+\max((s-1)/2,-r)}.
\end{align*}
\end{enumerate}
\end{theorem}

Theorem~\ref{theorem:CLT} and the Cramér--Wold criterion give joint convergence of the pairings of the normalized density fluctuation
\begin{equation*}
\mathcal X_N:=N^{-s/2}\left(\sum_{i=1}^N\delta_{x_i}-N\dd x\right)
\end{equation*}
against finitely many smooth real test functions. The limiting centered Gaussian random distribution $\mathcal X$ has covariance
\begin{equation*}
\Cov[\mathcal X(f),\mathcal X(h)]
=\frac{1}{2\beta c_s}\int_{\dT}f\,(-\Delta)^{(1-s)/2}h.
\end{equation*}
Consequently, the potential fluctuation
\begin{equation*}
g*\mathcal X_N
=N^{-s/2}\sum_{i=1}^N g(x_i-\cdot)
\end{equation*}
converges in the same sense to a fractional Gaussian field with covariance operator $c_s(-\Delta)^{-(1-s)/2}/(2\beta)$ on mean-zero functions.

\subsubsection{Examples}

By Remark~\ref{remark:assumption xi}, Theorem~\ref{theorem:CLT} applies to interval counts.
To state the variance explicitly, define the continuous periodic function
\begin{equation}\label{def:periodic zeta}
V_s(t):=\zeta(-s,t)+\zeta(-s,1-t)-2\zeta(-s),\quad 0<t<1,\qquad V_s(0):=0,
\end{equation}
where $\zeta(z)=\zeta(z,1)$. We extend $V_s$ with period $1$. Lemma~\ref{lemma:explicit formulas} shows that $V_s(t)>0$ for $0<t<1$ and $V_s(t)=t^s+O(t^2)$ as $t\downarrow0$.

\begin{corollary}[CLT for the number of points in an interval]\label{corollary:fluctuations gaps disc}
Let $a\in (0,1/2)$ and let $(\ell_N)$ be a sequence taking values in $(0,1]$ such that $\ell_N\gg 1/N$. We have
\[
\frac{1}{\sqrt{N^s V_s(2a\ell_N)}}  \left(\sum_{i=1}^N \mathds{1}_{(-a\ell_N,a\ell_N) }(x_i)-2aN\ell_N\right)\underset{\mathrm{Law}}{\Longrightarrow}\mathcal{N}(0,\sigma^2),
\]
 where
\begin{equation}\label{def:sig}
    \sigma^2:=\frac{\cot(\pi s/2)}{2\beta\pi s}.
\end{equation}

Let $(k_N)_{N\geq2}$ be a sequence taking values in $\{1,\ldots,\lfloor N/2\rfloor\}$ such that $k_N\to \infty$ as $N\to \infty$. We have
\[
\frac{1}{\sqrt{N^s V_s(k_N/N)}}(N(x_{1+k_N}-x_1)-k_N)\underset{\mathrm{Law}}{\Longrightarrow}\mathcal{N}(0,\sigma^2),
\]
where $\sigma^2$ is as in~\eqref{def:sig}.
\end{corollary}

Corollary~\ref{corollary:fluctuations gaps disc} is a Riesz analogue of bulk eigenvalue and counting fluctuations for log-gases. These fluctuations are established for the GUE in~\cite{gustavsson2005gaussian}, for one-cut regular unitary ensembles in~\cite{claeys2021much} and for general $\beta$ under regular one-cut assumptions in~\cite{bourgade2022optimal}.

\begin{corollary}[CLT for power-law observables]\label{corollary:power law}
Let $\alpha\in(0,s/2)$ and let $\xi$ be the periodic function given by $\xi(x)=\zeta(\alpha,x)+\zeta(\alpha,1-x)$ for $0<x<1$. Then,
\begin{equation*}
    N^{-s/2}\left(\sum_{i=1}^N\xi(x_i)-N\int\xi\right)\underset{\mathrm{Law}}{\Longrightarrow}\mathcal{N}\Bigl(0,(2\pi)^{-1+2\alpha-s}\frac{c_\alpha^2}{\beta c_s}\zeta(1+s-2\alpha)\Bigr),
\end{equation*}
where $c_\alpha$ and $c_s$ are given by~\eqref{eq:defcs}.
\end{corollary}

\begin{remark}
The borderline observable $x\mapsto|x|^{-s/2}$ does not belong to $H^{(1-s)/2}(\dT)$. It plays the role of the logarithmic singularity for log-gases, whose fluctuations require a $\sqrt{\log N}$ normalization~\cite[Theorem~1.8]{bourgade2022optimal}.
\end{remark}

\subsubsection{Formalization in Lean}
Theorems~\ref{theorem:almost optimal rigidity}, \ref{theorem:quantitative variance} and~\ref{theorem:CLT}, together with Corollaries~\ref{corollary:fluctuations gaps disc} and~\ref{corollary:power law}, have been formalized in the Lean proof assistant, using the Mathlib library. The formalization is available at \url{https://github.com/jboursier/circular-riesz-gas-lean}.

\subsection{Related results}
\subsubsection{Rigidity of \texorpdfstring{$\beta$}{beta}-ensembles}
Our proof of rigidity follows the method introduced by Bourgade, Erdős and Yau~\cite{bourgade2012bulk,bourgade2014universality,bourgade2014edge}. They proved rigidity of $\beta$-ensembles at the optimal polynomial scale, by adding a convex penalty on large gaps and comparing each particle with the average of a block of neighbors. On the line, rigidity is formulated through $x_i-\gamma_i$, where $\gamma_i$ is the corresponding quantile of the equilibrium measure, and estimating the expected positions is one of the main difficulties of~\cite{bourgade2014universality}. On the circle, cyclic symmetry gives the expected gaps exactly, and this step disappears. Theorem~\ref{theorem:almost optimal rigidity} is the Riesz analogue of~\cite[Theorem~3.1]{bourgade2014universality}. Bourgade, Mody and Pain~\cite{bourgade2022optimal} later obtained the optimal rigidity scale $(\log N)/N$ in the bulk, from an optimal local law derived from the loop equation hierarchy.

\subsubsection{Transport method}
The change of variables of Johansson~\cite{johansson1998fluctuations} is the starting point of many CLTs for log-gases and Coulomb gases. It is developed in~\cite{bekerman2018clt} for $\beta$-ensembles with general potentials, in~\cite{leble2018clt} for the $\Sine_\beta$ process, and in~\cite{leble2018fluctuations,leble2020local,serfaty2020gaussian} for Coulomb gases in dimension $2$ and higher. The subsequent work~\cite{peilen2025local} extends it to Riesz gases, and gives CLTs in dimension one for small $s>0$ at small mesoscopic scales. In these works, the transport is applied to Laplace transforms, and the error term it leaves is estimated in the sense of large deviations. We compute variances instead, and we control the fluctuations of this error term using the Helffer--Sjöstrand representation and two consequences of convexity, the Brascamp--Lieb inequality and a comparison principle.

\subsubsection{Stein's method}
Lambert, Ledoux and Webb~\cite{lambert2019quantitative} prove a quantitative CLT for $\beta$-ensembles with Stein's method. Their linear statistics are approximate eigenfunctions of the generator of Dyson Brownian motion, and rigidity controls the error terms. Our proof of Theorem~\ref{theorem:CLT} is closely related. For circular $\beta$-ensembles with inverse temperature proportional to $N^{-1}$, Hardy and Lambert~\cite{hardy2021clt} prove a quantitative CLT.

\subsubsection{Helffer--Sjöstrand equation and decay of correlations}
The Helffer--Sjöstrand representation also expresses covariances, so decay of correlations reduces to decay of the solution of the equation. For the circular Riesz gas, the subsequent work~\cite{boursier2022decay} analyzes the equation directly and proves decay of gap correlations. As a consequence, the gas, with positions rescaled by $N$, converges to a translation-invariant point process~\cite[Theorem~3]{boursier2022decay}. In the logarithmic case, this limit is the $\Sine_\beta$ process~\cite{valko2009continuum,killip2009eigenvalue}. For local $\beta$-ensembles, quantitative decay of correlations for gap observables is proved in~\cite[Theorem~8.1]{erdos2014gap}.

\subsection{Outline of the main proofs}

\begin{itemize}[leftmargin=*]
\item \emph{Rigidity.} The proof of Theorem~\ref{theorem:almost optimal rigidity} follows that of~\cite[Theorem~3.1]{bourgade2014universality}. The law of the gaps is log-concave, which gives a first local law. By Lemma~\ref{lemma:initial estimate}, for every $\delta>0$, the gap $N(x_{i+k}-x_i)$ exceeds $k^{1+\delta}$ only on an event of probability at most $Ce^{-ck^\delta}$. To reach the scale $k^{s/2}$, we compare each particle with the average position of a block of neighbors, through a sequence of block sizes up to $k$. At each step, a convex penalty on large gaps changes the measure only on an event of small probability. After this penalty, every pair of particles in the block contributes to the curvature, and the Gaussian concentration of Lemma~\ref{lemma:div free} bounds the difference of consecutive block averages. The same concentration bounds the difference between the block averages at $x_i$ and $x_{i+k}$, which proves the gap estimate. Conditioning on the gaps and averaging over the common rotation then gives the estimate on interval counts.
\end{itemize}

For Theorems~\ref{theorem:quantitative variance} and~\ref{theorem:CLT}, we set $\ell_N=1$ for simplicity and consider $\Fluct_N[\xi]$ for $\xi\in L^2(\dT)$ satisfying Assumption~\ref{assumption:test func}. Let $\alpha\in[s-1,s/2)$ be such that~\eqref{eqdef:singularity psi''} holds. We also assume $N\geq2$.

\begin{itemize}[leftmargin=*]
\item \emph{Mollification.} Let $\rho\in\mathcal C_c^\infty((-1/2,1/2))$ be even and nonnegative, with $\int_{\dR}\rho=1$, and let $\rho_h:=h^{-1}\rho(h^{-1}\cdot)$. We consider the smooth test function $\xi_\reg:=\xi*\rho_{1/N}$. We first show that, for all $\ve>0$, there exists $C>0$ such that
\[
    \Var_{\dGi}[\Fluct_N[\xi-\xi_\reg]]\leq CN^\ve(N^{2\alpha}+N^{3s-2}).
\]

\item \emph{The Helffer--Sjöstrand equation.} Let \( F : D_N \to \mathbb{R} \) be sufficiently smooth. The fluctuations of \( F \) are related to a partial differential equation via the representation
\begin{equation}\label{eq:helffer sjos repr}
    \Var_{\dGi}[F] = \mathbb{E}_{\dGi}[ \nabla F \cdot \nabla \phi],
\end{equation}
where \( \phi\in H^1(\dGi) \) solves the Poisson equation
\begin{equation}\label{eq:Lmuphi}
 \mathcal{L} \phi = F - \mathbb{E}_{\dGi}[F],
\end{equation}
where \( \mathcal{L} \) denotes the Langevin operator
\[
    \mathcal{L} := \beta \nabla \mathcal{H}_N \cdot \nabla - \Delta.
\]
Note that~\eqref{eq:helffer sjos repr} follows directly by integration by parts, once it is known that~\eqref{eq:Lmuphi} admits a solution. Differentiating~\eqref{eq:Lmuphi} yields the Helffer--Sjöstrand equation
\begin{equation}\label{eq:intro helffer equation}
A_1\nabla\phi=\nabla F\quad\text{in }D_N,
\end{equation}
where $A_1$ is formally given by
\[
A_1:=\beta\nabla^2\mathcal H_N+\mathcal L\otimes I_N.
\]
Section~\ref{section:helffer} and Appendix~\ref{section:existence} give the weak formulation and treat the collision boundary.

The comparison principle is formulated after anchoring the first particle. Let $\dGi^*$ be the law of $(x_i-x_1)_{2\leq i\leq N}$. The off-diagonal entries of the anchored Hessian are nonpositive, so $(A_1^{\dGi^*})^{-1}$ preserves coordinatewise nonnegativity. In particular, if $F,G\in H^1(\dGi^*)$ satisfy $|\partial_iF|\leq\partial_iG$ for every anchored coordinate, then
\[
\Var_{\dGi^*}[F]\leq\Var_{\dGi^*}[G].
\]
The same principle compares the quantities $\dE_{\dGi^*}[v\cdot(A_1^{\dGi^*})^{-1}v]$ for vector fields $v$ that are not gradients. We use it to control the singular terms left by the mean-field transport.

\item \emph{Mean-field transport.} When $F$ is a linear statistic, the solution $\nabla \phi$ of~\eqref{eq:intro helffer equation} is close to a transport of the form
\begin{equation*}
\Psi_N:X_N\in D_N\mapsto \frac{1}{N^{1-s}}(\psi(x_1),\ldots,\psi(x_N)),
\end{equation*}
where the map $\psi:\dT\to\dR$ is chosen below.

Let $F:=\Fluct_N[\xi_\reg]$, with $\xi_\reg=\xi*\rho_{1/N}$ as above. In terms of the empirical measure $\mu_N$,
\begin{equation*}
    \nabla \mc{H}_N\cdot \Psi_N=N\iint_{x\neq y}g'(x-y)(\psi(x)-\psi(y))\dd \mu_N(x)\dd\mu_N(y).
\end{equation*}
Let $\fluct_N:=N(\mu_N-\dd x)$ be the fluctuation of $\mu_N$ around the Lebesgue measure on $\dT$. After the substitution $\mu_N=\dd x+N^{-1}\fluct_N$, the constant term vanishes and
\[
    \nabla \mc{H}_N\cdot \Psi_N=2\int (-g'*\psi)(x)\fluct_N(\dd x)+\frac{1}{N^{1-s}}\Loop_1[\psi],
\]
where
\[
   \Loop_1[\psi]:=\iint_{x\neq y}(\psi(x)-\psi(y))N^{-s}g'(x-y)\fluct_N(\dd x)\fluct_N(\dd y).
\]

Let $\psi$ be given by
\begin{equation*}
    \psi'=-\frac{1}{2\beta c_s}(-\Delta)^{(1-s)/2}\xi_\reg\quad \text{and}\quad \int \psi=0.
\end{equation*}
Then, up to the sub-leading terms $\dive \Psi_N$ and $\Loop_1[\psi]$,
\begin{equation*}
    \beta\nabla \mc{H}_N\cdot \Psi_N-\dive\Psi_N \simeq F.
\end{equation*}
The central task of the paper is to show that, for all $\ve>0$, there exists a constant $C>0$ such that
\begin{equation}\label{eq:bound variance A intro}
   \Var_{\dGi}[\Loop_1[\psi]] \leq CN^{\ve+\max(1,2 (\alpha+1-s))}.
\end{equation}

\item \emph{Splitting the loop equation term.} We localize the transport near each singularity and treat the pieces separately. By rotation invariance, we translate each piece so that its singularity is at the origin, and still denote it by $\psi$. Below, $\ve>0$ is small and $\kappa>0$ is a constant independent of $\ve$. Lemma~\ref{lemma:gradient splitting} gives the splitting $\nabla\Loop_1[\psi]=\dW+\dV$. The field $\dW$ compares each pair interaction with its value at the lattice spacing, and is small when the gaps are rigid. The field $\dV$ carries the singularity of $\psi''$.

We then anchor the first particle and add a convex penalty on large gaps. Let $\dGiQ^*$ be the law of the anchored positions $y_i=x_i-x_1$ under the penalized measure~\eqref{eq:const}, and let $A_1=A_1^{\dGiQ^*}$. Let $\eta^*$ be the cutoff~\eqref{def:eta cutoff} to an event on which all gaps are close to their typical values, and let $\dW^*$ and $\dV^*$ be the fields $\dW$ and $\dV$ in anchored coordinates. The anchoring estimate of Lemma~\ref{lemma:auxiliary Loop} and the Helffer--Sjöstrand representation give
\begin{multline*}
\Var_{\dGi}[\Loop_1[\psi]]\leq CN^{\kappa\ve}\underbrace{\dE_{\dGiQ^*}[\eta^*\dV^*\cdot A_1^{-1}(\eta^*\dV^*)]}_{(\RomanNumeralCaps{1})}\\
+CN^{\kappa\ve}\underbrace{\dE_{\dGiQ^*}[\eta^*\dW^*\cdot A_1^{-1}(\eta^*\dW^*)]}_{(\RomanNumeralCaps{2})}+CN^{\kappa\ve}\bigl(N+N^{2(1-s+\alpha)}\bigr).
\end{multline*}

\item \emph{Control of the gap field by the Brascamp--Lieb inequality.} The field $\dW$ is a combination of increments of the gaps, in the sense that, for every $v\in\dR^N$,
\begin{equation*}
    \dW \cdot v = -\sum_{i=1}^N \dW^\gap_i N(v_{i+1} - v_i),
\end{equation*}
where $v_{N+1}=v_1$. On the support of $\eta^*$, Lemma~\ref{lemma:gradient splitting} gives $|\dW^\gap|^2\leq CN^{\kappa\ve}(N+N^{2(1-s+\alpha)})$. The penalty makes $\dGiQ^*$ uniformly log-concave in gap coordinates, with a constant that loses only a small power of $N$. The Brascamp--Lieb inequality therefore gives $(\RomanNumeralCaps{2})\leq CN^{\kappa\ve}(N+N^{2(1-s+\alpha)})$, after increasing $\kappa$.

\item \emph{Control of the singular field by the comparison principle.} For the singular field $\dV$, we use the anchored comparison principle. For $\xi=\mathds{1}_{(a,b)}$, the regularized transport satisfies, near each endpoint $a_j\in\{a,b\}$,
\begin{equation*}
|\psi''(x)|\leq C\max\bigl(|x-a_j|,N^{-1}\bigr)^{s-2}.
\end{equation*}
The pointwise bound~\eqref{eq:V1sup} on $\dV$ then gives nonnegative weights $b_i$ such that, on the support of $\eta^*$,
\begin{equation*}
|\dV_i^*|\leq CN^{\kappa\ve}\partial_iG_N^*,\qquad\text{where}\quad G_N^*(y)=\sum_{i=2}^N b_i\bigl(Ny_i-(i-1)\bigr).
\end{equation*}
Lemma~\ref{lemma:energy comparisons} and Theorem~\ref{theorem:almost optimal rigidity} give, after increasing $\kappa$,
\begin{equation*}
(\RomanNumeralCaps{1})\leq CN^{\kappa\ve}\Var_{\dGiQ^*}[\eta^*G_N^*]+CN^{-10}\leq CN^{\kappa\ve}\bigl(N^s+N^{2(1-s+\alpha)}\bigr).
\end{equation*}
Together with the bound on $(\RomanNumeralCaps{2})$, this yields~\eqref{eq:bound variance A intro} after choosing a smaller $\ve$. To recover the variance of the original statistic, we keep it in the transport identity and mollify only the transport. This completes the proof of Theorem~\ref{theorem:quantitative variance}.

\item \emph{Central limit theorem.} The proof of Theorem~\ref{theorem:CLT} uses Stein's method, with the larger regularization scale $\delta_N=N^{-(2-s)/(4+s)}$. Let $\xi_\reg=\xi*\rho_{\delta_N}$ and $G_N=N^{-s/2}\Fluct_N[\xi_\reg]$. By Stein's method, it suffices to prove that, for every $u\in\mathcal C^1(\dR)$ such that $u$ and $u'$ are bounded,
\begin{equation}\label{eq:intro approximate stein}
    \dE_{\dGi}[u(G_N)G_N]=\sigma_{1}^2(\xi_\reg)\dE_{\dGi}[u'(G_N)]+\Error_N,
\end{equation}
with $\sigma_1^2(\xi_\reg)$ as in~\eqref{def:sigmaN} and $\Error_N$ small relative to $\Vert u\Vert_{L^\infty(\dR)}+\Vert u'\Vert_{L^\infty(\dR)}$. Let $\psi$ and $\Psi_N$ be defined as above from this $\xi_\reg$, and let $\Phi$ be the periodic potential with $\nabla\Phi=N^{-s/2}\Psi_N$. Integration by parts with $\Phi$ in place of the solution of~\eqref{eq:Lmuphi} gives~\eqref{eq:intro approximate stein} with two error terms. The first compares $\frac1N\sum_{i=1}^N\xi_\reg'(x_i)\psi(x_i)$ with $\sigma_1^2(\xi_\reg)$, and the a priori variance estimate controls it. The second is $-N^{s/2-1}$ times the covariance of $u(G_N)$ with $\beta\Loop_1[\psi]-\Fluct_N[\psi']$, and it is controlled by~\eqref{eq:bound variance A intro} and the a priori variance estimate. Finally, the variance expansion compares $G_N$ with $N^{-s/2}\Fluct_N[\xi]$, and balancing the two errors that depend on $\delta_N$ gives the rate in Theorem~\ref{theorem:CLT}.
\end{itemize}

\subsection{Structure of the paper}

\begin{itemize}
  \item Section~\ref{section:preliminaries} presents useful preliminaries on the fractional Laplacian on the circle.
  \item Section~\ref{section:helffer} establishes the well-posedness of the Helffer--Sjöstrand equation and derives consequences of convexity and monotonicity.
  \item Section~\ref{section:rigidity} completes the proof of the near-optimal rigidity result (Theorem~\ref{theorem:almost optimal rigidity}).
  \item Section~\ref{section:optimal scaling} proves the variance expansion (Theorem~\ref{theorem:quantitative variance}).
  \item Section~\ref{section:CLTs} contains the proof of the central limit theorem (Theorem~\ref{theorem:CLT}) and its corollaries.
  \item Appendix~\ref{section:existence} proves the existence and uniqueness results stated in Section~\ref{section:helffer}.
  \item Appendix~\ref{section:auxiliary} contains proofs of various discrete convolution and energy estimates.
\end{itemize}

\subsection{Notation}
\begin{itemize}
\item Positions are represented by the cyclic lifts $x_1<x_2<\cdots<x_N<x_1+1$, and labels are extended to $\mathbb Z$ by $x_{j+N}=x_j+1$. Differences of positions such as $x_{i+k}-x_i$ are computed with these lifts.
\item We denote by $d:\{1,\ldots,N\}^2\to\mathbb{N}$ the cyclic index distance defined for each $1\leq i,j\leq N$ by
\[
d(i,j)=\min(|j-i|,N-|j-i|).
\]
\item On the circle $L\dT=\dR/L\mathbb Z$ of length $L>0$, we use arclength measure, the Fourier coefficients $\hat f(k)=L^{-1}\int_{L\dT}f(x)e^{-2i\pi kx/L}\dd x$ and the convolution $(f*v)(x)=\int_{L\dT}f(x-y)v(y)\dd y$. The operator $(-\Delta)^{\gamma}$ multiplies $e^{2i\pi kx/L}$ by $|2\pi k/L|^{2\gamma}$. Convolutions with singular kernels, such as $g'$, are symmetric principal values. In bounds on test functions, transport maps and their weights, such as~\eqref{eqdef:singularity psi''}, $|x|$ and $|x-a|$ denote distances on the circle for $x,a\in L\dT$. Elsewhere, $|\cdot|$ is the absolute value of real numbers, computed with real lifts.
\item Given a scale $\ell_N\in(0,1]$, we write $K:=N\ell_N$, the expected number of particles in an arc of length $\ell_N$.
\item For $\alpha\in[s-1,s/2)$, we write $\bar\alpha:=1-s+\alpha\in[0,1-s/2)$.
\item For a measurable map $T$ and a measure $\mu$, the push-forward $T\#\mu$ is given by $(T\#\mu)(A)=\mu(T^{-1}(A))$.
\item For a map $u:\Omega\to\dR^m$, where $\Omega$ is an open set of $\dR^n$, we let $Du$ be the matrix of partial derivatives of $u$. We also write $\nabla^2 f$ for the Hessian of a real-valued function $f$.
\item For $r\in (0,1)$, we let $\mathcal{C}^{r}(\dT)$ be the space of $r$-Hölder continuous functions from $\dT$ to $\dR$, and similarly on the circle $\ell_N^{-1}\dT$.
\end{itemize}

\subsection{Acknowledgments}
The author thanks Sylvia Serfaty, Thomas Leblé, Djalil Chafaï, Gaultier Lambert, David García-Zelada and Ahmed Bou-Rabee for many helpful comments, and the anonymous referee for invaluable feedback on an earlier version of this manuscript.

The author thanks Scott Armstrong for providing the skill files used for the Lean formalization. During the last round of revisions, the author used GPT-6.1 Sol, GPT-6 Astra and Claude Opus 5.5 to formalize the paper in Lean, to make corrections and to propose simplifications of the proofs.

This work was supported by a grant from the ``Fondation CFM pour la Recherche'' and by the ERC Advanced Grant LDRAM (ERC-2019-ADG 884584).
\vspace{1cm}

\section{Preliminaries}\label{section:preliminaries}

\subsection{The fundamental solution of the fractional Laplacian on the circle}
We begin by justifying that the fundamental solution of the fractional Laplace equation on the circle is given by~\eqref{eq:formula claim}. Roughly, \eqref{eq:formula claim} corresponds to the periodic summation of $x\mapsto|x|^{-s}$, which is the fundamental solution of $(-\Delta)^{(1-s)/2}$ on the real line.

For $x>0$ and complex $z$ with $\mathrm{Re}(z)>1$, let
\begin{equation}\label{eq:Hurwitz}
    \zeta(z,x):=\sum_{n=0}^{\infty}\frac{1}{(n+x)^z}.
\end{equation}
For each $x>0$, one can uniquely extend $\zeta(\cdot,x)$ to a meromorphic function on the full complex plane with a unique pole at $z=1$, which is simple with a residue equal to $1$. This function is called the Hurwitz zeta function~\cite{berndt}.

\begin{lemma}\label{lemma:fundamental}\
\begin{enumerate}
\item Let $g$ be the solution of~\eqref{eq:frac laplace}. Let $\zeta(s,x)$ be the Hurwitz zeta function~\eqref{eq:Hurwitz}. Then, for $x\in(0,1)$,
\begin{equation}\label{eq:expression g}
 g(x)=\zeta(s,x)+\zeta(s,1-x)=\lim_{n\to\infty}\left(\sum_{k=-n}^n \frac{1}{|k+x|^s}-\frac{2}{1-s}n^{1-s}\right) \, .
\end{equation}

\item Let $\gamma\in (0,1/2)$. For every $f\in \mathcal{C}^\infty(\dT)$, we have
\begin{equation}\label{eq:frac s}
(-\Delta)^{\gamma}f(x)
  =
  C_{\gamma}
  \sum_{n\in \mathbb{Z}} \int_0^1
  \frac{f(x)-f(y)}{|x-y-n|^{1+2\gamma}}
  \dd y,
  \quad \text{where}\quad
  C_{\gamma}
  :=
  \frac{2^{2\gamma}
        \Gamma((1+2\gamma)/2)}
       {\sqrt{\pi}\bigl|\Gamma(-\gamma)\bigr|}.
\end{equation}
\end{enumerate}
\end{lemma}

The above lemma is standard, but for completeness, we include a proof, following roughly~\cite{roncal2016fractional}.

\begin{proof}\
\paragraph{\emph{Step 1: semigroup representation}}
We first derive the semigroup representation for $(-\Delta)^{-(1-s)/2}$. Let $\gamma\in (0,1)$. Recall
\begin{equation}\label{eq:lambda-s}
    \lambda^{-\gamma}=\frac{1}{\Gamma(\gamma)}\int_0^{\infty}e^{-\lambda t}\frac{\dd t}{t^{1-\gamma}},\quad \text{for all $\lambda>0$} \, .
\end{equation}
The inverse fractional Laplacian $(-\Delta)^{-\gamma}$ is defined by the following Fourier multiplier: for all $f\in\mathcal{C}^{\infty}(\dT)$ such that $\int f=0$, letting $h=(-\Delta)^{-\gamma}f$, we have
\begin{equation*}
    \hat{h}(k)=(2\pi)^{-2\gamma}|k|^{-2\gamma}\hat{f}(k),\quad \text{for all $k\in \mathbb{Z}\setminus\{0\}$} \quad \text{and}\quad \hat{h}(0)=0.
\end{equation*}
Let $(W_t)_{t\geq 0}$ be the heat kernel of the Laplacian on $\dT$. Recall that the Fourier coefficients of $W_t$ are given by
\begin{equation*}
    \widehat{W_t}(k)=e^{-4t\pi^2|k|^2},\quad \text{for all $k\in \mathbb{Z}$, $t\geq 0$} \, .
\end{equation*}

We now derive an explicit expression for $g$. Applying the Poisson summation formula to the Gaussian density $y\mapsto(4\pi t)^{-1/2}e^{-y^2/(4t)}$ gives that for all $x\in \dT$ and $t>0$,
\begin{equation}\label{eq:heat}
    W_t(x)=\frac{1}{\sqrt{4\pi t}}\sum_{n\in \mathbb{Z}}e^{-(x-n)^2/(4t)}.
\end{equation}
Taking $\gamma=(1-s)/2$, the right-hand side of the display below converges absolutely in $L^1(\dT)$, since $\Vert W_t-1\Vert_{L^1(\dT)}\leq2$ and $W_t-1$ decays exponentially as $t\to\infty$. By~\eqref{eq:lambda-s} and~\eqref{eq:frac laplace}, its Fourier coefficients are those of $g$. Since it is continuous in $x\in(0,1)$, Fourier uniqueness gives, for $x\in(0,1)$,
\begin{equation*}
g(x)=c_s(-\Delta)^{-(1-s)/2}(\delta_0-1)\Bigr|_x=\frac{c_s}{\Gamma((1-s)/2)}\int_0^{\infty}(W_t(x)-1)\frac{\dd t}{t^{(1+s)/2}}.
\end{equation*}
Then, inserting~\eqref{eq:heat}, we get
\begin{equation}\label{eq:invert me}
    g(x)=\frac{c_s}{\Gamma((1-s)/2)\sqrt{4\pi}}\int_0^{\infty}\sum_{k\in \mathbb{Z}}\Bigl(e^{-(x-k)^2/(4t) }-\int_{\dT}e^{-(y-k)^2/(4t) }\dd y\Bigr) \frac{\dd t}{t^{1+s/2}}.
\end{equation}

\paragraph{\emph{Step 2: absolute integrability of $\sum_k u_k^x$}}
Fix $x\in(0,1)$ and define
\begin{equation*}
    u_k^x(t):=\frac{1}{t^{1+s/2}}\left(e^{-(x-k)^2/(4t)}-\int_0^1e^{-(y-k)^2/(4t)}\dd y\right),\qquad t>0,\quad k\in\mathbb Z.
\end{equation*}
We prove that $\sum_k\int_0^\infty|u_k^x(t)|\dd t<\infty$. For $0<t\leq1$, Gaussian decay gives
\[
\sum_{k\in\mathbb Z}e^{-(x-k)^2/(4t)}\leq C_xe^{-c_x/t},
\qquad
\int_0^1\sum_{k\in\mathbb Z}e^{-(y-k)^2/(4t)}\dd y=\sqrt{4\pi t}.
\]
Since $s<1$, it follows that
\begin{equation}\label{eq:01}
\sum_{k\in\mathbb Z}\int_0^1|u_k^x(t)|\dd t
\leq\int_0^1\frac{C_xe^{-c_x/t}+\sqrt{4\pi t}}{t^{1+s/2}}\dd t<\infty.
\end{equation}
For $t\geq1$, the mean value theorem gives
\[
|u_k^x(t)|\leq C\frac{|k|+1}{t^{2+s/2}}e^{-ck^2/t}.
\]
For $p=1+s/4$, the bound $e^{-u}\leq Cu^{-p}$ for $u>0$ and the convergence of $\sum_{k\neq0}|k|^{1-2p}$ yield $\sum_k(|k|+1)e^{-ck^2/t}\leq Ct^p$ for $t\geq1$. Therefore,
\begin{equation}\label{eq:1inf}
\sum_{k\in\mathbb Z}\int_1^\infty|u_k^x(t)|\dd t
\leq C\int_1^\infty t^{-1-s/4}\dd t<\infty.
\end{equation}
Together, \eqref{eq:01} and~\eqref{eq:1inf} prove the required absolute integrability.

\paragraph{\emph{Step 3: evaluation of the integral}}
Combining~\eqref{eq:01} and~\eqref{eq:1inf} and invoking Fubini's theorem, we may interchange the summation and the integral in~\eqref{eq:invert me}, which gives
\begin{equation*}
\begin{split}
    g(x)&=\frac{c_s}{\Gamma((1-s)/2)\sqrt{4\pi}}\sum_{k\in \mathbb{Z}}\int_0^{\infty}u_k^x(t)\dd t\\
    &=\frac{c_s}{\Gamma((1-s)/2)\sqrt{4\pi}}\sum_{k\in \mathbb{Z}}\left(\int_0^{\infty}e^{-(x-k)^2/(4t)}\frac{\dd t}{t^{1+s/2} } -\int_0^{\infty}\int_{\dT}e^{-(y-k)^2/(4t)}\frac{\dd t}{t^{1+s/2} }\dd y\right)\\
    &=\frac{c_s\Gamma(s/2)2^s}{\Gamma((1-s)/2)2\sqrt{\pi} }\sum_{k\in \mathbb{Z}}\left(\frac{1}{|x-k|^{s}}-\int_{\dT}\frac{\dd y}{|y-k|^{s}}\right)\\
    &=\lim_{n\to \infty}\sum_{k=-n}^n\left(\frac{1}{|x-k|^{s}}-\int_{\dT}\frac{\dd y}{|y-k|^{s}}\right)\\
    &=\lim_{n\to\infty}\left(\sum_{k=-n}^n\frac{1}{|x+k|^s}-\frac{2}{1-s}n^{1-s} \right),
\end{split}
\end{equation*}
where we have used the change of variables $t\in(0,\infty)\mapsto 1/t$ and~\eqref{eq:lambda-s} in the third equality, the expression for $c_s$ in~\eqref{eq:defcs} in the fourth equality, and the fact that
\begin{equation*}
    \int_{-n}^{n+1}\frac{1}{|y|^s}\dd y=\frac{n^{1-s}+(n+1)^{1-s}}{1-s}=\frac{2}{1-s}n^{1-s}+o_n(1)
\end{equation*}
in the last equality.

\paragraph{\emph{Step 4: rewriting $g$}}
By the Euler--Maclaurin formula, for all $x>0$ and complex $z$ with $\mathrm{Re}(z)>0$ and $z\neq1$,
\begin{equation}\label{eq:euler mac}
   \lim_{n\to\infty}\left(\sum_{k=0}^n \frac{1}{(k+x)^z}-\frac{1}{1-z}n^{1-z}\right)=-\frac{1}{1-z}x^{1-z}+\frac{1}{2x^z}-z\int_0^{\infty} \frac{t-\lfloor t\rfloor -1/2}{(t+x)^{1+z}}\dd t.
\end{equation}
For $\mathrm{Re}(z)>1$, the left-hand side of~\eqref{eq:euler mac} equals $\zeta(z,x)$. The integral on the right is absolutely convergent and analytic for $\mathrm{Re}(z)>0$, so analytic continuation identifies the right-hand side with $\zeta(z,x)$ throughout this half-plane, away from $z=1$. Taking $z=s$, we conclude that, for $x\in(0,1)$,
\begin{equation*}
    g(x)=\zeta(s,1-x)+\zeta(s,x).
\end{equation*}

\paragraph{\emph{Step 5: proof of~\eqref{eq:frac s}}}
Fix $\gamma\in(0,1/2)$ and $f\in \mathcal{C}^\infty(\dT)$. For all $\lambda>0$, we have
\begin{equation*}
    \lambda^\gamma=\frac{\gamma}{\Gamma(1-\gamma)}\int_0^\infty (1-e^{-\lambda t})\frac{\dd t}{t^{1+\gamma}}.
\end{equation*}
Hence,
\begin{equation*}
    \widehat{(-\Delta)^{\gamma}f}(k)=(4\pi^2|k|^2)^{\gamma}\hat{f}(k)=\frac{\gamma}{\Gamma(1-\gamma)}\int_0^\infty (1-e^{-4\pi^2|k|^2t})\hat{f}(k)\frac{\dd t}{t^{1+\gamma}}.
\end{equation*}
We deduce that
\begin{equation*}
(-\Delta)^{\gamma}f= \frac{\gamma}{\Gamma(1-\gamma)}\int_0^\infty (f-f*W_t)\frac{\dd t}{t^{1+\gamma}}.
\end{equation*}
Recalling~\eqref{eq:heat}, we have
\begin{equation*}
    f*W_t(x)=\frac{1}{\sqrt{4\pi t}}\sum_{n\in\mathbb Z}\int_0^1f(y)e^{-(x-y-n)^2/(4t)}\dd y.
\end{equation*}
For $x-y-n\neq0$, the change of variables $u=(x-y-n)^2/(4t)$ gives
\[
\int_0^\infty e^{-(x-y-n)^2/(4t)}\frac{\dd t}{t^{3/2+\gamma}}
=\frac{2^{1+2\gamma}\Gamma(\gamma+1/2)}{|x-y-n|^{1+2\gamma}}.
\]
Since the periodic extension of $f$ is Lipschitz and $\gamma<1/2$,
\[
\sum_{n\in\mathbb Z}\int_0^1\frac{|f(x)-f(y)|}{|x-y-n|^{1+2\gamma}}\dd y
\leq C\int_{\mathbb R}\frac{\min(1,|r|)}{|r|^{1+2\gamma}}\dd r<\infty.
\]
Tonelli's theorem applied to the absolute values therefore justifies interchanging the integrals and the sum. Inserting the preceding integral identity into the semigroup formula gives~\eqref{eq:frac s}.
\end{proof}

\subsection{Inverse Riesz transform}

\begin{lemma}Let $\xi\in \mathcal{C}^{\infty}(\dT)$. Let $\psi\in \mathcal{C}^\infty(\dT)$ be given by
\begin{equation}\label{eq:transport}
    \psi'=-\frac{1}{2\beta c_s}(-\Delta)^{(1-s)/2}\xi \quad \text{and}\quad \int \psi=0.
\end{equation}
\begin{enumerate}
\item We have
\begin{equation}\label{eq:an expression for psi}
  \psi(x)=-\frac{1}{4\pi\beta\tan(\pi s/2)}\sum_{k\in \mathbb{Z}}\int_{\dT}\frac{\xi(y)-\int \xi}{|x-y-k|^{1-s}}\sgn(x-y-k)\dd y,\quad \text{for all $x\in \dT$}.
\end{equation}
\item We have
\begin{equation}\label{eq:mf var}
\frac{(2\pi)^{1-s}}{2 \beta c_s}\Vert\xi\Vert_{H^{(1-s)/2}(\dT)}^2 =-\beta\iint g''(x-y)(\psi(x)-\psi(y))^2\dd x\dd y+2\int \xi'(x)\psi(x)\dd x 
  =\int \xi'(x)\psi(x)\dd x.
\end{equation}
\item Assume that $\xi$ is supported on $(-1/2,1/2)$, let $\ell_N\in (0,1]$, and let $\xi_0$ be the extension of $\xi$ by zero to $\dR$, as in~\eqref{eq:defxi0}. Here $\xi(\ell_N^{-1}\cdot)$ denotes the periodization of $\xi_0(\ell_N^{-1}\cdot)$. Then, there exists a constant $C>0$ depending on $s$ and the distance of $\operatorname{supp}\xi$ from the endpoints $\pm1/2$ such that
\begin{equation}\label{eq:scaling frac}
    \Bigr|\Vert\xi(\ell_N^{-1}\cdot)\Vert_{H^{(1-s)/2}(\dT)}^2-\ell_N^s\Vert\xi_0\Vert_{H^{(1-s)/2}(\dR)}^2\Bigr|\leq C\ell_N^2\Vert\xi_0\Vert_{L^2(\dR)}^2.
\end{equation}
\end{enumerate}
\end{lemma}

\begin{proof}
We first derive the integral formula for $\psi$. Let $\Xi\in \mathcal{C}^{\infty}(\dT)$ satisfy $\xi-\int\xi=\Xi'$. Since differentiation commutes with $(-\Delta)^{(1-s)/2}$, we have
\begin{equation*}
    \psi=-\frac{1}{2\beta c_s}(-\Delta)^{(1-s)/2}\Xi.
\end{equation*}
By~\eqref{eq:frac s} with $\gamma=(1-s)/2$,
\begin{equation}\label{eq:psi'}
    \psi(x)=-\frac{C_{(1-s)/2}}{2\beta c_s}\sum_{k\in\mathbb Z}\int_{\dT}\frac{\Xi(x)-\Xi(y)}{|x-y-k|^{2-s}}\dd y,
\end{equation}
where
\begin{equation*}
    C_{(1-s)/2}=\frac{2^{1-s}\Gamma(1-s/2)}{|\Gamma(-(1-s)/2)|\sqrt{\pi}}=\frac{2^{-s}(1-s)\Gamma(1-s/2)}{\Gamma((1+s)/2)\sqrt{\pi}}.
\end{equation*}
Using Euler's reflection formula, we compute
\begin{equation*}
    \frac{C_{(1-s)/2}}{c_s}=\frac{1-s}{2\pi}\frac{\Gamma(1-s/2)\Gamma(s/2)}{\Gamma((1-s)/2)\Gamma((1+s)/2)}=\frac{1-s}{2\pi\tan(\pi s/2)}.
\end{equation*}
We integrate by parts in~\eqref{eq:psi'}, using
\[
\partial_y\bigl(\sgn(x-y-k)|x-y-k|^{s-1}\bigr)=(1-s)|x-y-k|^{s-2}
\]
away from $x-y-k=0$. The boundary terms at this point are $O(\ve^s)$ after removing a neighborhood of radius $\ve$. The outer boundary terms in the sum over $-M\leq k\leq M$ are $O(M^{s-1})$. They therefore vanish, which gives~\eqref{eq:an expression for psi}. Its series is absolutely convergent after integration in $y$, since $\xi-\int\xi$ has mean zero and the $k$th integral is $O(|k|^{s-2})$.

We next prove~\eqref{eq:mf var}. By~\eqref{eq:frac laplace} and~\eqref{eq:transport},
\[
2\beta g*\psi'=-\xi+\int\xi.
\]
Integration by parts gives
\begin{align*}
\iint g''(x-y)(\psi(x)-\psi(y))^2\dd x\dd y
    &=2\iint g(x-y)\psi'(x)\psi'(y)\dd x\dd y\\
    &=\int\psi'(2g*\psi')=-\frac1\beta\int\psi'\xi=\frac1\beta\int\psi\xi'.
\end{align*}
To justify the integrations by parts, we first exclude $|x-y|<\ve$ on the circle. Since $\psi$ is smooth, the boundary terms are $O(\ve^{1-s})$ and vanish as $\ve\to0$. Moreover, Parseval's identity gives
\begin{equation*}
    \int_{\dT}\psi'\xi=\sum_{k\in\mathbb Z}\widehat{\psi'}(k)\overline{\hat\xi(k)}=-\frac{(2\pi)^{1-s}}{2\beta c_s}\sum_{k\in\mathbb Z}|k|^{1-s}|\hat\xi(k)|^2.
\end{equation*}
Combining these identities proves~\eqref{eq:mf var}.

It remains to prove~\eqref{eq:scaling frac}. Let $f(x)=\xi_0(\ell_N^{-1}x)$ on $\dR$. The integral $\vartheta_s=\int_{\dR}(1-\cos(2\pi v))|v|^{s-2}\dd v$ is finite and positive, and changing variables gives $|\eta|^{1-s}=\vartheta_s^{-1}\int_{\dR}(1-\cos(2\pi\eta u))|u|^{s-2}\dd u$ for every $\eta\in\dR$. With Parseval's identity on $\dT$, Plancherel's theorem on $\dR$ and Tonelli's theorem, we obtain
\begin{align*}
\Vert h\Vert_{H^{(1-s)/2}(\dT)}^2&=\frac{1}{2\vartheta_s}\int_{\dR}\int_{\dT}\frac{|h(x+u)-h(x)|^2}{|u|^{2-s}}\dd x\dd u,&&h\in L^2(\dT),\\
\Vert h\Vert_{H^{(1-s)/2}(\dR)}^2&=\frac{1}{2\vartheta_s}\iint_{\dR^2}\frac{|h(x+u)-h(x)|^2}{|u|^{2-s}}\dd x\dd u,&&h\in L^2(\dR),
\end{align*}
as identities in $[0,\infty]$. We apply them to the periodization $\xi(\ell_N^{-1}\cdot)$ of $f$ and to $f$. Since the integer translates of $f$ have disjoint supports, only products of distinct translates remain in the difference after expanding the squares. By the bound below, these products are absolutely integrable, so both seminorms are finite together, and we obtain
\begin{equation*}
\Vert\xi(\ell_N^{-1}\cdot)\Vert_{H^{(1-s)/2}(\dT)}^2-\Vert f\Vert_{H^{(1-s)/2}(\dR)}^2 
=-\frac{1}{\vartheta_s}\sum_{k\in\mathbb Z\setminus\{0\}}\iint_{\dR^2}\frac{f(x)f(y)}{|x-y+k|^{2-s}}\dd x\dd y.
\end{equation*}
The sum of the kernels on the right is bounded on $\operatorname{supp}f\times\operatorname{supp}f$, uniformly in $\ell_N\in(0,1]$, by a constant depending on $s$ and the distance of $\operatorname{supp}\xi$ from $\pm1/2$. Hence its absolute value is bounded by
\[
C\Vert f\Vert_{L^1(\dR)}^2=C\ell_N^2\Vert\xi_0\Vert_{L^1(\dR)}^2\leq C\ell_N^2\Vert\xi_0\Vert_{L^2(\dR)}^2.
\]
Finally, a change of variables in the Fourier transform gives
\[
\Vert f\Vert_{H^{(1-s)/2}(\dR)}^2=\ell_N^s\Vert\xi_0\Vert_{H^{(1-s)/2}(\dR)}^2,
\]
which proves~\eqref{eq:scaling frac}. The same proof gives the estimate for every $\xi\in H^{(1-s)/2}(\dT)$ whose support is contained in a closed subinterval of $(-1/2,1/2)$.
\end{proof}

\begin{lemma}[Explicit formulas]\label{lemma:explicit formulas}
Let $V_s$ be the periodic function defined in~\eqref{def:periodic zeta}. Then $V_s$ is positive on $(0,1)$ and satisfies $V_s(t)=t^s+O(t^2)$ as $t\downarrow0$.
\begin{enumerate}
\item Let $0<a<1/2$ and let $\xi=\mathds{1}_{(-a,a)}-2a$. Let $\psi$ be given by the right-hand side of~\eqref{eq:an expression for psi}. Then
\begin{equation}\label{eq:formulapsi}
    \psi(x)=-\frac{1}{4\pi\beta s\tan(\pi s/2)}\bigl(V_s(x+a)-V_s(x-a)\bigr),\quad \text{for all $x\in\dT$}.
\end{equation}
Moreover,
\begin{equation}\label{eq:formulaindic}
   \frac{(2\pi)^{1-s}}{2\beta c_s}\Vert\xi\Vert_{H^{(1-s)/2}(\dT)}^2=\frac{\cot(\pi s/2)}{2\pi\beta s}V_s(2a).
\end{equation}
\item Let $\alpha\in(0,s/2)$ and let $\xi$ be the periodic function given by $\xi(x)=\zeta(\alpha,x)+\zeta(\alpha,1-x)$ for $0<x<1$. Then
\begin{equation*}
    \frac{(2\pi)^{1-s}}{2\beta c_s}\Vert\xi\Vert_{H^{(1-s)/2}(\dT)}^2=(2\pi)^{-1+2\alpha-s}\frac{c_\alpha^2}{\beta c_s}\zeta(1+s-2\alpha),
\end{equation*}
where $c_\alpha$ is given by~\eqref{eq:defcs} with $s$ replaced by $\alpha$.
\end{enumerate}
\end{lemma}

\begin{proof}
We first record the properties of $V_s$ used in the indicator formula. The recurrence and differentiation identities for the Hurwitz zeta function follow from its defining series and analytic continuation. In particular,
\[
V_s(t)=t^s+\zeta(-s,1+t)+\zeta(-s,1-t)-2\zeta(-s)=t^s+O(t^2).
\]
Together with $V_s(1-t)=V_s(t)$, this proves continuity at the endpoints. Moreover,
\[
V_s''(t)=-s(1-s)\bigl(\zeta(2-s,t)+\zeta(2-s,1-t)\bigr)<0,\quad 0<t<1.
\]
Since $V_s(0)=V_s(1)=0$, we conclude that $V_s(t)>0$ for $0<t<1$. By Hurwitz's formula~\cite[Equation~(25.11.9)]{olver2010nist} and~\eqref{eq:defcs}, the function $V_s$ has the absolutely convergent Fourier series
\[
V_s(t)=\frac{2s\tan(\pi s/2)}{c_s(2\pi)^s}\sum_{k=1}^{\infty}\frac{1-\cos(2\pi kt)}{k^{1+s}}.
\]
Since $\hat\xi(k)=\sin(2\pi ka)/(\pi k)$ for $k\neq0$, the right-hand side of~\eqref{eq:formulapsi} has mean zero and $k$th Fourier coefficient $-(2\pi|k|)^{1-s}\hat\xi(k)/(2\beta c_s\cdot2\pi ik)$ for $k\neq0$. Since the right-hand side of~\eqref{eq:an expression for psi} is a convolution with $\xi$ that agrees with~\eqref{eq:transport} for smooth $\xi$, these are also the Fourier coefficients of $\psi$. Both sides of~\eqref{eq:formulapsi} are continuous, so Fourier uniqueness proves~\eqref{eq:formulapsi}.

To apply~\eqref{eq:mf var} to the indicator, we convolve $\xi$ and $\psi$ with a smooth periodic approximate identity. By these Fourier coefficients, the regularized functions converge to $\xi$ in $H^{(1-s)/2}(\dT)$. Formula~\eqref{eq:formulapsi} makes $\psi$ continuous, so its regularizations converge uniformly. Since $\xi'=\delta_{-a}-\delta_a$, passage to the limit in~\eqref{eq:mf var} gives
\begin{equation*}
    \frac{(2\pi)^{1-s}}{2\beta c_s}\Vert\xi\Vert_{H^{(1-s)/2}(\dT)}^2=\psi(-a)-\psi(a)=\frac{\cot(\pi s/2)}{2\pi\beta s}V_s(2a),
\end{equation*}
which proves~\eqref{eq:formulaindic}.

For the second assertion, Lemma~\ref{lemma:fundamental} gives
\begin{equation*}
    (-\Delta)^{(1-\alpha)/2}\xi=c_\alpha(\delta_0-1),\qquad \int\xi=0.
\end{equation*}
Therefore,
\begin{equation*}
    \hat\xi(k)=\frac{c_\alpha}{(2\pi|k|)^{1-\alpha}},\quad k\neq0,\qquad \hat\xi(0)=0.
\end{equation*}
It follows that
\begin{equation*}
\Vert\xi\Vert_{H^{(1-s)/2}(\dT)}^2
=\frac{c_\alpha^2}{(2\pi)^{2(1-\alpha)}}\sum_{k\in\mathbb Z\setminus\{0\}}\frac{1}{|k|^{1+s-2\alpha}}
=\frac{2c_\alpha^2}{(2\pi)^{2(1-\alpha)}}\zeta(1+s-2\alpha).
\end{equation*}
The series converges because $2\alpha<s$, and multiplication by $(2\pi)^{1-s}/(2\beta c_s)$ proves the second assertion.

These formulas also verify the regularity conditions in Assumption~\ref{assumption:test func} for the two examples. The nonzero Fourier coefficients of the centered indicator and of the inverse-power observable are $O(|k|^{-1})$ and $O(|k|^{\alpha-1})$. Dividing them by $2\pi ik$ gives a primitive $\Xi$ with $\sum_{k\neq0}|k|^r|\hat\Xi(k)|<\infty$, hence $\Xi\in\mathcal C^r$, for $r=1-s/2$ and $r=1-(s+\alpha)/2$, respectively. Thus~\ref{item:A1} holds with $\ve=s/2$ and $\ve=(s-\alpha)/2$. For the centered indicator,~\eqref{eq:formulapsi} gives singularities of order $|x\mp a|^{s-2}$ for $\psi''$. On an enlarged circle, differentiating the kernel in~\eqref{eq:an expression for psi} twice gives the required bound $C|x|^{s-3}$ away from the support. There, the primitive $\Xi$ of the centered indicator is bounded and $1$-Lipschitz, uniformly in the length of the circle, so its $\mathcal C^{1-s/2}$ norm is bounded uniformly. The transport derivative of the inverse-power observable is a constant multiple of the Riesz kernel with parameter $\bar\alpha$, so that $|\psi''(x)|\leq C|x|^{-1-\bar\alpha}$ near $0$.
\end{proof}

\section{The Helffer--Sjöstrand equation}\label{section:helffer}
In this section, we review some basic properties of the Helffer--Sjöstrand equation. We first state some existence and uniqueness results, then adapt a known comparison principle to the circular setting, and finally record various concentration estimates.

\subsection{Basic properties}

Let $\mu$ be a probability measure on $D_N$ of the form
\[
\dd\mu(X_N)\mathrel{\propto} e^{-H(X_N)}\mathds{1}_{D_N}(X_N)\dd X_N.
\]
We make the following assumption on $H$.

\begin{assumption}\label{assumption:gibbs measure}
We assume $N\geq2$. The Hamiltonian $H:D_N\to\dR$ is of the form
\[
H(X_N)=\sum_{i\neq j}\chi_{ij}(|x_i-x_j|)=2\sum_{i<j}\chi_{ij}(x_j-x_i),
\]
where $\chi_{ij}=\chi_{ji}$ and, for every $i\neq j$,
\[
\chi_{ij}\in\mathcal C^\infty(0,1),\qquad
\chi_{ij}''\geq c>0,\qquad
\lim_{r\downarrow0}\chi_{ij}(r)=\lim_{r\uparrow1}\chi_{ij}(r)=+\infty.
\]
Moreover, there exists $C_N<\infty$ such that
\begin{equation}\label{eq:boundary growth}
\frac{|\nabla H(X_N)|^2}{4N}-\frac{\Delta H(X_N)}2
\geq\frac{1}{\operatorname{dist}(X_N,\partial D_N)^2}-C_N.
\end{equation}
\end{assumption}

The Riesz measure $\dGi$ and the convex perturbations used below satisfy this assumption, as verified in Appendix~\ref{section:existence}. Condition~\eqref{eq:boundary growth} justifies the operator domains at the collision boundary. The common rotation is uniform and independent of the gaps. A cyclic change of the first label may replace a pair potential by $r\mapsto\chi_{ij}(1-r)$, which preserves the assumption.

We seek to rewrite the variance of a smooth enough function under $\mu$ in a tractable way. On smooth functions supported away from the collision boundary, define the Langevin operator
\[
\mc L^\mu=\nabla H\cdot\nabla-\Delta.
\]
The functions are periodic in the common rotation. For such $F,G$, integration by parts gives
\begin{equation}\label{eq:IPPLmu}
\dE_\mu[(\mc L^\mu F)G]=\dE_\mu[\nabla F\cdot\nabla G].
\end{equation}
Let $H^1(\mu)$ be the completion of smooth functions on $\overline{D_N}$, periodic in the common rotation, for the norm
\[
\Vert F\Vert_{H^1(\mu)}^2=\dE_\mu[F^2+|\nabla F|^2].
\]
We use the Neumann realization of $\mc L^\mu$ associated with this quadratic form. Define also
\[
\mathfrak q_\mu(u,h)=\dE_\mu\left[\sum_{i=1}^N\nabla u_i\cdot\nabla h_i+u\cdot\nabla^2Hh\right].
\]
Let $\mathcal V_\mu$ be the completion of smooth vector fields supported away from the collision boundary for the norm $(\dE_\mu[|u|^2]+\mathfrak q_\mu(u,u))^{1/2}$. The operator $A_1^\mu$ is associated with this closed form and is formally given by
\[
A_1^\mu=\nabla^2H+\mc L^\mu\otimes I_N.
\]
We interpret $A_1^\mu u=v$ in the weak sense
\begin{equation}\label{eq:pose me}
\mathfrak q_\mu(u,h)=\dE_\mu[v\cdot h],\qquad h\in\mathcal V_\mu.
\end{equation}

Define the map
\[
\Gap_N:X_N\in D_N\mapsto N(x_2-x_1,x_3-x_2,\ldots,x_N-x_{N-1},1+x_1-x_N)\in\dR^N.
\]

\begin{proposition}[Existence and representation]\label{proposition:existence HS equation}
Let $\mu$ satisfy Assumption~\ref{assumption:gibbs measure} and let $F\in H^1(\mu)$. There exists a unique, modulo an additive constant, $\phi\in H^1(\mu)$ solving $\mc L^\mu\phi=F-\dE_\mu[F]$. Its gradient belongs to $\mathcal V_\mu$ and satisfies
\[
A_1^\mu\nabla\phi=\nabla F,\qquad \dE_\mu\left[\sum_{i=1}^N\partial_i\phi\right]=0.
\]
Moreover, $\nabla\phi$ is the unique minimizer of
\begin{equation}\label{eq:variatt}
u\mapsto\mathfrak q_\mu(u,u)-2\dE_\mu[\nabla F\cdot u]
\end{equation}
over $u\in\mathcal V_\mu$ such that $\dE_\mu[\sum_iu_i]=0$. Finally,
\begin{equation}\label{eq:representation variance}
\Var_\mu[F]=\dE_\mu[\nabla\phi\cdot\nabla F],
\end{equation}
and, for every $G\in H^1(\mu)$,
\[
\Cov_\mu[F,G]=\dE_\mu[\nabla\phi\cdot\nabla G].
\]
If $F$ is a function of the gaps, then $\phi$ is also a function of the gaps and $\sum_i\partial_i\phi=0$ almost everywhere.
\end{proposition}

The identity~\eqref{eq:representation variance} is called the Helffer--Sjöstrand formula. The proof of Proposition~\ref{proposition:existence HS equation} is postponed to Appendix~\ref{section:existence}.

For non-gradient vector fields, the common rotation gives a one-dimensional kernel. We remove it by a mean-zero normalization.

\begin{proposition}[Well-posedness for non-gradient vector fields]\label{prop:well posed non grad}
Let $\mu$ satisfy Assumption~\ref{assumption:gibbs measure}. Let $v\in L^2(\mu;\dR^N)$ satisfy $\dE_\mu[\sum_i v_i]=0$. There exists a unique $u\in\mathcal V_\mu$ such that
\begin{equation}\label{eq:Helf non gradd}
A_1^\mu u=v,\qquad \dE_\mu\left[\sum_{i=1}^Nu_i\right]=0.
\end{equation}
The solution is the unique minimizer of
\[
u\mapsto\mathfrak q_\mu(u,u)-2\dE_\mu[v\cdot u]
\]
over this normalized subspace. If each $v_i$ is a function of the gaps and $\sum_i v_i=0$ almost everywhere, then the same properties hold for $u$.
\end{proposition}

The proof is postponed to Appendix~\ref{section:existence}. We denote the solution of~\eqref{eq:Helf non gradd} by $(A_1^\mu)^{-1}v$.

\begin{lemma}
Let $\mu$ satisfy Assumption~\ref{assumption:gibbs measure}, and let $v,w\in L^2(\mu;\dR^N)$ satisfy $\dE_\mu[\sum_i v_i]=\dE_\mu[\sum_i w_i]=0$. Then
\begin{equation}\label{eq:subadd sum 0}
\dE_\mu[(v+w)\cdot(A_1^\mu)^{-1}(v+w)]\leq
2\Bigl(\dE_\mu[v\cdot(A_1^\mu)^{-1}v]+\dE_\mu[w\cdot(A_1^\mu)^{-1}w]\Bigr).
\end{equation}
\end{lemma}

\begin{proof}
Let $u=(A_1^\mu)^{-1}(v-w)$. By~\eqref{eq:pose me},
\[
\dE_\mu[(v-w)\cdot(A_1^\mu)^{-1}(v-w)]=\mathfrak q_\mu(u,u)\geq0.
\]
Expanding this inequality and using the symmetry of the form gives~\eqref{eq:subadd sum 0}.
\end{proof}

\subsection{Monotonicity and consequences}

We now prove a comparison principle after fixing the common rotation. Let $\mu^*:=\Anch_N\#\mu$, where
\begin{equation}\label{def:Anch}
\Anch_N:X_N\in D_N\mapsto(x_2-x_1,\ldots,x_N-x_1)\in\dR^{N-1}.
\end{equation}
The measure $\mu^*$ has density proportional to $e^{-H^*}$ on the anchored simplex $D_N^*:=\Anch_N(D_N)$ of points $y\in\dR^{N-1}$ with $0<y_2<\cdots<y_N<1$, where $H^*(y)=H(0,y)$. We define $H^1(\mu^*)$, $\mathcal V_{\mu^*}$, $\mathfrak q_{\mu^*}$ and $A_1^{\mu^*}$ by the same construction as above, with $D_N^*$ and $H^*$ in place of $D_N$ and $H$ and coordinates indexed by $2\leq i\leq N$. The following maximum principle is similar to that derived in~\cite{cartier1974inegalites,helffer1994correlation}.

\begin{lemma}[Monotonicity]\label{lemma:comparison pri}
Let $\mu$ satisfy Assumption~\ref{assumption:gibbs measure}, and let $\mu^*=\Anch_N\#\mu$ be its anchored law, with $\Anch_N$ as in~\eqref{def:Anch}. Let $v\in L^2(\mu^*;\dR^{N-1})$ and let $u\in\mathcal V_{\mu^*}$ solve
\begin{equation}\label{eq:A1til}
A_1^{\mu^*}u=v.
\end{equation}
If $v_i\geq0$ almost everywhere for every $i\in\{2,\ldots,N\}$, then
\[
u_i\geq0\quad\text{almost everywhere},\qquad i\in\{2,\ldots,N\}.
\]
\end{lemma}

\begin{remark}\label{remark:strict conv}
For $u=(u_2,\ldots,u_N)$, set $u_1=0$. Then
\[
u\cdot\nabla^2H^*u=2\sum_{i<j}\chi_{ij}''(x_j-x_i)(u_j-u_i)^2\geq2c\sum_{i=2}^N u_i^2.
\]
Thus $\mathfrak q_{\mu^*}$ is coercive, and the variational argument gives existence and uniqueness in~\eqref{eq:A1til}. The scalar and gradient representations of Proposition~\ref{proposition:existence HS equation} also hold for $\mu^*$, as justified in Appendix~\ref{section:existence}.
\end{remark}

\begin{proof}[Proof of Lemma~\ref{lemma:comparison pri}]
Let $u^+$ and $u^-$ be the coordinatewise positive and negative parts of $u$. They belong to $\mathcal V_{\mu^*}$. Indeed, the maps $r\mapsto r^+$ and $r\mapsto r^-$ fix zero and are contractions for both the Dirichlet energy and each term $\chi_{ij}''(u_j-u_i)^2$, with $u_1=0$. Approximation by compactly supported fields and weak lower semicontinuity give the assertion.

For distinct $i,j\in\{2,\ldots,N\}$, we have
\[
\partial_{ij}H^*=-2\chi_{ij}''(|x_j-x_i|)\leq0.
\]
Since $u_i^+u_i^-=0$, it follows that
\[
u^-\cdot\nabla^2H^*u^+
=\sum_{i\neq j}(\partial_{ij}H^*)u_i^-u_j^+\leq0.
\]
Testing~\eqref{eq:A1til} with $u^-$ and using $\nabla u_i\cdot\nabla u_i^-=-|\nabla u_i^-|^2$, we obtain
\[
0\leq\dE_{\mu^*}[v\cdot u^-]
=\mathfrak q_{\mu^*}(u,u^-)
\leq-\mathfrak q_{\mu^*}(u^-,u^-).
\]
Remark~\ref{remark:strict conv} implies $u^-=0$, which proves the lemma.
\end{proof}

\begin{lemma}[Comparison of variances and energies]\label{lemma:energy comparisons}
Let $\mu$ satisfy Assumption~\ref{assumption:gibbs measure}, and let $\mu^*=\Anch_N\#\mu$ be its anchored law, with $\Anch_N$ as in~\eqref{def:Anch}. Then $A_1^{\mu^*}$ is invertible on $L^2(\mu^*;\dR^{N-1})$. Moreover, the following two comparisons are true.
\begin{enumerate}
\item Let $F,G\in H^1(\mu^*)$ satisfy
\[
|\partial_iF|\leq\partial_iG\quad\text{almost everywhere},\qquad i\in\{2,\ldots,N\}.
\]
Then
\[
\Var_{\mu^*}[F]\leq\Var_{\mu^*}[G].
\]
\item Let $v,w\in L^2(\mu^*;\dR^{N-1})$ satisfy
\[
|v_i|\leq w_i\quad\text{almost everywhere},\qquad i\in\{2,\ldots,N\}.
\]
Then
\[
\dE_{\mu^*}[v\cdot(A_1^{\mu^*})^{-1}v]\leq
\dE_{\mu^*}[w\cdot(A_1^{\mu^*})^{-1}w].
\]
\end{enumerate}
\end{lemma}

\begin{proof}
Invertibility follows from Remark~\ref{remark:strict conv}. The inverse is self-adjoint on $L^2(\mu^*;\dR^{N-1})$, by symmetry of the form. By the covariance representation and Lemma~\ref{lemma:comparison pri},
\[
\Var_{\mu^*}[G]-\Var_{\mu^*}[F]
=\Cov_{\mu^*}[G+F,G-F]
=\dE_{\mu^*}[\nabla(G+F)\cdot(A_1^{\mu^*})^{-1}\nabla(G-F)]\geq0,
\]
which proves the first comparison. For the second one, self-adjointness gives
\[
\dE_{\mu^*}[w\cdot(A_1^{\mu^*})^{-1}w]-\dE_{\mu^*}[v\cdot(A_1^{\mu^*})^{-1}v]
=\dE_{\mu^*}[(w+v)\cdot(A_1^{\mu^*})^{-1}(w-v)].
\]
Both $w+v$ and $w-v$ are nonnegative coordinatewise, so Lemma~\ref{lemma:comparison pri} makes the right-hand side nonnegative.
\end{proof}

We shall also use sub-additivity~\eqref{eq:subadd sum 0} for $A_1^{\mu^*}$. Moreover, dropping the Dirichlet term in its variational representation gives
\begin{equation}\label{eq:vector BL}
\dE_{\mu^*}[v\cdot(A_1^{\mu^*})^{-1}v]
=\sup_{u\in\mathcal V_{\mu^*}}\{2\dE_{\mu^*}[v\cdot u]-\mathfrak q_{\mu^*}(u,u)\} 
\leq\dE_{\mu^*}[v\cdot(\nabla^2H^*)^{-1}v].
\end{equation}

\subsection{Brascamp--Lieb inequality}

\begin{lemma}[Brascamp--Lieb inequality]\label{lemma:brascamp lieb inequality}
Let $\mu$ satisfy Assumption~\ref{assumption:gibbs measure}. For every rotation-invariant $F\in H^1(\mu)$,
\begin{equation}\label{eq:bblieb}
\Var_\mu[F]\leq\dE_\mu[\nabla F\cdot(\nabla^2H)^{-1}\nabla F],
\end{equation}
where the inverse acts on $\{u\in\dR^N:\sum_i u_i=0\}$.

More generally, let $\Omega\subset\dR^d$ be a bounded open convex domain and let $\nu$ have density proportional to $e^{-V}$ on $\Omega$, with $V\in\mathcal C^2(\Omega)$ and $\nabla^2V\geq\kappa I_d$ for some $\kappa>0$. Then, for every $F\in H^1(\nu)$,
\begin{equation}\label{eq:euclidean BL}
\Var_\nu[F]\leq\dE_\nu[\nabla F\cdot(\nabla^2V)^{-1}\nabla F].
\end{equation}
\end{lemma}

\begin{proof}
Bobkov and Ledoux~\cite[Proposition~2.1]{bobkov2000brunn} proved~\eqref{eq:euclidean BL} for smooth $F$, on an open convex domain and for a $\mathcal C^2$ potential with positive definite Hessian. The bound $(\nabla^2V)^{-1}\leq\kappa^{-1}I_d$ makes the right-hand side of~\eqref{eq:euclidean BL} continuous in $H^1(\nu)$, so approximation gives~\eqref{eq:euclidean BL} for every $F\in H^1(\nu)$.

For~\eqref{eq:bblieb}, we use the orthogonal decomposition
\[
z=X_N-\frac1N\Bigl(\sum_i x_i\Bigr)(1,\ldots,1),\qquad
\omega=\frac1{\sqrt N}\sum_i x_i\pmod{\sqrt N}.
\]
The variable $\omega$ is uniform and independent of $z$. The latter ranges over the bounded convex domain
\[
\Bigl\{z\in\dR^N:\sum_i z_i=0,\ z_1<\cdots<z_N<z_1+1\Bigr\}.
\]
For $u$ with $\sum_i u_i=0$, Assumption~\ref{assumption:gibbs measure} gives
\[
u\cdot\nabla^2H u=2\sum_{i<j}\chi_{ij}''(x_j-x_i)(u_j-u_i)^2\geq2cN|u|^2.
\]
Since $F$ depends only on $z$, we apply~\eqref{eq:euclidean BL} on this orthogonal subspace. This gives~\eqref{eq:bblieb} with the stated inverse.
\end{proof}

\subsection{Log-Sobolev inequalities and Gaussian concentration}

We prove Gaussian concentration for functions of the gaps in a block.
We first recall Brascamp and Lieb's convexity result for marginals
\cite{Brascamp2002}.

\begin{lemma}\label{lemma:convexity exterior}
Let $\mu$ satisfy Assumption~\ref{assumption:gibbs measure}. Let $I\subset\{1,\ldots,N\}$ have cardinality $1\leq m\leq N$, and let $\pi_I$ be the projection on the coordinates $(x_i)_{i\in I}$. Assume, after a cyclic relabeling of the particles, that the label $1$ belongs to $I$. Let
\[
H_1((x_i)_{i\in I})=\sum_{i\neq j\in I}\chi_{ij}(|x_i-x_j|),
\]
and let
\[
H_2(X_N)=\sum_{i\neq j\in I^c}\chi_{ij}(|x_i-x_j|)+2\sum_{i\in I^c}\sum_{j\in I}\chi_{ij}(|x_i-x_j|),
\]
so that $H=H_1\circ\pi_I+H_2$. Let
\begin{equation}\label{def:wh}
\tilde H(z)=-\log\int e^{-H_2(z,y)}\mathds{1}_{D_N}(z,y)\dd y,
\end{equation}
where $z=(x_i)_{i\in I}$ and $y=(x_i)_{i\in I^c}$ are the lifted coordinates. Then the density of $\pi_I\#\mu$ is given by
\begin{equation}\label{eq:nuw}
\dd(\pi_I\#\mu)(z)\propto e^{-(H_1+\tilde H)(z)}\mathds{1}_{D_m}(z)\dd z,
\end{equation}
and $\tilde H$ is convex on the convex set of lifts $z$ satisfying $z_i<z_j<z_i+1$ for all $i<j$ in $I$. In particular, $\nabla^2\tilde H\geq0$ in the distributional sense.
\end{lemma}

\begin{proof}
Integrating out the exterior gives~\eqref{eq:nuw}.
If $m=N$, then the integral in~\eqref{def:wh} equals one and $\tilde H=0$.
Suppose that $m<N$.

The set of ordered lifts is convex, and $H_2$ is a sum of convex functions of linear separations. Thus $e^{-H_2}\mathds{1}_{D_N}$ is log-concave in $(z,y)$. Since $1\in I$, each fiber over an admissible $z$ is a nonempty bounded open set. The pair potentials are bounded below, so the integral over each fiber is finite and positive. By the Pr\'ekopa--Leindler inequality~\cite[Equation~(1.1)]{bobkov2000brunn} applied to the fibers, the integral in~\eqref{def:wh} is log-concave in $z$. Hence $\tilde H$ is convex, and its distributional Hessian is nonnegative.

The same proof applies to a finite convex potential $V$ on a bounded convex polytope, after an affine change of coordinates in its affine span. We use this form for the gaps.
\end{proof}

\begin{lemma}[Bakry--Émery~\cite{bakryemery}]\label{lemma:log sobolev}
Let $\Omega\subset\dR^{d}$ be a convex domain and let $w>0$.
Let $\gamma_\Omega^w$ be $\mathcal N(0,wI_{d})$ conditioned on $\Omega$.
Suppose that $\nu=f\gamma_\Omega^w$ is a probability
measure, where $f$ is Borel and log-concave.
For every $F\in H^1(\nu)$ with bounded gradient and $t\in\dR$,
\[
\log\dE_\nu[e^{tF}]
\leq t\dE_\nu[F]+\frac w2t^2\sup_\Omega|\nabla F|^2.
\]
\end{lemma}

\begin{proof}
We extend $-\log f$ by $+\infty$ outside $\Omega$ and take its lower
semicontinuous closure. Infimal convolution followed by smoothing gives
smooth convex approximations of this closure with a common affine lower
bound. The corresponding perturbations of $\mathcal N(0,wI_d)$ have
potentials with Hessian at least $w^{-1}I_d$. By the Bakry--Émery
criterion~\cite[(2.17), p.~147]{ledoux1999concentration}, these measures
satisfy the logarithmic Sobolev inequality with constant $2w$. This
inequality extends to $\nu$ by dominated convergence and density in
$H^1(\nu)$. Herbst's
argument~\cite[Proposition~2.3, pp.~149--150]{ledoux1999concentration},
applied to truncations of $F$, then gives the stated bound.
\end{proof}

The following estimate is the circular version of Bourgade, Erd\H{o}s and
Yau's concentration bound \cite[Lemma~3.9 and Corollary~3.10]{bourgade2014universality}.

\begin{lemma}\label{lemma:div free}
Let $\mu$ satisfy Assumption~\ref{assumption:gibbs measure}.
Let $I$ be a nonempty cyclic interval of labels and let $\pi_I$ be the
projection on $(x_i)_{i\in I}$. Suppose that, for some $c_1>0$,
\begin{equation}\label{eq:chiij'}
\chi_{ij}''(x_j-x_i)\geq\frac{c_1}{2},
\qquad i,j\in I,\quad i<j,\quad X_N\in D_N.
\end{equation}
Let $F=G\circ\pi_I$ with $G\in H^1(\pi_I\#\mu)$ and bounded gradient.
If $\sum_{i\in I}\partial_iF=0$, then, for every $t\in\dR$,
\[
\log\dE_\mu[e^{tF}]
\leq t\dE_\mu[F]+\frac{t^2}{2c_1|I|}\sup|\nabla F|^2.
\]
\end{lemma}

\begin{proof}
We apply Lemma~\ref{lemma:log sobolev} after removing the common rotation.
Set $m=|I|$. If $m=1$, then $F$ is constant.
Otherwise, a cyclic relabeling allows us to take $I=\{1,\ldots,m\}$,
with the reflected pair potentials prescribed after
Assumption~\ref{assumption:gibbs measure}.

By Lemma~\ref{lemma:convexity exterior}, the block marginal has density
proportional to $e^{-(H_1+\tilde H)}$ on $D_m$, with $\tilde H$ convex.
Both potentials are invariant under common rotations.
Writing $x=(x_1,\ldots,x_m)$, set
\[
e=m^{-1/2}(1,\ldots,1),\qquad
z=x-(x\cdot e)e,\qquad \omega=x\cdot e\pmod{\sqrt m}.
\]
The law is a product of the uniform law of $\omega$ and a measure $\nu$
with potential $H_1+\tilde H$ on the bounded convex chamber
\[
\{z\in e^\perp:z_1<\cdots<z_m<z_1+1\}.
\]
For $u\in e^\perp$, \eqref{eq:chiij'} gives
\[
u\cdot\nabla^2H_1u
=2\sum_{1\leq i<j\leq m}\chi_{ij}''(z_j-z_i)(u_j-u_i)^2
\geq c_1\sum_{1\leq i<j\leq m}(u_j-u_i)^2
=c_1m|u|^2.
\]
Thus $(H_1+\tilde H)(z)-c_1m|z|^2/2$ is convex.
The condition on $F$ gives $F=G(z)$ and
$|\nabla_zG|=|\nabla F|$.
Lemma~\ref{lemma:log sobolev}, with $w=(c_1m)^{-1}$, yields
\[
\log\dE_\mu[e^{tF}]
=\log\dE_\nu[e^{tG}]
\leq t\dE_\mu[F]+\frac{t^2}{2c_1m}\sup|\nabla F|^2.
\]
\end{proof}

\section{Near-optimal rigidity}\label{section:rigidity}
This section is dedicated to the proof of the rigidity result stated in Theorem~\ref{theorem:almost optimal rigidity}. The approach draws on several techniques introduced in the seminal work~\cite[Theorem~3.1]{bourgade2014universality}. Since our setting is the circle, the expected values of the gaps can be computed explicitly. As a result, we avoid one of the main technical challenges of~\cite{bourgade2014universality}, namely estimating the accuracy of the standard positions. Our first objective is to establish the following local law for the gaps: we show that for every $i\in  \{1,\ldots,N\}$ and $1 \leq k \leq \lfloor N/2\rfloor$, the quantity $N(x_{i+k} - x_i)$ is typically of order $k$. To achieve this, we use log-concavity of the gap distribution.
The argument relies on a convexification procedure, which we describe first.

\subsection{Comparison with a constrained Gibbs measure}
Since the Hessian of the energy degenerates when particles are far apart, one cannot directly obtain Gaussian concentration estimates for $\dGi$. Following~\cite{bourgade2014universality}, we add a convexifying term to the Hamiltonian that penalizes configurations with large gaps. Let $\theta:\dR^+\to\dR^+$ be a smooth function that vanishes on $[0,1]$ and satisfies $0\leq\theta''\leq2$, with $\theta''=2$ on $[2,\infty)$. We obtain such a function by integrating twice from $0$ a smooth function with values in $[0,2]$ that is $0$ on $[0,1]$ and $2$ on $[2,\infty)$. In particular, $0\leq\theta'(x)\leq2x$ and $0\leq\theta(x)\leq x^2$. Let $i_0\in\{1,\ldots,N\}$, $L\in\{1,\ldots,N-1\}$ and $R>0$. Let us define
\begin{equation*}
I:=\{i_0,\ldots,i_0+L\},
\end{equation*}
with indices taken modulo $N$ and positions represented by the cyclic lift $x_{j+N}=x_j+1$ along this interval. Then, we define the forcing term
\begin{equation}\label{eq:def add F k_i}
\mathrm{F}:=2\sum_{i,j\in I:i<j}\theta\Bigl(\frac{N}{R}(x_{j}-x_{i})\Bigr)
\end{equation}
and the locally constrained Gibbs measure
\begin{equation}\label{eqdef:locally constrained}
\dd\dGiQ(X_N)\propto e^{-\beta(\Hc_N+\mathrm{F})(X_N) }\mathds{1}_{D_N}(X_N)\dd X_N.
\end{equation}

Recall that the total variation distance between two probability measures $\mu$ and $\nu$ on a measurable space $E$ is defined by
\begin{equation*}
    \mathrm{TV}(\mu,\nu)=\sup_{\mc{A}\subset E}|\mu(\mc{A})-\nu(\mc{A})|,
\end{equation*}
where the supremum runs over measurable subsets of $E$.

\begin{lemma}
Let $\theta$, $\mathrm{F}$ and $\dGiQ$ be as in~\eqref{eq:def add F k_i}--\eqref{eqdef:locally constrained}. Then
\begin{equation}\label{eq:TV forcing}
\mathrm{TV}(\dGi,\dGiQ)\leq\dGi(\mathrm F>0)\leq\dGi\bigl(N(x_{i_0+L}-x_{i_0})>R\bigr).
\end{equation}
\end{lemma}
\begin{proof}
Since $\mathrm F\geq0$, the density $\dd\dGiQ/{\dd\dGi}$ is at least $1$ on $\{\mathrm F=0\}$. Thus,
\begin{equation*}
\begin{split}
\mathrm{TV}(\dGi,\dGiQ)
&=1-\dE_{\dGi}\left[\min\left\{1,\frac{e^{-\beta\mathrm F}}{\dE_{\dGi}[e^{-\beta\mathrm F}]}\right\}\right]\\
&\leq\dGi(\mathrm F>0)
\leq\dGi\bigl(N(x_{i_0+L}-x_{i_0})>R\bigr).
\end{split}
\end{equation*}
\end{proof}

\subsection{First local law}
We now prove that each gap $N(x_{i+k}-x_i)$ is typically of order $k$ with an exponentially small probability of deviation.

\begin{lemma}\label{lemma:initial estimate}
Let $N\geq2$ be an integer, and let $c_0=(\log2)/4$. Then, with the cyclic lifts $x_{j+N}=x_j+1$, for every integer $j$ and every integer $L$ with $1\leq L\leq N$,
\begin{equation}\label{eq:gap upper exponential}
\dE_{\dGi}\left[e^{c_0N(x_{j+L}-x_j)/L}\right]\leq3\,.
\end{equation}
Moreover, for every $\delta>0$, every integer $i$ and every integer $k$ with $1\leq k\leq N$,
\begin{equation*}
\dGi\bigl(N(x_{i+k}-x_i)\geq k^{1+\delta}\bigr)\leq3e^{-c_0k^\delta}\,.
\end{equation*}
\end{lemma}

\begin{proof}[Proof of Lemma~\ref{lemma:initial estimate}]
Let $y_j:=N(x_{j+1}-x_j)$. After integrating out the rotation, the joint law of the gaps has a log-concave density on the simplex $\{y_j>0:\sum_j y_j=N\}$. Indeed, each pair interaction is $g$ evaluated at a linear partial sum of the gaps divided by $N$, and $g''>0$ on $(0,1)$. The marginal-convexity argument in the proof of Lemma~\ref{lemma:convexity exterior} therefore shows that $y_1$ has a log-concave density $f$. Moreover, cyclic symmetry gives $\dE_{\dGi}[y_1]=1$.

Let $S(t):=\dGi(y_1\geq t)$. This function is also log-concave. To see this, extend $f$ by zero outside $(0,N)$ and write, for $0<t<N$,
\[
\frac{S(t)}{f(t)}=\int_0^\infty\frac{f(t+u)}{f(t)}\dd u.
\]
The integrand is nonincreasing in $t$ by log-concavity of $f$. Hence $(\log S)'=-f/S$ is nonincreasing. Since $S(0)=1$ and Markov's inequality gives $S(2)\leq 1/2$, we obtain $S(t)\leq 2e^{-(\log 2)t/2}$ for every $t\geq0$. It follows that
\[
\dE_{\dGi}[e^{c_0y_1}]
=1+c_0\int_0^\infty e^{c_0t}S(t)\dd t\leq3.
\]
Thus, by Jensen's inequality and cyclic symmetry, for every $j$ and $1\leq L\leq N$,
\[
\dE_{\dGi}\left[e^{c_0N(x_{j+L}-x_j)/L}\right]
\leq\frac1L\sum_{r=0}^{L-1}\dE_{\dGi}[e^{c_0y_{j+r}}]\leq3.
\]
Markov's inequality gives the second bound.
\end{proof}

The exponential moment bound~\eqref{eq:gap upper exponential} also controls the expectation of gaps under the locally constrained measure:

\begin{lemma}\label{lemma:expectation}
Let $1\leq L\leq N-1$ and $1\leq i\leq N$ be integers. Let $j$ and $k$ be integers with $i\leq j<j+k\leq i+L$. Let $\gamma\in(0,1)$, and let $\dGiQ$ be the locally constrained measure~\eqref{eqdef:locally constrained}, with $i_0=i$ and $R=L^{1+\gamma}$ in~\eqref{eq:def add F k_i}. Then there exists $C>0$ depending only on $\gamma$, $s$ and $\beta$ such that
\begin{equation*}
 |\dE_{\dGiQ}[N(x_{j+k}-x_j)]-k|\leq C.
\end{equation*}
\end{lemma}

\begin{proof}
We compare the expectations using the change of density. Write $G:=N(x_{j+k}-x_j)$. Cyclic symmetry gives $\dE_{\dGi}[G]=k$, and~\eqref{eq:gap upper exponential} gives $\dE_{\dGi}[(G-k)^2]\leq Ck^2$. Moreover,
\[
\dGi(\mathrm F>0)\leq\dGi\bigl(N(x_{i+L}-x_i)>L^{1+\gamma}\bigr)
\leq3e^{-c_0L^\gamma}.
\]
For $L\geq2$, Markov's inequality and $\dE_{\dGi}[N(x_{i+L}-x_i)]=L$ give
\[
\dE_{\dGi}[e^{-\beta\mathrm F}]
\geq1-L^{-\gamma}\geq1-2^{-\gamma}>0.
\]
Since $|e^{-\beta\mathrm F}-1|\leq\mathds{1}_{\mathrm F>0}$, Cauchy--Schwarz yields
\[
|\dE_{\dGiQ}[G]-k|
=\frac{\left|\dE_{\dGi}[(G-k)(e^{-\beta\mathrm F}-1)]\right|}{\dE_{\dGi}[e^{-\beta\mathrm F}]} 
\leq\frac{Ck e^{-c_0L^\gamma/2}}{1-2^{-\gamma}}\leq C,
\]
because $k\leq L$. If $L=1$, then $k=1$ and the event $N(x_{i+1}-x_i)\leq2$ has probability at least $1/2$. On this event, $\mathrm F\leq2\theta(2)\leq2$, because $\theta$ is nondecreasing, $\theta(1)=\theta'(1)=0$ and $\theta''\leq2$. Hence $\dE_{\dGi}[e^{-\beta\mathrm F}]\geq e^{-2\beta}/2$. The bound follows directly from $\dE_{\dGi}[G]=1$.
\end{proof}

\subsection{Reduction to a block average}
Following~\cite{bourgade2014universality}, we estimate position fluctuations by replacing $x_i$ with the average over a block centered at $x_i$.

For every integer $i$ and $1\leq k\leq \lfloor N/2\rfloor$, let $I_k(i)$ stand for the interval of indices
\begin{equation*}
I_k(i):=\{j\in\{1,\ldots,N\}:d(i,j)\leq k\}.
\end{equation*}
We use cyclic lifts satisfying $x_{j+N}=x_j+1$ and represent $I_k(i)$ by the indices from $i-k$ to $i+k$. When $N$ is even and $k=N/2$, we count the opposite index once, using the indices from $i-k+1$ to $i+k$. Define the block average
\begin{equation}\label{def:block average}
    x_i^{[k]}:=\frac{1}{|I_k(i)|}\sum_{j\in I_k(i)}x_j.
\end{equation}

\begin{lemma}[Comparison to a block average]\label{lemma:reduction av}
Let $\ve>0$ be small enough. Then there exist $C>0$ and $c>0$, depending only on $\beta$, $s$ and $\ve$, with the following property. Let $i$ and $k$ be integers with $1\leq k\leq\lfloor N/2\rfloor$, and let $x_i^{[k]}$ be the block average~\eqref{def:block average}. Then
\begin{equation*}
    \dGi( |N(x_i-x_i^{[k]})|\geq k^{s/2+\ve})\leq C e^{-c k^{\ve/(s+2)}}.
\end{equation*}
\end{lemma}

\begin{proof}
We use the exponential moment bound~\eqref{eq:gap upper exponential} and a multiscale decomposition.

We may assume that $1\leq i\leq N$. Choose an integer $q\geq2/\ve$, and let $k_m:=\lfloor k^{m/q}\rfloor$ for $1\leq m\leq q$. Define
\[
G_0:=N\bigl(x_i-x_i^{[k_1]}\bigr)
\]
and, for $1\leq m<q$,
\[
G_m:=N\bigl(x_i^{[k_m]}-x_i^{[k_{m+1}]}\bigr).
\]
Then
\[
N(x_i-x_i^{[k]})=\sum_{m=0}^{q-1}G_m.
\]
Fix $0\leq m<q$ and set
\[
B_m:=I_{k_{m+1}}(i),\qquad R:=2k^{(m+1)/q+\ve/(s+2)}.
\]
Let $\dGiQ$ be given by~\eqref{eqdef:locally constrained}, with the forcing term~\eqref{eq:def add F k_i} taken over $B_m$ using these lifts and the cutoff $\theta$ fixed above.

By~\eqref{eq:TV forcing}, $\mathrm{TV}(\dGi,\dGiQ)\leq\dGi(\mathrm F>0)$. The forcing term vanishes whenever $N(x_{i+k_{m+1}}-x_{i-k_{m+1}})\leq R$. Applying~\eqref{eq:gap upper exponential} with $L=2k_{m+1}$ and using $R/L\geq k^{\ve/(s+2)}$, we get
\begin{equation}\label{eq:tot}
\mathrm{TV}(\dGi,\dGiQ)\leq3e^{-c_0k^{\ve/(s+2)}}.
\end{equation}

We now bound the fluctuations of $G_m$ under $\dGiQ$. This function depends only on the variables in $B_m$, and
\[
\sum_{j\in B_m}\partial_jG_m=0,
\qquad \sup|\nabla G_m|^2\leq CN^2k^{-m/q}.
\]
For every pair in $B_m$, the second derivative of the interaction is at least $c\beta N^2R^{-(s+2)}$, where $c>0$ depends only on $s$. Indeed, at a separation $u\in(0,1)$ with $Nu\leq2R$, we have $g''(u)\geq s(s+1)u^{-s-2}$. For $Nu>2R$, we use $\theta''(Nu/R)=2$ and $R\geq1$. Thus, by Lemma~\ref{lemma:div free}, for every $t\in\dR$,
\[
\log\dE_{\dGiQ}[e^{t(G_m-\dE_{\dGiQ}[G_m])}]
\leq Ct^2\frac{R^{s+2}}{|B_m|k^{m/q}} 
\leq Ct^2k^{s(m+1)/q+1/q+\ve}
\leq Ct^2k^{s+1/q+\ve}.
\]

If $2k_{m+1}<N$, reflection about $x_i$, together with reversal of the cyclic labels, sends $x_{i+j}$ to $2x_i-x_{i-j}$. It preserves $\dGiQ$ and sends $G_m$ to $-G_m$, so $\dE_{\dGiQ}[G_m]=0$. If $N$ is even and $k_{m+1}=N/2$, then $R\geq N$ and $\mathrm F=0$. In this case, cyclic symmetry gives $\dE_{\dGi}[N(x_{i+r}-x_i)]=r$, and our convention for the opposite index gives $\dE_{\dGi}[N(x_i^{[N/2]}-x_i)]=1/2$. Hence, in all cases, $|\dE_{\dGiQ}[G_m]|\leq1/2$.

Consequently, for $k$ large enough, Markov's inequality and $1/q\leq\ve/2$ give
\begin{equation}\label{eq:Gm fluct}
\dGiQ\left(|G_m|\geq\frac1q k^{s/2+\ve}\right)
\leq Ce^{-ck^{\ve-1/q}}\leq Ce^{-ck^{\ve/2}}.
\end{equation}
Combining~\eqref{eq:Gm fluct} with~\eqref{eq:tot} and taking a union bound over $m$, we obtain
\[
\dGi\left(|N(x_i-x_i^{[k]})|\geq k^{s/2+\ve}\right)
\leq\sum_{m=0}^{q-1}\dGi\left(|G_m|\geq\frac1q k^{s/2+\ve}\right)
\leq Ce^{-ck^{\ve/(s+2)}}.
\]
The constants depend only on $\beta$, $s$, and $\ve$. Increasing $C$ covers the remaining bounded values of $k$ and concludes the proof.
\end{proof}

\subsection{Proof of Theorem~\ref{theorem:almost optimal rigidity}}
By Lemma~\ref{lemma:reduction av}, $x_i$ and $x_{i+k}$ can be replaced by their block averages at scale $k$, up to a well-controlled error. Moreover, the fluctuations of the difference between these block averages can be bounded using Lemma~\ref{lemma:div free}.

\begin{proof}[Proof of Theorem~\ref{theorem:almost optimal rigidity}]
We first prove the gap estimate and then pass to the number of particles in an interval. Let $i\in\{1,\ldots,N\}$ and $1\leq k\leq\lfloor N/2\rfloor$. Fix $\ve>0$ small enough that Lemma~\ref{lemma:reduction av} applies with $\ve/2$ and $s/2+\ve<1$. We may assume $k$ is large enough, since bounded values of $k$ are covered by increasing $C$. Set
\[
G:=N(x_{i+k}^{[k]}-x_i^{[k]})-k.
\]
Using the cyclic lifts and the block average~\eqref{def:block average}, we write
\begin{equation}\label{eq:key splitting}
N(x_{i+k}-x_i)-k
=N(x_{i+k}-x_{i+k}^{[k]})-N(x_i-x_i^{[k]})+G.
\end{equation}
By Lemma~\ref{lemma:reduction av} with $\ve/2$, for $j=i$ and $j=i+k$,
\begin{equation}\label{eq:bout 1}
\dGi\left(|N(x_j-x_j^{[k]})|\geq\frac13 k^{s/2+\ve}\right)
\leq Ce^{-ck^{\ve/(2(s+2))}}.
\end{equation}

Suppose first that $3k<N$. Let $I:=\{i-k,\ldots,i+2k\}$ and let $\dGiQ$ be the constrained measure~\eqref{eqdef:locally constrained} with
\[
R:=(3k)^{1+\ve/(2(s+2))}.
\]
This block contains all variables on which $G$ depends. By~\eqref{eq:TV forcing} and~\eqref{eq:gap upper exponential},
\begin{equation}\label{eq:diff TV}
\mathrm{TV}(\dGi,\dGiQ)
\leq3e^{-c_0(3k)^{\ve/(2(s+2))}}.
\end{equation}
Moreover, $\sup|\nabla G|^2\leq CN^2/k$ and $\sum_{j\in I}\partial_jG=0$. The same curvature bound as in Lemma~\ref{lemma:reduction av} and Lemma~\ref{lemma:div free} give
\[
\log\dE_{\dGiQ}[e^{t(G-\dE_{\dGiQ}[G])}]
\leq Ct^2\frac{R^{s+2}}{k^2}
\leq Ct^2k^{s+\ve/2}.
\]
Since
\[
G=\frac1{2k+1}\sum_{r=-k}^k\bigl(N(x_{i+k+r}-x_{i+r})-k\bigr),
\]
Lemma~\ref{lemma:expectation}, applied in the block $I$ of length $3k$, gives $|\dE_{\dGiQ}[G]|\leq C$. Markov's inequality therefore yields
\[
\dGiQ\left(|G|\geq\frac13 k^{s/2+\ve}\right)
\leq Ce^{-ck^{3\ve/2}}.
\]
Combining this with~\eqref{eq:diff TV} gives the same bound as in~\eqref{eq:bout 1} for $G$ under $\dGi$.

If $3k\geq N$, we work directly under $\dGi$. Cyclic symmetry gives $\dE_{\dGi}[G]=0$. The pair interactions have second derivatives bounded below by $c\beta N^{-s}$. Applying Lemma~\ref{lemma:div free} with all $N$ indices gives
\[
\log\dE_{\dGi}[e^{tG}]
\leq Ct^2\frac{N^{s+1}}{k}\leq Ct^2k^s.
\]
Thus the required bound for $G$ follows in this case as well. Taking a union bound in~\eqref{eq:key splitting}, we have proved the stronger estimate
\begin{equation}\label{eq:gap intermediate}
\dGi\left(|N(x_{i+k}-x_i)-k|\geq k^{s/2+\ve}\right)
\leq Ce^{-ck^{\ve/(2(s+2))}}.
\end{equation}
In particular, this proves~\eqref{statement:gaps}.

We now prove~\eqref{statement:dis}. Recall that $K=N\ell_N$. Write $\lambda:=K^{s/2+\ve}$, and let $M$ be the number of particles in $(x-\ell_N,x+\ell_N)$. We may assume $K$ is large enough that $\lambda\geq2$ and $\lambda+1\leq K$. Let $J$ be the cyclic label of the particle immediately preceding $x-\ell_N$, and write $y_j=N(x_{j+1}-x_j)$. Conditional on the gaps, the independent uniform rotation gives
\[
\dGi(J=j\mid(y_r)_{r=1}^N)=\frac{y_j}{N}.
\]
Consequently, for a family $(\mathcal A_j)$ of events depending on the gaps and obtained from each other by cyclic relabeling, cyclic symmetry and Cauchy--Schwarz give
\begin{equation}\label{eq:boundary particle}
\dGi(\mathcal A_J)
=\frac1N\sum_{j=1}^N\dE_{\dGi}[y_j\mathds{1}_{\mathcal A_j}]
=\dE_{\dGi}[y_1\mathds{1}_{\mathcal A_1}]
\leq C\dGi(\mathcal A_1)^{1/2},
\end{equation}
where~\eqref{eq:gap upper exponential} bounds $\dE_{\dGi}[y_1^2]$ uniformly in $N$.

Set
\[
k_+:=\lceil2K+\lambda\rceil,\qquad k_-:=\lfloor2K-\lambda\rfloor+1.
\]
On $\{M\geq2K+\lambda\}$, the particle $x_{J+k_+}$ lies in the interval, so $N(x_{J+k_+}-x_J)<2K+y_J$. On $\{M\leq2K-\lambda\}$, the particle $x_{J+k_-}$ lies at or beyond its right endpoint, so $N(x_{J+k_-}-x_J)\geq2K$. Hence
\begin{multline*}
\dGi(|M-2K|\geq \lambda)
\leq\dGi(y_J>\lambda/2)\\
+\dGi\left(|N(x_{J+k_+}-x_J)-k_+|\geq \lambda/2\right)+\dGi\left(|N(x_{J+k_-}-x_J)-k_-|\geq \lambda/2\right).
\end{multline*}
Since $N\geq4K$, the integers $k_\pm$ lie between $1$ and $N-1$, and $\min(k_\pm,N-k_\pm)$ lies between $K$ and $3K$. For $k>N/2$, the identity
\[
N(x_{j+k}-x_j)-k
=-\bigl(N(x_{j+N}-x_{j+k})-(N-k)\bigr)
\]
reduces a gap deviation to that of the complementary arc. If $K$ is large enough, then $\lambda/2\geq(3K)^{s/2+\ve/2}$. Hence~\eqref{eq:gap intermediate} with $\ve/2$, followed by~\eqref{eq:boundary particle}, bounds the last two terms by $Ce^{-cK^{\ve/(4(s+2))}}$. The same boundary-particle identity and~\eqref{eq:gap upper exponential} give $\dGi(y_J>\lambda/2)\leq Ce^{-c\lambda}$. We conclude that
\[
\dGi(|M-2N\ell_N|\geq(N\ell_N)^{s/2+\ve})
\leq Ce^{-c(N\ell_N)^{\ve/(4(s+2))}}.
\]
Increasing $C$ covers bounded $N\ell_N$ and concludes the proof.
\end{proof}

\subsection{Controlling nearby particles}

\begin{lemma}\label{lemma:no explosion}
For all $p>0$, there exists a constant $C>0$ depending on $\beta$, $s$ and $p$ such that
 \begin{equation*}
     \dE_{\dGi}\left[\frac{1}{(N(x_2-x_1))^p }\right]\leq C.
 \end{equation*}
\end{lemma}

\begin{proof}[Proof of Lemma~\ref{lemma:no explosion}]
We prove the bound by integration by parts under the law of one gap. Choose the cyclic lift
\[
x_1<x_2<\cdots<x_N<x_1+1
\]
and let $y_i=N(x_{i+1}-x_i)$ for $1\leq i\leq N$, with $x_{N+1}=x_1+1$. The change of variables from the particle positions to $(x_1,y_1,\ldots,y_{N-1})$ has constant Jacobian, and the energy is independent of $x_1$. Thus, the law of the gaps has density proportional to $e^{-\beta H^{\gap}}$ with respect to surface measure on
\[
\{(y_1,\ldots,y_N)\in(0,\infty)^N:y_1+\cdots+y_N=N\},
\]
where
\[
H^{\gap}(y_1,\ldots,y_N)
=2N^{-s}\sum_{1\leq i<j\leq N}
g\left(\frac{y_i+\cdots+y_{j-1}}{N}\right).
\]
Cyclic invariance and the identity $y_1+\cdots+y_N=N$ give $\dE_{\dGi}[y_1]=1$.

We first describe the density of $y_1$. By~\eqref{eq:expression g}, for $0<t<1$,
\[
\frac{\dd^2}{\dd t^2}\bigl(g(t)-t^{-s}\bigr)
=s(s+1)\sum_{n\in\mathbb{Z}\setminus\{0\}}|t+n|^{-s-2}\geq0.
\]
In particular, $g$ and $t\mapsto g(t)-t^{-s}$ are convex. Subtracting $2y_1^{-s}$ from the interaction between particles $1$ and $2$ therefore leaves a convex function, as do all the other pair contributions. Hence $H^{\gap}-2y_1^{-s}$ is convex on the gap simplex. After eliminating $y_N=N-\sum_{j<N}y_j$, the remaining domain is convex. The marginal-convexity argument in the proof of Lemma~\ref{lemma:convexity exterior} then shows that the density $f$ of $y_1$ is of the form
\begin{equation}\label{eq:gap marginal density}
f(y)=e^{-2\beta y^{-s}-V(y)},\qquad 0<y<N,
\end{equation}
with $V$ convex. For $N=2$, there are no further coordinates to integrate out, and the same conclusion follows directly.

We next justify the integrations by parts. The function $g(t)-t^{-s}-(1-t)^{-s}$ is bounded on $(0,1)$. Since the other pair contributions are bounded below and the remaining coordinates range over a bounded set, for each fixed $N$ there exists $C_N>0$ such that
\begin{equation}\label{eq:gap marginal endpoints}
f(y)\leq C_Ne^{-2\beta(y^{-s}+(N-y)^{-s})},\qquad 0<y<N.
\end{equation}
In particular, all inverse moments are finite for fixed $N$. Moreover, $f$ is locally absolutely continuous and log-concave, so it increases up to its maximum and then decreases. Choose $b\in(0,N)$ such that $f$ is increasing on $(0,b)$. For every $q>0$, integration by parts on $(0,b)$ gives
\[
\int_0^b y^{-q}|f'(y)|\dd y
=b^{-q}f(b)+q\int_0^b y^{-q-1}f(y)\dd y<\infty,
\]
where the boundary term at zero vanishes by~\eqref{eq:gap marginal endpoints}. On $[b,N)$, the weight is bounded and $f$ has finite total variation. Differentiating~\eqref{eq:gap marginal density} almost everywhere gives
\[
V'(y)f(y)=2\beta s\,y^{-s-1}f(y)-f'(y).
\]
It follows that all the weighted terms below are absolutely integrable. Thus, using~\eqref{eq:gap marginal endpoints} to remove the boundary terms, we obtain
\begin{equation}\label{eq:gap marginal ibp}
\dE_{\dGi}[h'(y_1)]
=\dE_{\dGi}\left[h(y_1)\bigl(V'(y_1)-2\beta s\,y_1^{-s-1}\bigr)\right]
\end{equation}
for $h(y)=1$, $h(y)=y$, and $h(y)=y^{-q}$ with $q>0$.

For $r>0$, let $M_r=\dE_{\dGi}[y_1^{-r}]$. Taking $h=1$ in~\eqref{eq:gap marginal ibp} gives $\dE_{\dGi}[V'(y_1)]=2\beta sM_{s+1}$. Taking $h(y)=y$ next and using $\dE_{\dGi}[y_1]=1$, we obtain
\[
1=2\beta s(M_{s+1}-M_s)+\Cov_{\dGi}[y_1,V'(y_1)]
\geq2\beta s(M_{s+1}-M_s).
\]
Indeed, the covariance is nonnegative because $V'$ is nondecreasing. By H\"older's inequality,
\[
M_{s+1}\leq M_s+\frac{1}{2\beta s}
\leq M_{s+1}^{s/(s+1)}+\frac{1}{2\beta s}.
\]
Since $s/(s+1)<1$, this yields $M_{s+1}\leq C$ for a constant $C>0$ depending only on $\beta$ and $s$.

It remains to bound the higher inverse moments. Let $q>0$. Since $y\mapsto y^{-q}$ is decreasing and $V'$ is nondecreasing, \eqref{eq:gap marginal ibp} gives
\[
-qM_{q+1}
=-2\beta s(M_{q+s+1}-M_qM_{s+1})+\Cov_{\dGi}[y_1^{-q},V'(y_1)] 
\leq-2\beta s(M_{q+s+1}-M_qM_{s+1}).
\]
Consequently,
\[
M_{q+s+1}\leq CM_q+\frac{q}{2\beta s}M_{q+1}.
\]
Let $r=q+s+1$. Another application of H\"older's inequality gives
\[
M_r\leq CM_r^{q/r}+\frac{q}{2\beta s}M_r^{(q+1)/r}.
\]
Both exponents on the right are strictly smaller than $1$, so $M_r$ is bounded by a constant depending only on $\beta$, $s$, and $r$. This proves the desired bound for every $r>s+1$. For $0<r\leq s+1$, it follows from the bound on $M_{s+1}$ by H\"older's inequality. This concludes the proof.
\end{proof}

\section{Variance of singular linear statistics}\label{section:optimal scaling}
In this section, we prove the optimal variance estimate stated in Theorem~\ref{theorem:quantitative variance}.

\subsection{Regularization estimates}

Let $\rho\in\mathcal C_c^\infty((-1/2,1/2))$ be even and nonnegative, with $\int_{\dR}\rho=1$. For $0<h\leq1$, let $\rho_h$ denote the periodization, on the circle under consideration, of the function
\begin{equation}\label{def:mollifier}
x\mapsto h^{-1}\rho(h^{-1}x).
\end{equation}

\begin{lemma}\label{lemma:reg}
Let $\rho$ and $\rho_h$ be as in~\eqref{def:mollifier}, on $\dT$. Suppose that $\xi\in H^{(1-s)/2+\delta}(\dT)$ for some $\delta>0$, and let $r=\min(\delta,1)$. There exists $C>0$ depending on $\rho$ such that, for every $0<h\leq1$,
\[
\Vert\xi-\xi*\rho_h\Vert_{H^{(1-s)/2}(\dT)}^2
\leq Ch^{2r}\Vert\xi\Vert_{H^{(1-s)/2+r}(\dT)}^2.
\]
The same upper bound holds for
$\Vert\xi\Vert_{H^{(1-s)/2}(\dT)}^2-\Vert\xi*\rho_h\Vert_{H^{(1-s)/2}(\dT)}^2$.

Let $0<\ell\leq1$ and $0\leq\gamma<1$. Let $\rho_h$ be as in~\eqref{def:mollifier}, on the circle $\ell^{-1}\dT$, and let $f\in L^1(\ell^{-1}\dT)$. There exists $C>0$ depending only on $\rho$ and $\gamma$ such that the following two bounds are true for every $0<h\leq1$ and every $x\in\ell^{-1}\dT$. If $|f(y)|\leq|y|^{-\gamma}$ for every $y\neq0$, then
\[
|(f*\rho_h)(x)|\leq C\max(|x|,h)^{-\gamma}.
\]
If $f\in\mathcal C^1(\ell^{-1}\dT\setminus\{0\})$ and $|f'(y)|\leq|y|^{-1-\gamma}$ for every $y\neq0$, then
\begin{equation}\label{eq:mol psi'}
|(f*\rho_h)'(x)|\leq
\begin{cases}
C\max(|x|,h)^{-1-\gamma} & \text{if }\gamma>0,\\
C(1+|\log(h\ell)|)\max(|x|,h)^{-1} & \text{if }\gamma=0.
\end{cases}
\end{equation}
\end{lemma}

\begin{proof}
We use the Fourier multiplier of the mollifier. Since $\rho$ is even, nonnegative, and has integral one,
\[
|1-\widehat \rho(t)|^2\leq C\min(|t|^2,1),\qquad
0\leq1-|\widehat \rho(t)|^2\leq C\min(|t|^2,1).
\]
For $0<r\leq1$, both right-hand sides are at most $C|t|^{2r}$. Multiplication by $|k|^{1-s}|\widehat\xi(k)|^2$ and summation prove the two Sobolev estimates.

We next bound $f*\rho_h$. The kernel $\rho_h$ is supported in $\{|y|\leq h/2\}$. If $|x|\geq h$, then $|x-y|\geq|x|/2$ on this support, so $|(f*\rho_h)(x)|\leq2^\gamma|x|^{-\gamma}$. If $|x|\leq h$, then, since $\gamma<1$,
\[
|(f*\rho_h)(x)|\leq h^{-1}\Vert\rho\Vert_\infty\int_{|y|\leq3h/2}|y|^{-\gamma}\dd y\leq Ch^{-\gamma}.
\]

To prove~\eqref{eq:mol psi'}, subtract the constant $f((2\ell)^{-1})$. Integration of the derivative bound from $y$ to the antipodal point gives $|f(y)|\leq C|y|^{-\gamma}$ when $\gamma>0$, and $|f(y)|\leq\log((2\ell|y|)^{-1})$ when $\gamma=0$. If $|x|\geq h$, the convolution does not cross the singularity, and
\[
|(f*\rho_h)'(x)|
\leq\int |f'(x-y)|\rho_h(y)\dd y
\leq C|x|^{-1-\gamma}.
\]
If $|x|\leq h$ and $\gamma>0$, then differentiation of the kernel gives
\[
|(f*\rho_h)'(x)|
\leq Ch^{-2}\int_{|y-x|\leq h/2}|f(y)|\dd y
\leq Ch^{-1-\gamma}.
\]
For $\gamma=0$, the last bound is $Ch^{-1}(1+|\log(h\ell)|)$. These estimates prove the assertion.
\end{proof}

\begin{lemma}\label{lemma:a priori}
Let $\ell_N\in(0,1]$ with $N\ell_N\geq1$, and let $\xi\in\mathcal C^1(\ell_N^{-1}\dT)$. Let $\lambda\geq0$ and suppose that
\[
|\xi'(x)|\leq
\frac{\mathds{1}_{|x|\leq2}}{\max(|x|,(N\ell_N)^{-1})^{1+\lambda}}
+\frac{\mathds{1}_{|x|\geq2}}{|x|^{3-s}}.
\]
For every $\ve>0$, there exists $C>0$ depending on $\beta$, $s$, $\lambda$ and $\ve$ such that
\begin{equation}\label{eq:anch}
\Var_{\dGi}\left[\sum_{i=1}^N\xi(\ell_N^{-1}(x_i-x_1))\right]
\leq C(N\ell_N)^{\ve+\max(2\lambda,s)}
\end{equation}
and
\begin{equation}\label{eq:non anch}
\Var_{\dGi}\left[\sum_{i=1}^N\xi(\ell_N^{-1}x_i)\right]
\leq C(N\ell_N)^{\ve+\max(2\lambda,s)}.
\end{equation}
\end{lemma}

\begin{proof}
We first use the comparison principle on the anchored simplex, and then average over the common rotation. Recall that $K=N\ell_N$. Subtracting a constant from $\xi$ does not change either variance, so we take $\xi((2\ell_N)^{-1})=0$. The derivative bound gives, for every $\ve>0$,
\[
|\xi(x)|\leq C_\ve K^{\lambda+\ve}\quad (|x|\leq2),
\qquad |\xi(x)|\leq C|x|^{s-2}\quad (|x|\geq2).
\]
The logarithm when $\lambda=0$ is absorbed by $K^\ve$.
We also use the following consequence of Lemma~\ref{lemma:no explosion}. By Jensen's inequality, for every $p>0$,
\[
\dE_{\dGi}[(N(x_{i+j}-x_i))^{-p}]\leq C_pj^{-p},
\qquad 1\leq j<N,
\]
using cyclic lifts. The corresponding positive moments of each nearest-neighbor gap are bounded by~\eqref{eq:gap upper exponential}.

Fix $\gamma$ sufficiently large, depending on $s$ and $\lambda$, and recall the anchored map $\Anch_N$ from~\eqref{def:Anch}. Let $I_\gamma$ be the set of labels $1\leq i\leq N$ with $d(i,1)\leq K^\gamma$ if $K^\gamma<N/8$, and the set of all labels otherwise. Set $F_N=\sum_{i\in I_\gamma}\xi(\ell_N^{-1}(x_i-x_1))$ and write $F_N=F_N^*\circ\Anch_N$.

\paragraph{\emph{Step 1: reduction to a good event}}
Let $\eta_{\ve,\gamma}$ be the cutoff~\eqref{def:eta cutoff}, for this set $I_\gamma$. We write $\eta=\eta_{\ve,\gamma}=\eta^*\circ\Anch_N$ and interpret its gradients as weak derivatives. For every $B>0$,
\[
\dGi(\eta\ne1)\leq C_BK^{-B}.
\]
Indeed, there are at most a fixed power of $K$ factors. For $|k|\geq K^\delta$, Theorem~\ref{theorem:almost optimal rigidity} gives the assertion if its exponent is chosen smaller than $\ve/\gamma$. For $|k|<K^\delta$, with $\delta<\ve/(2(1-s/2))$, the exponential moment~\eqref{eq:gap upper exponential} and Jensen's inequality control the upper deviation, and the lower deviation is impossible for large $K$. The inverse gap moments control the remaining lower nearest-neighbor cutoff, with an arbitrarily large power. Since $|F_N|$ is bounded by a fixed power of $K$, it follows that
\[
\Var_{\dGi}[(1-\eta)F_N]\leq CK^{-10}.
\]

\paragraph{\emph{Step 2: convexification}}
Let $\theta$ be the smooth cutoff used in Section~\ref{section:rigidity}, with $0\leq\theta''\leq2$, $\theta=0$ on $[0,1]$, and $\theta''=2$ on $[2,\infty)$. Set
\[
\FF:=\sum_{\substack{i,i+1\in I_\gamma}}
\theta\left(\frac{N(x_{i+1}-x_i)}{K^\ve}\right),
\]
and define
\begin{equation}\label{eq:const}
\dd\dGiQ\propto e^{-\beta(\mathcal H_N+\FF)}\mathds{1}_{D_N}\dd X_N.
\end{equation}
The forcing vanishes unless one of these nearest-neighbor gaps exceeds $K^\ve$. Hence~\eqref{eq:gap upper exponential} gives, for some $c>0$,
\begin{equation}\label{eq:TV a prio}
\mathrm{TV}(\dGi,\dGiQ)\leq Ce^{-cK^\ve}.
\end{equation}
For large $K$, the density of $\dGiQ$ relative to $\dGi$ is at most two. Since $\eta F_N$ is polynomially bounded in $K$, its variances under these two measures differ by at most $Ce^{-cK^\ve}$. Put $\dGiQ^*=\Anch_N\#\dGiQ$ and $A_1=A_1^{\dGiQ^*}$.

\paragraph{\emph{Step 3: the cutoff term}}
For $U=(u_2,\ldots,u_N)$ and $u_1=u_{N+1}=0$, the convexification gives
\[
U\cdot\nabla^2(\mathcal H_N^*+\FF^*)U
\geq cK^{-(s+2)\ve}\sum_{\substack{i,i+1\in I_\gamma}}N^2(u_{i+1}-u_i)^2.
\]
Write $\eta^*=\widetilde\eta\circ\Gap_N^*$, where $\Gap_N^*(y)=\Gap_N(0,y)$. The vector energy bound~\eqref{eq:vector BL} therefore yields
\[
\dE_{\dGiQ^*}[F_N^*\nabla\eta^*\cdot A_1^{-1}(F_N^*\nabla\eta^*)]
\leq CK^{(s+2)\ve}\dE_{\dGiQ^*}[(F_N^*)^2|\nabla\widetilde\eta|^2\circ\Gap_N^*].
\]
The gradient is supported where $\eta\ne1$. There it is polynomially bounded in $K$, including the derivative of the inverse-gap cutoff. The probability bound in Step 1, with a sufficiently large $B$, bounds this energy by $CK^{-10}$.

\paragraph{\emph{Step 4: the comparison principle}}
For $y\in D_N^*$, write $d_i=d(i,1)$ and let
\[
b_i=\frac{K^\lambda}{(1+d_i)^{1+\lambda}}\mathds{1}_{d_i\leq K^{1+\ve}}
+\frac{K^{2-s}}{(1+d_i)^{3-s}}\mathds{1}_{d_i\geq K^{1-\ve}}
{}+\frac{K^\lambda}{(1+d_i)^{1+s/2}},
\qquad i\ne1.
\]
The gap bounds on the support of $\eta$ and the derivative assumption imply
\[
\eta^*|\partial_iF_N^*|\leq CK^{\kappa\ve}\eta^*Nb_i,
\qquad i\in I_\gamma\setminus\{1\},
\]
where $\kappa$ does not depend on $\ve$. Define
\[
G_N^*(y)=\sum_{i\in I_\gamma\setminus\{1\}}b_i\bigl(Ny_i-(i-1)\bigr).
\]
Its partial derivatives are nonnegative. Lemma~\ref{lemma:energy comparisons}, followed by sub-additivity, gives
\begin{equation*}
\dE_{\dGiQ^*}[\eta^*\nabla F_N^*\cdot A_1^{-1}(\eta^*\nabla F_N^*)]
\leq CK^{\kappa\ve}\Var_{\dGiQ^*}[\eta^*G_N^*] 
+CK^{\kappa\ve}\dE_{\dGiQ^*}[G_N^*\nabla\eta^*\cdot A_1^{-1}(G_N^*\nabla\eta^*)].
\end{equation*}
On the support of $\eta$, using the complementary arc for indices beyond the opposite particle gives
\[
|Ny_i-(i-1)|\leq K^\ve d_i^{s/2}.
\]
Consequently,
\[
|\eta^*G_N^*|\leq K^\ve\sum_{i\in I_\gamma\setminus\{1\}}b_i d_i^{s/2}
\leq CK^{\kappa\ve}(K^\lambda+K^{s/2}).
\]
The possible logarithm at $\lambda=s/2$ is absorbed by $K^{\kappa\ve}$, and so is the logarithm from the last term of $b_i$. The last energy is at most $CK^{-10}$ by Step 3, since the same bound on $G_N^*$ holds on the support of $\nabla\eta^*$. We obtain
\[
\dE_{\dGiQ^*}[\eta^*\nabla F_N^*\cdot A_1^{-1}(\eta^*\nabla F_N^*)]
\leq CK^{\kappa\ve}(K^{2\lambda}+K^s).
\]

\paragraph{\emph{Step 5: the anchored estimate}}
The anchored variance representation, Steps 3 and 4, and sub-additivity bound $\Var_{\dGiQ^*}[\eta^*F_N^*]$ by the same quantity. Steps 1 and 2 transfer this bound to $\Var_{\dGi}[F_N]$.
If $I_\gamma$ is not the full circle, the inverse block moments and Minkowski's inequality give, for every fixed $p\geq1$,
\[
\left\|\sum_{i\notin I_\gamma}\xi(\ell_N^{-1}(x_i-x_1))\right\|_{L^p(\dGi)}
\leq C K^{2-s}\sum_{j\geq K^\gamma}j^{s-2}+CK^{-10}\leq C,
\]
for $\gamma$ sufficiently large. To obtain the first bound, use the decay of $\xi$ when the microscopic distance exceeds $2K$. On the remaining event, the inverse block moments and Markov's inequality give, for every $A>0$, an $L^p$ bound $C K^{\lambda+\ve}(K/j)^A$. Choosing $A$ and then $\gamma$ sufficiently large makes its sum at most $CK^{-10}$. The full-circle case has no exterior term. Taking a smaller $\ve$ proves~\eqref{eq:anch}.

\paragraph{\emph{Step 6: the unanchored estimate}}
Let $J$ be the particle immediately preceding $0$. Using cyclic lifts, set $D=-Nx_J$ and $y_i=x_{J+i-1}-x_J$ for $1\leq i\leq N$. Integrating the uniform rotation and using cyclic symmetry gives the joint law of $(y_2,\ldots,y_N,D)$:
\[
\mathds{1}_{0<D<w(y)}\dd D\dd\dGi^*(y),
\qquad w(y)=Ny_2,\qquad \dGi^*=\Anch_N\#\dGi.
\]
Since $F_N^*$ is polynomially bounded, \eqref{eq:gap upper exponential} and Step 5 give
\begin{equation*}
\Var_{w\dGi^*}[F_N^*]
\leq K^\ve\Var_{\dGi^*}[F_N^*]
+4\Vert F_N^*\Vert_\infty^2\dE_{\dGi^*}[w\mathds{1}_{w>K^\ve}] 
\leq CK^{\kappa\ve}(K^{2\lambda}+K^s).
\end{equation*}

To estimate the translation error, define
\[
\Delta_D(y):=\sum_{i\in I_\gamma}\xi\left(\frac{Ny_i-D}{K}\right)-F_N^*(y).
\]
On $\{\eta^*(y)=1\}$, we have $D\leq K^\ve$ and at most $CK^\ve$ particles of $I_\gamma$ per microscopic unit interval. Summing the derivative bound gives
\[
|\Delta_D(y)|\leq CK^{\lambda+(3+\lambda)\ve}.
\]
Step 1, H\"older's inequality and the polynomial bound on $\Delta_D$ therefore yield
\[
\dE_{\dGi}[\Delta_D(y)^2]\leq CK^{(6+2\lambda)\ve}(1+K^{2\lambda}).
\]

For $i\notin I_\gamma$, we have
\[
\operatorname{dist}(Ny_i-D,N\mathbb Z)
\geq N\min(y_i-y_2,1-y_i).
\]
The inverse block moments and Cauchy--Schwarz with the weight $w$ thus give, for sufficiently large $A$ and $\gamma$,
\[
\left\|\sum_{i\notin I_\gamma}\xi\left(\frac{Ny_i-D}{K}\right)\right\|_{L^2(\dGi)}
\leq CK^{2-s}\sum_{j\geq K^\gamma}j^{s-2}
+C_AK^{\lambda+\ve}\sum_{j\geq K^\gamma}(K/j)^A\leq C.
\]
Sub-additivity and a smaller choice of $\ve$ prove~\eqref{eq:non anch}. For bounded $K$, the last term of $b_i$ gives the domination in Step~4. In Step~2, the gaps in the forcing have unit mean, so by Markov's inequality the forcing is bounded on an event of probability at least one half, and the density of $\dGiQ$ relative to $\dGi$ is bounded.
\end{proof}

\begin{lemma}\label{lemma:reg micro}
Let $\xi$ and $(\ell_N)$ satisfy Assumption~\ref{assumption:test func},
with $N\ell_N\geq1$ and $\alpha\in[s-1,s/2)$ as in~\eqref{eqdef:singularity psi''}.
Let $h=(N\ell_N)^{-1}$, and let $\rho$ and $\rho_h$ be as in~\eqref{def:mollifier}, on $\ell_N^{-1}\dT$. For every $\ve>0$, there exists $C>0$ depending
on $\beta$, $s$, $\xi$, $\ve$ and $\rho$ such that
\[
 \Var_{\dGi}\left[\Fluct_N[(\xi-\xi*\rho_h)(\ell_N^{-1}\cdot)]\right]
 \leq C(N\ell_N)^\ve
       \bigl((N\ell_N)^{2\alpha}+(N\ell_N)^{3s-2}\bigr).
\]
\end{lemma}

\begin{proof}
We split $e_h=\xi-\xi*\rho_h$ near the singularities. Write
$L=\ell_N^{-1}$.
We first assume $\bar\alpha>0$, $\alpha\ne0$, and $h$ small.
All distances are on $L\dT$. In the mesoscopic case of~\ref{item:A4}, fix $0<\sigma<1/2$ such that $\xi_0$ vanishes outside a closed subinterval of $(-\sigma,\sigma)$, and use the marks $\{0\}\cup\{a_j:|a_j|<1/2\}$, whose reductions are separated uniformly in $L\geq1$. By~\eqref{eq:an expression for psi}, the transports on two circles differ on $(-\sigma,\sigma)$ by a function with first two derivatives bounded uniformly, since the singular term of the kernel cancels. We compare with one long circle of the sequence, on which the other marks are at distance at least $1/2-\sigma$ from $(-\sigma,\sigma)$. This comparison and the derivatives of the kernel outside $(-\sigma,\sigma)$ give~\ref{item:A2}--\ref{item:A3} with these marks and a constant independent of $L$.

By~\ref{item:A1}, $\psi=-(2\beta c_s)^{-1}(-\Delta)^{(1-s)/2}\Xi$ is continuous with bounded
oscillation. Integrating the bound on $\psi''$ from a fixed subinterval
of each arc gives
\[
 |\psi'(x)|\leq C\left(1+\sum_j\operatorname{dist}(x,a_j)^{-\bar\alpha}\right).
\]
In the mesoscopic case, fix $R\geq1$ such that the marks lie in $[-R,R]$. For $L\geq4R$, this argument gives the bound with a constant independent of $L$. For $1\leq L<4R$, the comparison above with one such circle of the sequence transfers the bound to $(-\sigma,\sigma)$, and the derivatives of the kernel bound $\psi'$ outside $(-\sigma,\sigma)$.
Since $\bar\alpha<1$, continuity of $\psi$ excludes an atomic part in $\psi'$.
For $L>1$, the support condition in~\ref{item:A4} also gives
$|\psi'(x)|\leq C|x|^{s-2}$ for $|x|\geq1$.
Thus $\|\psi'\|_{L^1(L\dT)}\leq C$.

Let $P_L(t)=-2\beta L^{-s}g(t/L)$. Then
$\xi-\ell_N\int\xi=P_L*\psi'$ almost everywhere. Below, $\xi$ denotes the representative $\ell_N\int\xi+P_L*\psi'$, which is $\mathcal C^1$ off the marks, and the Riesz series gives
\[
 |P_L^{(m)}(t)|\leq C\operatorname{dist}(t,L\mathbb Z)^{-s-m},
 \qquad m=0,1,2.
\]
Set $r=\min(1,\operatorname{dist}(x,\{a_j\}))$.
Let $\chi$ be a smooth cutoff equal to one for $|t|\leq r/4$ and zero for $|t|\geq r/2$, with $|\chi'|\leq C/r$. We write $P_L*\psi'=(\chi P_L)*\psi'+((1-\chi)P_L)*\psi'$, and we differentiate $\psi'$ in the first term and the kernel in the second. Both integrals converge absolutely. The bound on $\psi''$ in the first term and the bound on $\psi'$ in the second give
\begin{equation}\label{eq:micro xi derivative}
 |\xi'(x)|\leq Cr^{-1-\bar\alpha}\int_0^{r/2}t^{-s}\dd t
              +Cr^{-s-\bar\alpha}\leq Cr^{-1-\alpha}.
\end{equation}
Let $E_h=P_L-P_L*\rho_h$, so that $e_h=E_h*\psi'$. With $d=\operatorname{dist}(t,L\mathbb Z)$, the bounds on $P_L^{(m)}$ and the mean value theorem give $\Vert E_h\Vert_{L^1(L\dT)}\leq Ch^{1-s}$ and $|E_h(t)|+d|E_h'(t)|\leq Ch^{1-s}d^{-1}$. Splitting $E_h$ with $\chi$ and differentiating as above give, for $r\geq2h$,
\begin{equation}\label{eq:micro far derivative}
 |e_h'(x)|\leq Ch^{1-s}r^{-1-\bar\alpha}.
\end{equation}
For $\alpha<0$, the coefficient bound in the proof of Lemma~\ref{lemma:uniform smoothing}, with $0<\delta<-\alpha$, shows that the Fourier series of $\xi$ converges absolutely. Its sum is continuous and agrees with $\xi$ off the marks, so $\xi$ extends continuously to each $a_j$.
Integration of~\eqref{eq:micro xi derivative} and convolution therefore give the first bound below. For $\alpha>0$, we split the integral defining $P_L*\psi'(x)$ at distance $1$ from $x$. Far from $x$, the kernel $P_L$ is bounded and $\|\psi'\|_{L^1(L\dT)}\leq C$. Near $x$, we use $|P_L(t)|\leq C|t|^{-s}$ and the bound on $\psi'$. Since $s+\bar\alpha=1+\alpha>1$, we obtain $|(P_L*\psi')(x)|\leq C(1+\sum_j\operatorname{dist}(x,a_j)^{-\alpha})$. The pointwise bound of Lemma~\ref{lemma:reg}, applied near each mark, bounds the convolution of this function with $\rho_h$ by $Ch^{-\alpha}$. Hence
\[
 |e_h(x)|\leq C
 \begin{cases}
 h^{-\alpha},&\alpha<0,\\
 h^{-\alpha}+\displaystyle\sum_j\operatorname{dist}(x,a_j)^{-\alpha},
    &\alpha>0.
 \end{cases}
\]
In particular, $|e_h|\leq Ch^{-\alpha}$ when $2h\leq r\leq4h$.

Let $I\subset N\dT$ be an interval of bounded length. The common rotation is uniform and
independent of the gaps. Lemma~\ref{lemma:no explosion} and Jensen's
inequality thus bound the expected number of particles in $I$, conditionally on $Nx_i=u$ with $u\in I$, by $C(1+|I|^2)$, uniformly in $i$ and $u$.
For every measurable $f$ supported in $I$, Cauchy--Schwarz and this bound give
\begin{equation}\label{eq:micro local second moment}
 \dE_{\dGi}\left[\left(\sum_i|f(Nx_i)|\right)^2\right]
 \leq C(1+|I|^2)\int_I|f(u)|^2\dd u.
\end{equation}

Choose a smooth cutoff $\chi_h$ equal to one for $r\leq2h$ and zero
for $r\geq4h$, with $|\chi_h'|\leq C/h$.
Each neighborhood has bounded length in microscopic coordinates.
Since $2\alpha<1$, \eqref{eq:micro local second moment} gives
\[
 \Var_{\dGi}\left[\Fluct_N[(\chi_he_h)(\ell_N^{-1}\cdot)]\right]
 \leq Ch^{-1}\int_{r\leq4h}|e_h(x)|^2\dd x\leq Ch^{-2\alpha}.
\]

Choose fixed smooth cutoffs $\chi_{a_j}$ equal to one near $a_j$, with
disjoint supports. On their transition regions,
\eqref{eq:micro xi derivative} gives $|e_h|\leq Ch$.
Together with~\eqref{eq:micro far derivative} and the bound for
$2h\leq r\leq4h$, this yields
\begin{gather*}
 \left|(\chi_{a_j}(1-\chi_h)e_h)'(x)\right|
 \leq\frac{Ch^{1-s}}{\max(\operatorname{dist}(x,a_j),h)^{1+\bar\alpha}},\\
 \left\|\bigl((1-\sum_j\chi_{a_j})(1-\chi_h)e_h\bigr)'\right\|_\infty
 \leq Ch^{1-s}.
\end{gather*}
When $L>1$, these functions have fixed bounded support.
Lemma~\ref{lemma:a priori}, applied to these pieces after translation, with $\lambda=\bar\alpha$
and $\lambda=0$, and sub-additivity give
\begin{equation*}
\Var_{\dGi}\left[\Fluct_N[((1-\chi_h)e_h)(\ell_N^{-1}\cdot)]\right]
 \leq Ch^{2-2s-\ve}(h^{-s}+h^{-2\bar\alpha}) 
 =Ch^{-\ve}(h^{2-3s}+h^{-2\alpha}).
\end{equation*}
Combining the two variance bounds proves the claim for small $h$.

A sufficiently small increase of $\alpha$ by less than $\ve/4$, with $\ve$ replaced by $\ve/2$, covers $\bar\alpha=0$ and $\alpha=0$. For bounded $N\ell_N$, the $L^2$ bounds on $\xi$ and
$\xi*\rho_h$ give the result by~\eqref{eq:micro local second moment}
in the mesoscopic case. When $\ell_N=1$, $N$ is bounded and the uniform
one-particle marginal gives the same conclusion.
\end{proof}

\begin{lemma}\label{lemma:uniform smoothing}
Let $\xi$ and $(\ell_N)$ satisfy Assumption~\ref{assumption:test func}, with $\alpha\in[s-1,s/2)$ as in~\eqref{eqdef:singularity psi''}, and let $\psi$ be the transport in that assumption. Let $\rho$ and $\rho_h$ be as in~\eqref{def:mollifier}, on $\ell_N^{-1}\dT$. For every $\ve>0$, there exists $C>0$ depending only on $\beta$, $s$, $\ve$, $\rho$ and the profile $\xi$ with its fixed data in Assumption~\ref{assumption:test func}, such that, for every $0<h\leq1$,
\begin{equation}\label{eq:uniform smoothing bounds}
\sigma_{\ell_N}(\xi-\xi*\rho_h)
\leq Ch^{s/2-\alpha-\ve}
\end{equation}
and, for $m\in\{1,2\}$,
\begin{equation}\label{eq:uniform derivative bounds}
\Vert\partial^m(\xi*\rho_h)\Vert_\infty
\leq Ch^{-m-\alpha-\ve},
\qquad \Vert\psi*\rho_h\Vert_\infty\leq C.
\end{equation}
Moreover, for all sufficiently small $h>0$, the functions $\xi*\rho_h$ and $\xi-\xi*\rho_h$ satisfy the bounds of Assumption~\ref{assumption:test func} with constants independent of $h$, after an arbitrarily small increase of $\alpha$ when $\alpha=s-1$.
\end{lemma}

\begin{proof}
Let $L=\ell_N^{-1}$ and first assume $\bar\alpha>0$. Write $\omega_k=2\pi k/L$ and use Fourier coefficients normalized by $L^{-1}$. In the mesoscopic case, let $R$ be as in the proof of Lemma~\ref{lemma:reg micro}, and first let $L\geq4R$. For $|\omega_k|\geq1$, split the Fourier integral of $\psi'$ at distance $|\omega_k|^{-1}$ from each singularity. The bound $|\psi'(x)|\leq C|x-a_j|^{-\bar\alpha}$ controls the near part. Integration by parts on the complement, using the bound on $\psi''$, gives, for every small $\delta>0$,
\[
|\widehat{\psi'}(k)|\leq CL^{-1}|\omega_k|^{\bar\alpha-1+\delta},\qquad
|\widehat\xi(k)|\leq CL^{-1}|\omega_k|^{\alpha-1+\delta}.
\]
In the mesoscopic case, $L\widehat\xi(k)=\widehat{\xi_0}(\omega_k/(2\pi))$. The points $\omega_k$ become dense as $\ell_N\to0$, so by continuity of $\widehat{\xi_0}$ the bound holds with $\omega_k$ replaced by every real $\omega$ with $|\omega|\geq1$, hence for every $L\geq1$. The low frequencies are bounded by $L^{-1}\Vert\xi\Vert_1$ when $L>1$. The definition of $\sigma_{\ell_N}$ reads
\[
\sigma_{\ell_N}^2(\xi)=\frac{L}{2\beta c_s}
\sum_{k\ne0}|\omega_k|^{1-s}|\widehat\xi(k)|^2.
\]
Fix $r\in(0,s/2-\alpha)$ with $r>s/2-\alpha-\ve$, and take $\delta<s/2-\alpha-r$. Comparing the sum with an integral gives $L\sum_{k\ne0}|\omega_k|^{1-s+2r}|\widehat\xi(k)|^2\leq C$. The first part of Lemma~\ref{lemma:reg}, rescaled to $L\dT$, therefore gives $\sigma_{\ell_N}(\xi-\xi*\rho_h)\leq Ch^r$, which proves~\eqref{eq:uniform smoothing bounds}. The rapid decay of the mollifier's Fourier transform gives the derivative estimates in~\eqref{eq:uniform derivative bounds}. Also, by~\eqref{eq:psi'} on $\ell_N^{-1}\dT$, the H\"older bound on $\Xi$ in~\ref{item:A1} bounds the integrand by an integrable function independent of $\ell_N$, so $\Vert\psi*\rho_h\Vert_\infty\leq\Vert\psi\Vert_\infty\leq C$. At $\bar\alpha=0$, an arbitrarily small increase of $\alpha$ absorbs the logarithm and proves the stated bounds.

For the last statement, we use the marks chosen in the proof of Lemma~\ref{lemma:reg micro}. Convolution preserves the H\"older bound on the centered primitive, and Lemma~\ref{lemma:reg} gives
\[
\begin{split}
|(\psi*\rho_h)'(x)|&\leq C\left(1+\sum_j\max(|x-a_j|,h)^{-\bar\alpha}\right),\\
|(\psi*\rho_h)''(x)|&\leq C\left(1+\sum_j\max(|x-a_j|,h)^{-1-\bar\alpha}\right)
\end{split}
\]
near the singularities, with the original integrable tails away from them. The support condition in~\ref{item:A4} persists for $h$ smaller than a fixed distance to the endpoints.
\end{proof}

\subsection{Mean-field transport}
We first introduce the transportation argument of~\cite{johansson1998fluctuations,shcherbina2013fluctuations}, which serves as the starting point for numerous CLTs for $\beta$-ensembles and Coulomb gases, including those in~\cite{leble2018fluctuations,bekerman2018clt,leble2018clt,serfaty2020gaussian}. Given a linear statistic $F$, the method involves constructing a mean-field approximation of the gradient $\nabla\phi$ of the solution to the equation
\[
-\Delta\phi+\beta \nabla \mc{H}_N\cdot \nabla\phi=F-\dE_{\dGi}[F]
\]
and then using this approximation to derive an expansion for the variance of $F$. This mean-field transport creates a nonlocal error term, sometimes called the ``loop equation term'', which is defined for every Lipschitz function $\psi:\ell_N^{-1}\dT\to\dR$ by
\begin{equation}\label{def:Loop}
    \Loop_{\ell_N}[\psi]:=\iint_{x\neq y}N\ell_N(\psi(\ell_N^{-1}x)-\psi(\ell_N^{-1}y))N^{-(1+s)}g'(x-y)\fluct_N(\dd x)\fluct_N(\dd y),
\end{equation}
where we recall that
\begin{equation*}
    \fluct_N=\sum_{i=1}^N \delta_{x_i}-N\dd x
\end{equation*}
and $\dd x$ is the Lebesgue measure on $\dT$.

\begin{remark}
For two-dimensional Coulomb gases, the loop equation term~\eqref{def:Loop} is replaced by an angle contribution. See~\cite{leble2018fluctuations,bauerschmidt2016two}.
In $\beta$-ensembles, the quantity in~\eqref{def:Loop} is smooth, and one can control it directly by bounding the fluctuation measure $\fluct_N$ via local laws.
In our setting, however, the term~\eqref{def:Loop} is as singular as the energy itself, so a more delicate \emph{analysis} is required.
\end{remark}

\begin{proposition}\label{proposition:loop equations}
Let $\ell_N\in (0,1]$ and $\xi\in \mathcal{C}^{\infty}(\ell_N^{-1}\dT)$. Let $\psi_0\in \mathcal{C}^{\infty}(\ell_N^{-1}\dT)$ be given by
\[
 \psi_0'=-\frac{1}{2\beta c_s}\ell_N^{1-s}(-\Delta)^{(1-s)/2}(\xi(\ell_N^{-1}\cdot))(\ell_N\cdot)\quad \text{and}\quad \int \psi_0=0.
\]
We have
\begin{align*}
    \Var_{\dGi}[\Fluct_N[\xi(\ell_N^{-1}\cdot)]]&=-\frac{1}{(N\ell_N)^{2(1-s)}}\dE_{\dGi}[\beta \Psi_0\cdot \nabla^2 \Hc_N \Psi_0+\Vert D\Psi_0\Vert_F^2 ] \\&+\frac{2}{(N\ell_N)^{1-s}}\dE_{\dGi}\left[\sum_{i=1}^N\xi'(\ell_N^{-1}x_i)\psi_0(\ell_N^{-1}x_i) \right] \\ &+\frac{1}{(N\ell_N)^{2(1-s)}}\Var_{\dGi}\left[\beta \Loop_{\ell_N}[\psi_0]-\sum_{i=1}^N \psi_0'(\ell_N^{-1}x_i) \right],
\end{align*}
where $\Vert\cdot\Vert_F$ is the Frobenius norm and
\begin{equation*}
    \Psi_0:X_N\in  D_N\mapsto \ell_N(\psi_0(\ell_N^{-1}x_1),\ldots, \psi_0(\ell_N^{-1}x_N)).
\end{equation*}

\end{proposition}

\begin{remark}[Mean-field approximation]
For $\ell_N=1$, Proposition~\ref{proposition:loop equations} gives a mean-field approximation for the gradient $\nabla\phi$ of the solution to
\begin{equation}\label{eq:A1xi}
\mc L^{\dGi}\phi=\Fluct_N[\xi]-\dE_{\dGi}\bigl[\Fluct_N[\xi]\bigr]\quad\text{in }D_N,
\end{equation}
where $\mc L^{\dGi}$ has the weak Neumann realization of Section~\ref{section:helffer}.
Because of the energy scaling and the long-range character of the interaction, one can seek an approximate solution within the class of \emph{diagonal transports}
\[
  \Psi : X_N=(x_1,\dots,x_N)\in D_N \longmapsto (\psi(x_1),\dots,\psi(x_N)),
  \qquad \psi:\dT\to\dR.
\]

For the hypersingular Riesz gas, namely the case $g(x)\sim |x|^{-s}$ with $s>1$, such a diagonal transport is no longer adequate to approximate the gradient of the solution to~\eqref{eq:A1xi}.
Here the energy is dominated by local interactions, so the term $-\dive(\Psi) + \beta\nabla\Hc_N\cdot\Psi$ fails to be, to leading order, a linear statistic.
\end{remark}

\begin{proof}[Proof of Proposition~\ref{proposition:loop equations}]
Let $\mc{L}:=\mc{L}^{\dGi}$ be the operator
\begin{equation*}
    \mc{L}=-\Delta +\beta \nabla \mc{H}_N\cdot \nabla,
\end{equation*}
with the weak realization of Section~\ref{section:helffer}. Let $\xi\in \mc{C}^{\infty}(\dT)$ and $F:=\Fluct_N[\xi]$.

Let $\psi \in\mathcal{C}^{\infty}(\dT)$ have zero mean. Define
\begin{equation*}
    \Phi:X_N\in D_N\mapsto \sum_{i=1}^N\int_0^{x_i}\psi,
\end{equation*}
which satisfies $\partial_i\Phi=\psi(x_i)$. We wish to find $\psi$ such that $\mc{L} \Phi\simeq F$. Let us expand $\nabla \Hc_N\cdot \nabla\Phi$. In terms of the empirical measure $\mu_N$, we have
\begin{equation*}
    \nabla \Hc_N\cdot \nabla\Phi=\iint_{x\neq y}N^{-(1+s)}g'(x-y)N(\psi(x)-\psi(y))\dd (N\mu_N)(x)\dd(N\mu_N)(y)
\end{equation*}
almost everywhere. By decomposing $\mu_N$ into $\mu_N=\dd x+N^{-1}\fluct_N$, where $\dd x$ is the Lebesgue measure on the circle, one can break $\nabla \Hc_N\cdot \nabla\Phi$ into
\begin{multline}\label{eq:loop 1}
    \nabla \Hc_N\cdot \nabla\Phi=N^2\iint N(\psi(x)-\psi(y))N^{-(1+s)}g'(x-y)\dd x \dd y\\+2N\int\Bigl(\int N(\psi(x)-\psi(y))N^{-(1+s)}g'(x-y)\dd y\Bigr)\fluct_N(\dd x) +\Loop_{1}[\psi],
\end{multline}
where $\Loop_1[\psi]$ is as in~\eqref{def:Loop} for $\ell_N=1$. For the cross term, we can write
\begin{equation*}
 N\int N(\psi(x)-\psi(y))N^{-(1+s)}g'(x-y)\dd y=-N^{1-s}g'*\psi,
\end{equation*}
where $g'*\psi=g*\psi'$ is well-defined since $\psi\in\mathcal{C}^\infty(\dT)$. The first term in~\eqref{eq:loop 1} equals $-N^{2-s}\int g'*\psi=0$. We thus get
\begin{equation}\label{eq:LphiF}
    \mc{L}\Phi-F=\int\Bigl(-2\beta N^{1-s} g'*\psi-\xi\Bigr)(x)\fluct_N(\dd x)-\Fluct_N[\psi']+\beta \Loop_{1}[\psi].
\end{equation}

Let $\psi_0\in \mathcal{C}^{\infty}(\dT)$ be the solution of the master equation
\begin{equation*}
    -2\beta g'*\psi_0=\xi-\int \xi\quad\text{with}\quad \int \psi_0=0.
\end{equation*}
Since $g$ is the fundamental solution of the fractional Laplace equation~\eqref{eq:frac laplace}, $\psi_0$ is the unique solution of
\[
    \psi_0'=-\frac{1}{2\beta c_s}(-\Delta)^{(1-s)/2}\xi\quad \text{with}\quad \int \psi_0=0.
\]
Let us choose $\psi=\psi_0/N^{1-s}$ so that by~\eqref{eq:LphiF},
\begin{equation}\label{eq:LP}
    \mc{L}\Phi-F=-\frac{1}{N^{1-s}}\Fluct_N[\psi_0']+\frac{\beta}{N^{1-s}}\Loop_{1}[\psi_0].
\end{equation}
One can now write
\begin{equation*}
    \Var_{\dGi}[F]=\Var_{\dGi}[F-\mc{L}\Phi]-\Var_{\dGi}[\mc{L}\Phi]+2\Cov_{\dGi}[F,\mc{L}\Phi].
\end{equation*}
By integration by parts under $\dGi$,
\begin{equation*}
  \Cov_{\dGi}[F,\mc{L}\Phi]=\dE_{\dGi}[\nabla F\cdot \nabla \Phi]
\end{equation*}
and
\begin{equation*}
    \Var_{\dGi}[\mc{L}\Phi]=\dE_{\dGi}[\beta \nabla\Phi\cdot \nabla^2\mc{H}_N\nabla\Phi+\Vert\nabla^2 \Phi\Vert_F^2 ].
\end{equation*}
Assembling the above gives
\begin{multline}\label{eq:ell=1}
    \Var_{\dGi}[F]=-\frac{1}{N^{2(1-s)}}\dE_{\dGi}\left[\beta \sum_{i\neq j}N^{-s}g''(x_i-x_j)(\psi_0(x_i)-\psi_0(x_j))^2+\sum_{i=1}^N \psi_0'(x_i)^2\right]\\ +\frac{2}{N^{1-s}}\dE_{\dGi}\left[\sum_{i=1}^N\xi'(x_i)\psi_0(x_i) \right]+\frac{1}{N^{2(1-s)}}\Var_{\dGi}\left[\beta \Loop_{1}[\psi_0]-\sum_{i=1}^N \psi_0'(x_i) \right].
\end{multline}

This establishes the proposition when $\ell_N=1$. We now prove it when $\ell_N\in(0,1)$ by rescaling. Let $\psi_0\in \mathcal{C}^{\infty}(\ell_N^{-1}\dT)$ be as in the statement. Let $\psi_1:=\ell_N^s\psi_0(\ell_N^{-1}\cdot)$. We can check that $\psi_1\in \mathcal{C}^\infty(\dT)$ solves
\begin{equation*}
    \psi_1'=-\frac{1}{2\beta c_s}(-\Delta)^{(1-s)/2}(\xi(\ell_N^{-1}\cdot))\quad \text{and}\quad \int_{\dT} \psi_1=0.
\end{equation*}
Indeed, by the definition of $\psi_1$, for all $x\in \dT$,
\begin{equation*}
    \psi_1'(x)=\ell_N^{s-1}\psi_0'(\ell_N^{-1}x)=-\frac{1}{2\beta c_s}(-\Delta)^{(1-s)/2}(\xi(\ell_N^{-1}\cdot))(x).
\end{equation*}
Therefore, by~\eqref{eq:ell=1},
\begin{multline}\label{eq:varpsi1}
    \Var_{\dGi}[\Fluct_N[\xi(\ell_N^{-1}\cdot)]]=-\frac{1}{N^{2(1-s)}}\dE_{\dGi}\left[\beta \sum_{i\neq j}N^{-s}g''(x_i-x_j)(\psi_1(x_i)-\psi_1(x_j))^2+\sum_{i=1}^N \psi_1'(x_i)^2\right] \\
    +\frac{2}{N^{1-s}}\dE_{\dGi}\left[\sum_{i=1}^N\ell_N^{-1}\xi'(\ell_N^{-1}x_i)\psi_1(x_i) \right] +\frac{1}{N^{2(1-s)}}\Var_{\dGi}\left[\beta \Loop_{1}[\psi_1]-\sum_{i=1}^N \psi_1'(x_i) \right].
\end{multline}
By the definition of $\psi_1$,
\begin{multline*}
   \frac{1}{N^{2(1-s)}}\dE_{\dGi}\left[\beta \sum_{i\neq j}N^{-s}g''(x_i-x_j)(\psi_1(x_i)-\psi_1(x_j))^2+\sum_{i=1}^N \psi_1'(x_i)^2\right]\\=\frac{1}{(N\ell_N)^{2(1-s)}}\dE_{\dGi}\left[\beta \sum_{i\neq j}N^{-s}g''(x_i-x_j)\ell_N^2(\psi_0(\ell_N^{-1}x_i)-\psi_0(\ell_N^{-1}x_j))^2+\sum_{i=1}^N \psi_0'(\ell_N^{-1}x_i)^2\right].
\end{multline*}
Moreover,
\begin{equation*}
  \frac{1}{N^{1-s}}\dE_{\dGi}\left[\sum_{i=1}^N\ell_N^{-1}\xi'(\ell_N^{-1}x_i)\psi_1(x_i) \right]= \frac{1}{(N\ell_N)^{1-s}}\dE_{\dGi}\left[\sum_{i=1}^N\xi'(\ell_N^{-1}x_i)\psi_0(\ell_N^{-1}x_i) \right].
\end{equation*}
Finally, recalling the definition of $\Loop_{\ell_N}$ from~\eqref{def:Loop},
\begin{equation*}
  \frac{1}{N^{2(1-s)}}\Var_{\dGi}\left[\beta \Loop_{1}[\psi_1]-\sum_{i=1}^N \psi_1'(x_i) \right]=\frac{1}{(N\ell_N)^{2(1-s)}}\Var_{\dGi}\left[\beta \Loop_{\ell_N}[\psi_0]-\sum_{i=1}^N \psi_0'(\ell_N^{-1}x_i) \right].
\end{equation*}
Thus, inserting the last three displays into~\eqref{eq:varpsi1}, we obtain
\begin{align*}
    & \Var_{\dGi}[\Fluct_N[\xi(\ell_N^{-1}\cdot)]]\\&=-\frac{1}{(N\ell_N)^{2(1-s)}}\dE_{\dGi}\left[\beta \sum_{i\neq j}N^{-s}g''(x_i-x_j)\ell_N^2(\psi_0(\ell_N^{-1}x_i)-\psi_0(\ell_N^{-1}x_j))^2+\sum_{i=1}^N \psi_0'(\ell_N^{-1}x_i)^2\right] \\
   &+\frac{2}{(N\ell_N)^{1-s}}\dE_{\dGi}\left[\sum_{i=1}^N\xi'(\ell_N^{-1}x_i)\psi_0(\ell_N^{-1}x_i) \right]+\frac{1}{(N\ell_N)^{2(1-s)}}\Var_{\dGi}\left[\beta \Loop_{\ell_N}[\psi_0]-\sum_{i=1}^N \psi_0'(\ell_N^{-1}x_i) \right],
\end{align*}
which concludes the proof of the proposition.
\end{proof}

\subsection{Variance expansion via control of the loop equation term}
Building on Proposition~\ref{proposition:loop equations}, we first establish Theorem~\ref{theorem:quantitative variance}, provided we can bound the variance of the loop equation term~\eqref{def:Loop}.

For $h_N\in[(N\ell_N)^{-1},1]$, let $\xi_\reg=\xi*\rho_{h_N}$, with $\rho_{h_N}$ as in~\eqref{def:mollifier} on $\ell_N^{-1}\dT$. Let $\psi_\reg$ be the mean-zero transport of $\xi_\reg$, given by
\begin{equation}\label{def:psi reg}
\psi_\reg'=-\frac1{2\beta c_s}\ell_N^{1-s}(-\Delta)^{(1-s)/2}
(\xi_\reg(\ell_N^{-1}\cdot))(\ell_N\cdot),\qquad \int\psi_\reg=0.
\end{equation}
For $\psi\in\mathcal C^\infty(\ell_N^{-1}\dT)$, define
\begin{equation}\label{def:Bell}
\dB_{\ell_N}[\psi]
:=\iint_{x\neq y}N^{-(s+2)}g''(x-y)(N\ell_N)^2
\bigl(\psi(\ell_N^{-1}x)-\psi(\ell_N^{-1}y)\bigr)^2
\fluct_N(\dd x)\fluct_N(\dd y).
\end{equation}

\begin{proposition}\label{prop:auxiliary}
Let $\xi$ and $(\ell_N)$ satisfy Assumption~\ref{assumption:test func}, with $N\ell_N\geq1$ and $\alpha\in[s-1,s/2)$ as in~\eqref{eqdef:singularity psi''}. Let $\xi_\reg$ and $\psi_\reg$ be as in~\eqref{def:psi reg}, with $h_N=(N\ell_N)^{-1}$. For every $\ve>0$, there exists $C>0$ depending on $\beta$, $s$, $\xi$, $\ve$ and $\rho$ such that
\begin{multline}\label{eq:mainV}
\left|\Var_{\dGi}[\Fluct_N[\xi(\ell_N^{-1}\cdot)]]
-(N\ell_N)^s\sigma_{\ell_N}^2(\xi)\right|\\
\leq C\left((N\ell_N)^{2(s-1)}\Var_{\dGi}[\Loop_{\ell_N}[\psi_\reg]]
+(N\ell_N)^{\max(2s-1,2\alpha)+\ve}\right).
\end{multline}
\end{proposition}

\begin{proof}
We apply the transport identity to the original statistic, while constructing the transport from its mollification. Recall that $K=N\ell_N$, and write $F=\Fluct_N[\xi(\ell_N^{-1}\cdot)]$. Let $\partial_i\Phi=K^{s-1}\ell_N\psi_\reg(\ell_N^{-1}x_i)$ for $1\leq i\leq N$. By the calculation in Proposition~\ref{proposition:loop equations},
\[
\mathcal L\Phi=\Fluct_N[\xi_\reg(\ell_N^{-1}\cdot)]
+K^{s-1}\left(\beta\Loop_{\ell_N}[\psi_\reg]
-\Fluct_N[\psi_\reg'(\ell_N^{-1}\cdot)]\right).
\]
The exact variance identity is
\begin{equation}\label{eq:original transport variance}
\Var_{\dGi}[F]=2\Cov_{\dGi}[F,\mathcal L\Phi]
-\Var_{\dGi}[\mathcal L\Phi]+\Var_{\dGi}[F-\mathcal L\Phi].
\end{equation}
All terms are finite. Indeed, $F\in L^2(\dGi)$ by the uniform one-particle marginal, and the inverse gap moments control $\mathcal L\Phi$.

Integration by parts gives
\[
\Cov_{\dGi}[F,\mathcal L\Phi]
=-K^s\int_{\ell_N^{-1}\dT}\xi\psi_\reg'.
\]
For nonsmooth $\xi$, this follows by first mollifying $\xi$ further, passing to the limit in $L^2(\dGi)$, and integrating its derivative onto the smooth function $\psi_\reg$ in the uniform one-particle integral. Moreover,
\[
\Var_{\dGi}[\mathcal L\Phi]
=\dE_{\dGi}[\beta\nabla\Phi\cdot\nabla^2\mathcal H_N\nabla\Phi+\Vert\nabla^2\Phi\Vert_F^2].
\]
In the expansion of the two empirical measures, the terms linear in $\fluct_N$ vanish in expectation by rotation invariance. The mean-field identity~\eqref{eq:mf var}, together with polarization of the Sobolev seminorm, therefore gives
\begin{multline}\label{eq:mixed mean field variance}
2\Cov_{\dGi}[F,\mathcal L\Phi]-\Var_{\dGi}[\mathcal L\Phi]
=K^s\sigma_{\ell_N}^2(\xi)-K^s\sigma_{\ell_N}^2(\xi-\xi_\reg)\\
-K^{2s-1}\Vert\psi_\reg'\Vert_{L^2(\ell_N^{-1}\dT)}^2
-\beta K^{2s-2}\dE_{\dGi}[\dB_{\ell_N}[\psi_\reg]].
\end{multline}
The regularization estimate~\eqref{eq:uniform smoothing bounds} bounds the second term by $CK^{2\alpha+\ve}$. Let $0<h\leq1$, and let $f=\xi*\rho_h$ or $f=\xi-\xi*\rho_h$. The Fourier coefficients of $f-f*\rho_{1/K}$ are those of $\xi-\xi*\rho_{1/K}$ multiplied by factors of modulus at most two, so $\sigma_{\ell_N}^2(f-f*\rho_{1/K})\leq4\sigma_{\ell_N}^2(\xi-\xi*\rho_{1/K})$, and the same bound holds for $f$. Recall that $\bar\alpha=1-s+\alpha$. The transport bounds and Lemma~\ref{lemma:reg} give
\[
\Vert\psi_\reg'\Vert_2^2\leq C_\ve K^\ve(1+K^{2\bar\alpha-1}).
\]
This estimate includes the endpoint $\bar\alpha=0$, by an arbitrarily small increase of $\alpha$. The bound~\eqref{eq:Bell general} of Lemma~\ref{lemma:Bell} gives
\[
\left|\dE_{\dGi}[\dB_{\ell_N}[\psi_\reg]]\right|\leq CK\Vert\psi_\reg'\Vert_2^2.
\]
Thus the last two terms in~\eqref{eq:mixed mean field variance} are bounded in absolute value by $CK^\ve(K^{2s-1}+K^{2\alpha})$.

It remains to estimate the last variance in~\eqref{eq:original transport variance}. Lemma~\ref{lemma:reg micro} gives
\[
\Var_{\dGi}[\Fluct_N[(\xi-\xi_\reg)(\ell_N^{-1}\cdot)]]
\leq CK^\ve(K^{2\alpha}+K^{3s-2}).
\]
The same bound holds for $f$ in place of $\xi$. Indeed, conditionally on the anchored configuration, the common rotation is uniform, and Parseval in this rotation, with the factors above, bounds the variance of the linear statistic of $f-f*\rho_{1/K}$ by four times that of $\xi-\xi*\rho_{1/K}$.
Also, Lemma~\ref{lemma:a priori}, applied with $\lambda=\bar\alpha$ to the transport derivative near each singularity and with $\lambda=0$ to its regular remainder, gives
\[
\Var_{\dGi}[\Fluct_N[\psi_\reg'(\ell_N^{-1}\cdot)]]
\leq CK^\ve(K^s+K^{2\bar\alpha}).
\]
For this application, use fixed cutoffs with disjoint supports around the marks chosen in the proof of Lemma~\ref{lemma:reg micro}. On their transition regions $\psi_\reg'$ and $\psi_\reg''$ are uniformly bounded, and the remaining derivative decays like $|x|^{s-3}$. Each piece therefore satisfies the derivative hypothesis of Lemma~\ref{lemma:a priori}, up to a constant. At $\bar\alpha=0$, first increase $\alpha$ by an arbitrarily small amount and absorb the resulting powers into $K^\ve$.
Sub-additivity now yields
\[
\Var_{\dGi}[F-\mathcal L\Phi]
\leq CK^{2s-2}\Var_{\dGi}[\Loop_{\ell_N}[\psi_\reg]]
+CK^\ve(K^{2\alpha}+K^{2s-1}),
\]
because $3s-2\leq2s-1$ and $K\geq1$. Combining this with~\eqref{eq:original transport variance} and~\eqref{eq:mixed mean field variance} proves~\eqref{eq:mainV}.
\end{proof}

\subsection{Control of the gradient of the loop equation term}

In view of Proposition~\ref{proposition:loop equations}, expanding the variance of a linear statistic reduces to controlling the loop equation term~\eqref{def:Loop}. We first illustrate the loss from a crude local estimate at the macroscopic scale $\ell_N=1$ with $N\geq2$. By using local laws on gaps, one may control the integral~\eqref{def:Loop} away from the diagonal. Nevertheless, $\Loop_{\ell_N}[\psi]$ contains local terms such as
\begin{equation*}
    \sum_{i=1}^N N\ell_N(\psi(\ell_N^{-1}x_{i+1})-\psi(\ell_N^{-1}x_i))N^{-(1+s)}g'(x_{i+1}-x_i).
\end{equation*}
At this scale, the local sum has $L^2(\dGi)$ norm at most $CN\Vert\psi'\Vert_{L^\infty(\dT)}$, by inverse block moments and Minkowski's inequality. The crude loop estimate~\eqref{eq:crude loop bound} gives the bound
\begin{equation*}
  \Var_{\dGi}\Bigl[\Loop_{\ell_N}[\psi]\Bigr]=O\Bigl((N\ell_N)^2\Vert\psi'\Vert_{L^\infty(\ell_N^{-1}\dT)}^2\Bigr).
\end{equation*}
Inserting this into Proposition~\ref{proposition:loop equations} gives an upper bound of order $O((N\ell_N)^{2s})$, too large to recover the expected variance.

Instead, we exploit the convexity of the Hamiltonian and bound the fluctuations of $\Loop_{\ell_N}[\psi]$ by means of various concentration inequalities. As emphasized in Section~\ref{section:helffer}, the variance of a smooth function under a log-concave probability measure is related to the norm of its gradient. The first task is therefore to control the gradient of the loop equation term~\eqref{def:Loop}.

We first define a localized version of $\Loop_{\ell_N}[\psi]$ and control its gradient.

Let $N\geq2$ be an integer with $N\ell_N\geq1$, and let $\gamma>0$. Let $t_\gamma=(N\ell_N)^\gamma/N$ if $(N\ell_N)^\gamma<N/8$ and $t_\gamma=1/2$ otherwise, and let $I_\gamma$ be the set of labels $1\leq i\leq N$ with $d(i,1)\leq Nt_\gamma$. For $\psi\in\mathcal C^\infty(\ell_N^{-1}\dT)$, let
\begin{multline}\label{def:Loop0}
\Loop_{\ell_N,\gamma}[\psi]
=\sum_{i\neq j\in I_\gamma}N^{-s}\ell_N
\bigl(\psi(\ell_N^{-1}x_i)-\psi(\ell_N^{-1}x_j)\bigr)g'(x_i-x_j)\\
-2N\sum_{i\in I_\gamma}\int_{|y|\leq t_\gamma}N^{-s}\ell_N
\bigl(\psi(\ell_N^{-1}x_i)-\psi(\ell_N^{-1}y)\bigr)g'(x_i-y)\dd y.
\end{multline}
For $\alpha\in[s-1,s/2)$ and $x\in\ell_N^{-1}\dT$, let
\begin{equation}\label{def:weights}
\wo(x)=\frac{\mathds{1}_{|x|<1}}{\max(|x|,(N\ell_N)^{-1})^{1-s+\alpha}}
+\frac{\mathds{1}_{|x|\geq1}}{|x|^{2-s}},
\qquad
\wt(x)=\frac{\mathds{1}_{|x|<1}}{\max(|x|,(N\ell_N)^{-1})^{2-s+\alpha}}
+\frac{\mathds{1}_{|x|\geq1}}{|x|^{3-s}}.
\end{equation}
For $\ve>0$, let $\chi_{1,2}(r)=\min(1,\max(0,2-r))$ for $r\geq0$, and define, with all gaps using cyclic lifts,
\begin{multline}\label{def:eta cutoff}
\eta_{\ve,\gamma}:=
\prod_{\substack{i,i+k\in I_\gamma\\0<|k|\leq\lfloor N/2\rfloor}}
\chi_{1,2}\left(\frac{2|N(x_{i+k}-x_i)-k|}{(N\ell_N)^\ve|k|^{s/2}}\right)\\
\times\prod_{\substack{i,i+1\in I_\gamma}}
\chi_{1,2}\left(\frac{2N(x_{i+1}-x_i)}{(N\ell_N)^\ve}\right)
\chi_{1,2}\left(\frac{2}{(N\ell_N)^\ve N(x_{i+1}-x_i)}\right).
\end{multline}

\begin{lemma}[Control of the gradient of the localized loop equation term]\label{lemma:gradient splitting}
Let $N\geq2$ be an integer with $N\ell_N\geq1$, and let $\alpha\in[s-1,s/2)$. Let $C_0>0$ and $\psi\in\mathcal C^\infty(\ell_N^{-1}\dT)$. For every sufficiently large $\gamma$, depending on $s$, and every sufficiently small $\ve>0$, the following is true. Let $t_\gamma$, $I_\gamma$ and $\Loop_{\ell_N,\gamma}[\psi]$ be as in~\eqref{def:Loop0}. Using cyclic lifts, let
\begin{multline*}
\mathcal A_{\ve,\gamma}
=\bigcap_{\substack{i,i+k\in I_\gamma\\0<|k|\leq\lfloor N/2\rfloor}}
\bigl\{|N(x_{i+k}-x_i)-k|\leq(N\ell_N)^\ve|k|^{s/2}\bigr\}\\
\cap\bigcap_{\substack{i,i+1\in I_\gamma}}
\bigl\{(N\ell_N)^{-\ve}\leq N(x_{i+1}-x_i)\leq(N\ell_N)^\ve\bigr\}.
\end{multline*}
Let $\wo$ and $\wt$ be as in~\eqref{def:weights}, and suppose that $|\psi'|\leq C_0\wo$ and $|\psi''|\leq C_0\wt$. In the case $t_\gamma<1/2$, suppose also that $\psi(\ell_N^{-1}x)=0$ for $|x|\geq t_\gamma/4$.

Then there exist $\dV$, $\dW$ and $\dW^\gap$, functions of $X_N\in D_N$ with values in $\dR^N$, such that $\nabla\Loop_{\ell_N,\gamma}[\psi]=\dV+\dW$, $\dV_i=\dW_i=0$ for $i\notin I_\gamma$, and $\dW_i^\gap=0$ unless $i,i+1\in I_\gamma$. There exist $C,\kappa>0$ such that, whenever $X_N\in\mathcal A_{\ve,\gamma}$ and the lift of $x_1$ satisfies $|Nx_1|\leq(N\ell_N)^\ve$,
\begin{equation}\label{eq:V1sup}
|\dV_i|\leq C(N\ell_N)^{\kappa\ve}
\left(\ell_N^{-1}\wt(\ell_N^{-1}x_i)
+\frac{N(N\ell_N)^{1-s+\alpha}}{(1+N\operatorname{dist}(x_i,\mathbb Z))^{1+s/2}}\right),\qquad i\in I_\gamma.
\end{equation}
Moreover, for every $U_N=(u_1,\ldots,u_N)\in\dR^N$, with $u_{N+1}=u_1$,
\[
\dW\cdot U_N=-\sum_{i=1}^N\dW_i^\gap N(u_{i+1}-u_i),
\]
and on the same event,
\begin{equation}\label{eq:nab tG}
|\dW^\gap|^2\leq C(N\ell_N)^{\kappa\ve}
\bigl(N\ell_N+(N\ell_N)^{2(1-s+\alpha)}\bigr).
\end{equation}
The constant $\kappa$ depends only on $s$, $\alpha$ and $\gamma$, and $C$ may also depend on $\ve$ and $C_0$.
\end{lemma}

The proof of Lemma~\ref{lemma:gradient splitting} is deferred to Appendix~\ref{section:auxiliary}.

For a general transport $\psi$, we cut it off before applying this lemma. The cutoff $\psi_\gamma$ is defined in Lemma~\ref{lemma:auxiliary Loop} below, and the derivative bounds in Lemma~\ref{lemma:gradient splitting} remain valid for it, with a different constant. Indeed, periodicity and the bound on $\psi'$ give
\[
|\psi(u)-\psi((2\ell_N)^{-1})|\leq C|u|^{s-1},\qquad |u|\geq1.
\]

\begin{lemma}\label{lemma:auxiliary Loop}
Let $N\geq2$ be an integer with $N\ell_N\geq1$, let $\alpha\in[s-1,s/2)$, and let $C_0>0$. Let $\psi\in\mathcal C^\infty(\ell_N^{-1}\dT)$ satisfy $|\psi'|\leq C_0\wo$ and $|\psi''|\leq C_0\wt$, with $\wo$ and $\wt$ as in~\eqref{def:weights}. For $\gamma>0$, let $t_\gamma$ and $I_\gamma$ be as in~\eqref{def:Loop0}. Let $\chi\in\mathcal C_c^\infty((-1/4,1/4))$ satisfy $0\leq\chi\leq1$ and $\chi=1$ on $[-1/8,1/8]$. Let $\psi_\gamma=\psi$ if $t_\gamma=1/2$, and otherwise let
\[
\psi_\gamma(u)=\chi(\ell_Nu/t_\gamma)
\bigl(\psi(u)-\psi((2\ell_N)^{-1})\bigr),
\qquad |u|\leq(2\ell_N)^{-1}.
\]
Let $\Loop_{\ell_N,\gamma}[\psi_\gamma]$ be given by~\eqref{def:Loop0} with $\psi_\gamma$ in place of $\psi$. For $\ve>0$, let $\eta_{\ve,2\gamma}$ be the cutoff~\eqref{def:eta cutoff} with $2\gamma$ in place of $\gamma$. Let $\dGi^*=\Anch_N\#\dGi$, with $\Anch_N$ as in~\eqref{def:Anch}. Let $F_\gamma^*$ be the function on the anchored simplex given by
\[
F_\gamma^*(y)=\eta_{\ve,2\gamma}(0,y)\Loop_{\ell_N,\gamma}[\psi_\gamma](0,y).
\]
For $\gamma$ sufficiently large and every sufficiently small $\ve>0$, there exist $C,\kappa>0$ such that
\begin{equation}\label{eq:aux step1}
\Var_{\dGi}[\Loop_{\ell_N}[\psi]]
\leq C(N\ell_N)^{\kappa\ve}\Var_{\dGi^*}[F_\gamma^*] 
+C(N\ell_N)^{\kappa\ve}\bigl(N\ell_N+(N\ell_N)^{2(1-s+\alpha)}\bigr).
\end{equation}
Here $\kappa$ depends only on $s$, $\alpha$ and $\gamma$, and $C$ may also depend on $\beta$, $\ve$, $C_0$ and $\chi$.
\end{lemma}

The proof of Lemma~\ref{lemma:auxiliary Loop} is also deferred to Appendix~\ref{section:auxiliary}.

\subsection{Control of the variance of the loop equation term}

We now control the variance of the localized loop equation term~\eqref{def:Loop0}, for $N\geq2$. Let $\dV$ and $\dW$ be as in Lemma~\ref{lemma:gradient splitting}. By sub-additivity, it suffices to bound $\dE[\dV\cdot A_1^{-1}\dV]$ and $\dE[\dW\cdot A_1^{-1}\dW]$ separately. For $\dW$, we use that, after the appropriate convexification, the measure becomes uniformly log-concave in gap coordinates. For $\dV$, we will use the comparison principle of Lemma~\ref{lemma:energy comparisons} and the rigidity estimate of Theorem~\ref{theorem:almost optimal rigidity}.

\begin{proposition}\label{prop:VarAl}
Let $\xi$ and $(\ell_N)$ satisfy Assumption~\ref{assumption:test func}, with $N\ell_N\geq1$ and $\alpha\in[s-1,s/2)$ as in~\eqref{eqdef:singularity psi''}. Let $h_N\in[(N\ell_N)^{-1},1]$, and let $\psi_\reg$ be as in~\eqref{def:psi reg}. For every $\ve>0$, there exists $C>0$ depending on $\beta$, $\ve$, $s$, $\xi$ and $\rho$ such that
\begin{equation}\label{eq:Loopvariance}
        \Var_{\dGi}[\Loop_{\ell_N}[\psi_\reg]]\leq C (N\ell_N)^{\ve+\max(1,2 (\alpha+1-s))}.
 \end{equation}
\end{proposition}

\begin{proof}
We may assume $N\geq2$. Recall that $K=N\ell_N$ and $\bar\alpha=1-s+\alpha$. We localize $\psi_\reg$ by fixed smooth cutoffs with disjoint supports around the marks chosen in the proof of Lemma~\ref{lemma:reg micro}. Linearity, rotation invariance and the variance triangle inequality reduce the proof to one translated piece or to the smooth remainder. In the rest of this proof, we denote this function by $\psi$. By the proof of Lemma~\ref{lemma:uniform smoothing}, it satisfies
\[
|\psi'|\leq C\wo,\qquad |\psi''|\leq C\wt,
\]
uniformly for $h_N\geq K^{-1}$. The case $\bar\alpha=0$ follows by increasing $\alpha$ slightly and absorbing the loss into $K^\ve$.

Fix $\gamma$ as in Lemma~\ref{lemma:auxiliary Loop}, and take $\ve>0$ sufficiently small.

\paragraph{\emph{Step 1: localization}}
Set
\[
\Loop^*(y):=\Loop_{\ell_N,\gamma}[\psi_\gamma](0,y),\qquad
\eta:=\eta_{\ve,2\gamma},\qquad \eta=\eta^*\circ\Anch_N.
\]
By~\eqref{eq:aux step1}, it suffices to bound $\Var_{\dGi^*}[\eta^*\Loop^*]$ by $CK^{\kappa\ve}(K+K^{2\bar\alpha})$. The constants $\kappa$ below do not depend on $\ve$.

\paragraph{\emph{Step 2: convexification}}
Let $\dGiQ$ be the measure in~\eqref{eq:const}, with $2\gamma$ in place of $\gamma$, and let $\dGiQ^*:=\Anch_N\#\dGiQ$. The estimate~\eqref{eq:TV a prio} holds for this measure. On the support of $\eta$, the lower gap bound and $|I_\gamma|\leq CK^\gamma$ bound $|\Loop^*|$ by a fixed power of $K$. The background integral is bounded using $s<1$. Hence
\begin{equation}\label{eq:Loopcomp}
\bigl|\Var_{\dGi^*}[\eta^*\Loop^*]-\Var_{\dGiQ^*}[\eta^*\Loop^*]\bigr|
\leq Ce^{-cK^\ve},
\end{equation}
after changing $c$. We now estimate the second variance.

\paragraph{\emph{Step 3: the gap field and the cutoff}}
Set $A_1=A_1^{\dGiQ^*}$. For $y\in D_N^*$, let
\[
\dV^*(y)=(\dV_i(0,y))_{i=2}^N,\qquad
\dW^*(y)=(\dW_i(0,y))_{i=2}^N,
\]
where the fields are those of Lemma~\ref{lemma:gradient splitting} for $\psi_\gamma$. Thus $\nabla\Loop^*=\dV^*+\dW^*$. The anchored variance representation and sub-additivity give
\begin{multline}\label{eq:etaLoop}
\Var_{\dGiQ^*}[\eta^*\Loop^*]
\leq4\dE_{\dGiQ^*}[\eta^*\dV^*\cdot A_1^{-1}(\eta^*\dV^*)]
+4\dE_{\dGiQ^*}[\eta^*\dW^*\cdot A_1^{-1}(\eta^*\dW^*)]\\
+2\dE_{\dGiQ^*}[\Loop^*\nabla\eta^*\cdot A_1^{-1}(\Loop^*\nabla\eta^*)].
\end{multline}
The convexification gives, for $U=(u_2,\ldots,u_N)$ and $u_1=u_{N+1}=0$,
\[
U\cdot\nabla^2(\mathcal H_N^*+\FF^*)U
\geq cK^{-(s+2)\ve}
\sum_{\substack{i,i+1\in I_{2\gamma}}}N^2(u_{i+1}-u_i)^2.
\]
By~\eqref{eq:vector BL}, the representation of $\dW^*$ by $\dW^\gap$ and Cauchy--Schwarz in these gap differences imply
\[
\dE_{\dGiQ^*}[\eta^*\dW^*\cdot A_1^{-1}(\eta^*\dW^*)]
\leq CK^{\kappa\ve}\dE_{\dGiQ^*}[(\eta^*)^2|\dW^\gap|^2].
\]
Since $\eta$ is supported on $\mathcal A_{\ve,2\gamma}$, \eqref{eq:nab tG} gives
\begin{equation}\label{eq:WDir}
\dE_{\dGiQ^*}[\eta^*\dW^*\cdot A_1^{-1}(\eta^*\dW^*)]
\leq CK^{\kappa\ve}(K+K^{2\bar\alpha}).
\end{equation}

Write $\eta^*=\widetilde\eta\circ\Gap_N^*$. The same argument bounds the last term of~\eqref{eq:etaLoop} by
\[
CK^{\kappa\ve}\dE_{\dGiQ^*}[(\Loop^*)^2|\nabla\widetilde\eta|^2\circ\Gap_N^*].
\]
The weak gradient of $\widetilde\eta$ is supported where $\eta\ne1$, within $\mathcal A_{\ve,2\gamma}$. On this support it is bounded by a fixed power of $K$, as is $\Loop^*$. The probability estimate~\eqref{eq:aux cutoff probability}, with an arbitrarily large power, and~\eqref{eq:TV a prio} therefore give
\begin{equation}\label{eq:bad}
\dE_{\dGiQ^*}[\Loop^*\nabla\eta^*\cdot A_1^{-1}(\Loop^*\nabla\eta^*)]
\leq CK^{-10}.
\end{equation}

\paragraph{\emph{Step 4: using the comparison principle for $\eta^*\dV^*$}}
For $i\ne1$, write $d_i=d(i,1)$ and set
\[
b_i=\frac{K^{\bar\alpha}}{(1+d_i)^{1+\bar\alpha}}\mathds{1}_{d_i\leq K^{1+\ve}}
+\frac{K^{2-s}}{(1+d_i)^{3-s}}\mathds{1}_{d_i\geq K^{1-\ve}}
+\frac{K^{\bar\alpha}}{(1+d_i)^{1+s/2}}.
\]
The good-event comparison between positions and labels and~\eqref{eq:V1sup} imply
\[
\eta^*|\dV_i^*|\leq CK^{\kappa\ve}\eta^*Nb_i,\qquad i\in I_\gamma\setminus\{1\}.
\]
Define, for $y\in D_N^*$,
\[
G_N^*(y)=\sum_{i\in I_\gamma\setminus\{1\}}b_i\bigl(Ny_i-(i-1)\bigr).
\]
Its nonzero partial derivatives are $Nb_i\geq0$. Lemma~\ref{lemma:energy comparisons} and sub-additivity consequently give
\begin{equation}\label{eq:V comparison}
\dE_{\dGiQ^*}[\eta^*\dV^*\cdot A_1^{-1}(\eta^*\dV^*)]
\leq CK^{\kappa\ve}\Var_{\dGiQ^*}[\eta^*G_N^*] 
+CK^{\kappa\ve}\dE_{\dGiQ^*}[G_N^*\nabla\eta^*\cdot A_1^{-1}(G_N^*\nabla\eta^*)].
\end{equation}
On $\mathcal A_{\ve,2\gamma}$, the gap estimate gives
$|Ny_i-(i-1)|\leq K^\ve d_i^{s/2}$. For labels beyond the opposite particle, this follows by using the complementary arc. Thus
\[
|\eta^*G_N^*|\leq K^\ve\sum_{i\in I_\gamma\setminus\{1\}}b_i d_i^{s/2}
\leq CK^{\kappa\ve}(K^{\bar\alpha}+K^{s/2}).
\]
Indeed, the first two sums are bounded by $CK^{\kappa\ve}(K^{\bar\alpha}+K^{s/2})$, and the last is at most $CK^{\bar\alpha}\log(2+|I_\gamma|)$. Since $|I_\gamma|\leq CK^\gamma$, all logarithms are absorbed by $K^{\kappa\ve}$. We obtain
\[
\Var_{\dGiQ^*}[\eta^*G_N^*]\leq CK^{\kappa\ve}(K^{2\bar\alpha}+K^s).
\]
The cutoff term in~\eqref{eq:V comparison} is at most $CK^{-10}$ by the argument for~\eqref{eq:bad}, using the same bound on $G_N^*$ on the support of $\nabla\eta^*$. Therefore,
\begin{equation}\label{eq:CCU}
\dE_{\dGiQ^*}[\eta^*\dV^*\cdot A_1^{-1}(\eta^*\dV^*)]
\leq CK^{\kappa\ve}(K+K^{2\bar\alpha}).
\end{equation}

\paragraph{\emph{Step 5: conclusion}}
Combining~\eqref{eq:etaLoop}, \eqref{eq:WDir}, \eqref{eq:bad} and~\eqref{eq:CCU} yields
\[
\Var_{\dGiQ^*}[\eta^*\Loop^*]\leq CK^{\kappa\ve}(K+K^{2\bar\alpha}).
\]
Equations~\eqref{eq:Loopcomp} and~\eqref{eq:aux step1} give the same bound for $\Var_{\dGi}[\Loop_{\ell_N}[\psi]]$, after increasing $\kappa$. Taking a smaller $\ve$ proves~\eqref{eq:Loopvariance}.
\end{proof}

\subsection{Proof of Theorem~\ref{theorem:quantitative variance}}

\begin{proof}
Fix $\rho$ as in Proposition~\ref{prop:auxiliary}. Proposition~\ref{prop:VarAl}, applied with $h_N=(N\ell_N)^{-1}$ and this $\rho$, gives
\[
(N\ell_N)^{2(s-1)}\Var_{\dGi}[\Loop_{\ell_N}[\psi_\reg]]\leq C(N\ell_N)^{\max(2s-1,2\alpha)+\ve}.
\]
Inserting this bound into~\eqref{eq:mainV} proves~\eqref{eq:variance expansion theorem}. The estimate for $\sigma_{\ell_N}^2(\xi)-\sigma_0^2(\xi)$ follows from~\eqref{eq:scaling frac} and the definitions of these variances. By Lemma~\ref{lemma:uniform smoothing} and the comparisons in the proof of Proposition~\ref{prop:auxiliary}, the same argument gives the upper bound on the variance in~\eqref{eq:variance expansion theorem} for $\xi*\rho_h$ and $\xi-\xi*\rho_h$, for all sufficiently small $h>0$, with a constant independent of $(\ell_N)$, $N$ and $h$.
\end{proof}

\section{Central limit theorem}\label{section:CLTs}

\subsection{Proof of the CLT}
We prove the CLT stated in Theorem~\ref{theorem:CLT} using Stein's method and the variance estimates of Section~\ref{section:optimal scaling}.

\begin{proof}[Proof of Theorem~\ref{theorem:CLT}]
Recall that $K=N\ell_N$, and let
\[
 F_N=K^{-s/2}\Fluct_N[\xi(\ell_N^{-1}\cdot)].
\]
If the limiting variance is zero, then $\xi$ is constant in the
macroscopic case of~\ref{item:A4} and zero in the mesoscopic case. The centered
statistic and its Gaussian comparison are then identically zero.
We therefore suppose that the limiting variance is positive.
Let $f\in\mathcal C^1(\dR)$ satisfy
$\Vert f\Vert_\infty\leq1$ and $\Vert f'\Vert_\infty\leq1$.

\paragraph{\emph{Step 1: regularization}}
We use a regularization scale larger than the microscopic scale.
Set
\[
 \delta_N=K^{-(2-s)/(4+s)},\qquad
 \xi_\reg=\xi*\rho_{\delta_N},\qquad
 G_N=K^{-s/2}\Fluct_N[\xi_\reg(\ell_N^{-1}\cdot)].
\]
All estimates below are for sufficiently large $K$. In particular,
$K^{-1}\leq\delta_N$, and in the mesoscopic case the regularized
test function remains supported in $(-1/2,1/2)$, because, by~\ref{item:A4}, the original support has a fixed positive distance from its endpoints.
The remaining bounded values of $K\geq1$ are covered by increasing the
constant in the final estimate.

Fix a sufficiently small $\tau>0$. The estimate~\eqref{eq:uniform smoothing bounds} of Lemma~\ref{lemma:uniform smoothing} gives
\begin{equation}\label{eq:coarse Sobolev error}
 \sigma_{\ell_N}(\xi-\xi_\reg)
 \leq C_\tau\delta_N^{s/2-\alpha-\tau}.
\end{equation}
The uniform upper bound on the variance at the end of the proof of Theorem~\ref{theorem:quantitative variance}, applied to $\xi-\xi*\rho_{\delta_N}$,
therefore gives
\begin{equation}\label{eq:compGNFN}
 \bigl|\dE_{\dGi}[f(F_N)-f(G_N)]\bigr|
 \leq\Vert F_N-G_N\Vert_{L^2(\dGi)} 
 \leq C_\tau\left(
   \delta_N^{s/2-\alpha-\tau}
   +K^{\max((s-1)/2,\alpha-s/2)+\tau}\right).
\end{equation}

\paragraph{\emph{Step 2: Stein's method}}
Set $\sigma_\reg=\sigma_{\ell_N}(\xi_\reg)$ and let
$Z\sim\mathcal N(0,\sigma_\reg^2)$. For a nonzero compactly supported profile, \eqref{eq:scaling frac} shows that $\sigma_\ell(\xi)$ is bounded for $\ell\in(0,1]$ and tends to $\sigma_0(\xi)>0$ as $\ell\downarrow0$. Away from $\ell=0$, it is positive and, by Fatou's lemma, lower semicontinuous in $\ell$, hence bounded below. In the macroscopic case it is a fixed positive number. Thus~\eqref{eq:coarse Sobolev error} gives constants $0<c<C<\infty$ such that $c\leq\sigma_\reg\leq C$ for all sufficiently large $K$.
Consider the Stein equation
\begin{equation}\label{eq:ste}
 \sigma_\reg^2u_f'(x)-xu_f(x)=f(x)-\dE[f(Z)].
\end{equation}
Its bounded solution is
\[
 u_f(x)=\frac1{\sigma_\reg^2}e^{x^2/(2\sigma_\reg^2)}
  \int_{-\infty}^x e^{-y^2/(2\sigma_\reg^2)}
       \bigl(f(y)-\dE[f(Z)]\bigr)\dd y.
\]
The rescaled function $v(y)=\sigma_\reg u_f(\sigma_\reg y)$ is the
bounded solution of the standard Stein equation
$v'(y)-yv(y)=h(y)-\dE[h(Z/\sigma_\reg)]$, where $h(y)=f(\sigma_\reg y)$
satisfies $\Vert h'\Vert_\infty\leq\sigma_\reg$.
By~\cite[Lemma~2.5]{ross2011fundamentals}, $\Vert v\Vert_\infty\leq2\Vert h'\Vert_\infty$
and $\Vert v'\Vert_\infty\leq\sqrt{2/\pi}\Vert h'\Vert_\infty$.
Since $u_f(x)=\sigma_\reg^{-1}v(x/\sigma_\reg)$, it follows that
\[
 \Vert u_f\Vert_\infty\leq2,\qquad
 \Vert u_f'\Vert_\infty\leq\frac{\sqrt{2/\pi}}{\sigma_\reg}.
\]
In particular, both Stein bounds are uniform in $N$.
By~\eqref{eq:ste},
\begin{equation}\label{eq:stein}
 \dE_{\dGi}[f(G_N)]-\dE[f(Z)]
 =\dE_{\dGi}[\sigma_\reg^2u_f'(G_N)-G_Nu_f(G_N)].
\end{equation}

Let $\psi_\reg=\psi*\rho_{\delta_N}$ be the mean-zero transport of $\xi_\reg$, given by~\eqref{def:psi reg} with $h_N=\delta_N$.
Define
\[
 \nabla\Phi=\frac{\ell_N}{K^{1-s/2}}
       \bigl(\psi_\reg(\ell_N^{-1}x_1),\ldots,
             \psi_\reg(\ell_N^{-1}x_N)\bigr).
\]
The zero mean of $\psi_\reg$ gives a periodic primitive, so $\Phi$ is
well defined. The transport identity~\eqref{eq:LP} gives
\[
 \mc L\Phi-G_N
 =K^{-1+s/2}\left(\beta\Loop_{\ell_N}[\psi_\reg]
                -\Fluct_N[\psi_\reg'(\ell_N^{-1}\cdot)]\right).
\]
Integration by parts consequently gives
\begin{equation}\label{eq:ipp tG CLT}
 \dE_{\dGi}[u_f(G_N)G_N]
 =\sigma_\reg^2\dE_{\dGi}[u_f'(G_N)]
       +\Error_N^1+\Error_N^2,
\end{equation}
where
\[
 \begin{split}
 \Error_N^1
 &=\dE_{\dGi}\left[u_f'(G_N)
       \left(\frac1K\sum_{i=1}^N
       \psi_\reg(\ell_N^{-1}x_i)\xi_\reg'(\ell_N^{-1}x_i)
                       -\sigma_\reg^2\right)\right],\\
 \Error_N^2
 &=-K^{-1+s/2}\Cov_{\dGi}
       \left[u_f(G_N),\beta\Loop_{\ell_N}[\psi_\reg]
                -\Fluct_N[\psi_\reg'(\ell_N^{-1}\cdot)]\right].
 \end{split}
\]

\paragraph{\emph{Step 3: the error term $\Error_N^1$}}
Write $q=\psi_\reg\xi_\reg'$. Rotation invariance and
\eqref{eq:mf var} give the exact centering
\[
 \dE_{\dGi}\left[\sum_{i=1}^Nq(\ell_N^{-1}x_i)\right]
       =K\sigma_\reg^2.
\]
We estimate the regularized product before applying the a priori
variance bound. The derivative estimates~\eqref{eq:uniform derivative bounds} imply
\[
 \Vert\xi_\reg'\Vert_\infty
       \leq C_\tau\delta_N^{-1-\alpha-\tau},\qquad
 \Vert\xi_\reg''\Vert_\infty
       \leq C_\tau\delta_N^{-2-\alpha-\tau}.
\]
The transport bounds give
\[
 \Vert\psi_\reg\Vert_\infty\leq C,\qquad
 \Vert\psi_\reg'\Vert_\infty
       \leq C_\tau\delta_N^{-\bar\alpha-\tau}.
\]
Since $\alpha<s$, the identity
$q'=\psi_\reg'\xi_\reg'+\psi_\reg\xi_\reg''$ yields, after decreasing
the auxiliary parameter $\tau$ if necessary,
\[
 \Vert q\Vert_\infty+\Vert q'\Vert_\infty
       \leq C_\tau\delta_N^{-2-\alpha-\tau}.
\]
In the mesoscopic case $q$ is supported in a fixed bounded
interval. Lemma~\ref{lemma:a priori}, applied with exponent zero, together with the explicit
dependence of its bound on the derivative bound, gives
\[
 \Var_{\dGi}[\Fluct_N[q(\ell_N^{-1}\cdot)]]
       \leq C_\tau K^{s+\tau}
                      \delta_N^{-2(2+\alpha+\tau)}.
\]
The Cauchy--Schwarz inequality and the uniform Stein bound
therefore give
\begin{equation}\label{eq:errorN1 CLT}
 |\Error_N^1|
       \leq C_\tau K^{s/2-1+\tau}
                      \delta_N^{-2-\alpha-\tau}.
\end{equation}

\paragraph{\emph{Step 4: the error term $\Error_N^2$}}
Because $\delta_N\geq K^{-1}$, the transport $\psi_\reg$ satisfies
the derivative envelopes used in the loop variance estimate
with constants independent of $\delta_N$. Apply~\eqref{eq:non anch}
with $\lambda=\bar\alpha$ to $\psi_\reg'$ localized by fixed smooth
cutoffs with disjoint supports around the marks chosen in the proof of
Lemma~\ref{lemma:reg micro}, and with $\lambda=0$ to the remainder.
The cutoff terms are uniformly bounded and the derivative of the
remainder decays like $|x|^{s-3}$, so the same envelopes verify its
hypotheses. At $\bar\alpha=0$, increase $\alpha$ slightly and absorb
the loss into $K^\tau$. Together with~\eqref{eq:Loopvariance}, this gives
\[
 \Var_{\dGi}[\Loop_{\ell_N}[\psi_\reg]]
       +\Var_{\dGi}[\Fluct_N[\psi_\reg'(\ell_N^{-1}\cdot)]]
 \leq C_\tau K^\tau
             \left(K+K^{2\bar\alpha}\right).
\]
The covariance inequality and $\Vert u_f\Vert_\infty\leq2$
give
\begin{equation}\label{eq:errorN2 CLT}
 |\Error_N^2|
       \leq C_\tau
             K^{\max((s-1)/2,\alpha-s/2)+\tau}.
\end{equation}

\paragraph{\emph{Step 5: conclusion}}
Recall that $r=(2-s)(s-2\alpha)/(2(4+s))$. The choice of $\delta_N$
gives
\[
 \delta_N^{s/2-\alpha}=K^{-r},\qquad
 K^{s/2-1}\delta_N^{-2-\alpha}=K^{-r}.
\]
Combining~\eqref{eq:stein}, \eqref{eq:ipp tG CLT},
\eqref{eq:errorN1 CLT} and~\eqref{eq:errorN2 CLT}, and choosing
$\tau$ sufficiently small in terms of a prescribed $\ve>0$,
we obtain
\begin{equation}\label{eq:concGN}
 \bigl|\dE_{\dGi}[f(G_N)]-\dE[f(Z)]\bigr|
       \leq C_\ve K^{\ve+
                     \max(-r,(s-1)/2,\alpha-s/2)}.
\end{equation}
Let $Z_N\sim\mathcal N(0,\sigma_{\ell_N}^2(\xi))$, as in Theorem~\ref{theorem:CLT}.
Coupling $Z$ and $Z_N$ by the same standard Gaussian and using
the triangle inequality for the Sobolev seminorm gives
\begin{equation}\label{eq:comZZ'}
 \bigl|\dE[f(Z)]-\dE[f(Z_N)]\bigr|
       \leq\bigl|\sigma_{\ell_N}(\xi_\reg)
                          -\sigma_{\ell_N}(\xi)\bigr|
       \leq\sigma_{\ell_N}(\xi_\reg-\xi)
       \leq C_\tau\delta_N^{s/2-\alpha-\tau}.
\end{equation}
Since $(2-s)/(4+s)<1$, we have $\alpha-s/2<-r$.
The triangle inequality, \eqref{eq:compGNFN},
\eqref{eq:concGN} and~\eqref{eq:comZZ'} prove
\[
 \bigl|\dE_{\dGi}[f(F_N)]-\dE[f(Z_N)]\bigr|
       \leq C_\ve K^{\ve+
                       \max((s-1)/2,-r)}.
\]
Taking the supremum over $f$ proves the quantitative assertion.
The exponents are strictly negative before adding $\ve$.
The convergence of $\sigma_{\ell_N}^2(\xi)$ to
$\sigma_\ell^2(\xi)$ therefore also proves the asserted
convergence in distribution.
\end{proof}

\subsection{Proof of Corollary~\ref{corollary:fluctuations gaps disc}}

\begin{proof}[Proof of Corollary~\ref{corollary:fluctuations gaps disc}]
We prove the count CLT, invert the counting relation, and remove the
bias in the selected label.
For $\xi=\mathds{1}_{(-a,a)}$, Lemma~\ref{lemma:explicit formulas} gives
\[
\sigma_{\ell_N}^2(\xi)=\sigma^2\ell_N^{-s}V_s(2a\ell_N),
\qquad \sigma_0^2(\xi)=\sigma^2(2a)^s.
\]
We multiply the two random variables compared in part~2 of Theorem~\ref{theorem:CLT} by $\ell_N^{s/2}V_s(2a\ell_N)^{-1/2}$, which is bounded. This yields
\begin{equation}\label{eq:zetaNs}
\frac{\Fluct_N[\mathds{1}_{(-a,a)}(\ell_N^{-1}\cdot)]}
{\sqrt{N^sV_s(2a\ell_N)}}
\underset{\mathrm{Law}}{\Longrightarrow}\mathcal N(0,\sigma^2)
\end{equation}
when $\ell_N\to0$ or $\ell_N\equiv1$.

On a subsequence with $\ell_N\to\ell>0$, put
$t_N=a|\ell_N-\ell|$. The two boundary strips have width $t_N$.
Theorem~\ref{theorem:quantitative variance} at scale $2t_N$ when
$Nt_N\geq1$, and the mean count otherwise, give
\[
N^{-s/2}\left\Vert\Fluct_N\left[
\mathds{1}_{(-a\ell_N,a\ell_N)}-\mathds{1}_{(-a\ell,a\ell)}
\right]\right\Vert_{L^1(\dGi)}
\leq C\bigl(N^{-s/2}+t_N^{s/2}\bigr)\longrightarrow0.
\]
The fixed-interval CLT and continuity of $V_s$ prove~\eqref{eq:zetaNs}
on this subsequence. Subsequence extraction gives the general case.

Write $k=k_N$, use cyclic lifts $x_{i+N}=x_i+1$, and set
\[
S_{i,r}=N(x_{i+r}-x_i),\qquad
 b_N=\sqrt{N^sV_s(k/N)}\asymp k^{s/2}.
\]
For sufficiently small $0<\ve\leq s/4$, Theorem~\ref{theorem:almost optimal rigidity},
\eqref{eq:gap upper exponential} and Cauchy--Schwarz give
\begin{equation}\label{eq:gap CLT moment}
\Vert S_{i,r}-r\Vert_{L^2(\dGi)}^2
\leq r^{s+2\ve}+Cr^2e^{-cr^{\delta(\ve)}/2}
\leq Cr^{3s/2},\qquad 1\leq r\leq N/2.
\end{equation}

Let $x_J$ be the lifted particle nearest to $0$, and put
\[
w_i=\frac N2(x_{i+1}-x_{i-1}).
\]
The rotation interval selecting $i$ has length $w_i/N$.
Uniform rotation and cyclic symmetry give
\begin{equation}\label{eq:nearest particle weight}
\dE_{\dGi}[F_J]=\dE_{\dGi}[w_1F_1]
\end{equation}
for cyclic relabelings $F_i$ of a bounded or nonnegative gap functional.
Since $N|x_J|\leq w_J$, \eqref{eq:gap upper exponential} implies
\[
\Vert w_1\Vert_{L^2(\dGi)}\leq C,\qquad
\dE_{\dGi}[N|x_J|]\leq\dE_{\dGi}[w_1^2]\leq C.
\]

Set $\Delta_N=\#\{i\in\mathbb Z:x_i\in(x_J,x_J+k/N]\}-k$.
On $N|x_J|\leq R$, shifting to $(0,k/N]$ changes the count only within
two fixed arcs of total length at most $4R/N$. Hence, for $R,\eta>0$,
\[
\dGi\left(\left|\Delta_N-\Fluct_N[\mathds{1}_{(0,k/N]}]\right|>\eta b_N\right)
\leq \frac CR+\frac{4R}{\eta b_N}.
\]
We let $N\to\infty$, then $R\to\infty$.
Rotation invariance and~\eqref{eq:zetaNs} with $a=1/4$ and
$\ell_N=2k/N$ give
\begin{equation}\label{eq:Deli0}
\frac{\Delta_N}{b_N}
\underset{\mathrm{Law}}{\Longrightarrow}\mathcal N(0,\sigma^2).
\end{equation}

For deterministic integers $r=O(b_N)$, set
$T_r=S_{J,k+r}-S_{J,k}-r$.
Equations~\eqref{eq:nearest particle weight}, \eqref{eq:gap CLT moment}
and Cauchy--Schwarz give
\[
\dE_{\dGi}[|T_r|]\leq C|r|^{3s/4},\qquad
T_r/b_N\longrightarrow0\quad\text{in probability}.
\]
Since $b_N=o(k)$ and $k\leq N/2$, we have $1\leq k+r\leq N-1$ for large $N$,
and
\[
\{\Delta_N\geq r\}
=\{S_{J,k+r}\leq k\}
=\{S_{J,k}-k+r+T_r\leq0\}.
\]
For $u\in\dR$ and $\eta>0$, this gives
\begin{equation*}
\dGi\bigl(\Delta_N\geq\lceil(-u+\eta)b_N\rceil\bigr)-o(1)
\leq\dGi\bigl(S_{J,k}-k\leq ub_N\bigr) 
\leq\dGi\bigl(\Delta_N\geq\lceil(-u-\eta)b_N\rceil\bigr)+o(1).
\end{equation*}
We let $N\to\infty$ and then $\eta\downarrow0$.
Equation~\eqref{eq:Deli0} and Gaussian symmetry give
\begin{equation}\label{eq:rand}
\frac{S_{J,k}-k}{b_N}
\underset{\mathrm{Law}}{\Longrightarrow}\mathcal N(0,\sigma^2).
\end{equation}

Finally, put $X_i=(S_{i,k}-k)/b_N$ and $m=\lfloor\sqrt{k}\rfloor$.
The identities
\[
\sum_{j=1}^m w_j=\frac{S_{0,m}+S_{1,m}}2,\qquad
S_{j,k}-S_{1,k}=S_{k+1,j-1}-S_{1,j-1}
\]
and~\eqref{eq:gap CLT moment} give, by cyclic symmetry,
\begin{align*}
\bigl|\dE_{\dGi}[w_1f(X_1)]-\dE_{\dGi}[f(X_1)]\bigr|
&\leq\dE_{\dGi}\left[\left|\frac1m\sum_{j=1}^m w_j-1\right|\right]
+\frac Cm\sum_{j=1}^m\Vert X_j-X_1\Vert_{L^2(\dGi)}\\
&\leq C\left(m^{3s/4-1}+\frac{m^{3s/4}}{k^{s/2}}\right)
\longrightarrow0
\end{align*}
for every $f$ with $\Vert f\Vert_\infty\leq1$ and Lipschitz constant at most $1$.
Together with~\eqref{eq:nearest particle weight} and~\eqref{eq:rand},
this proves the gap CLT at label $1$.
\end{proof}

\subsection{Proof of Corollary~\ref{corollary:power law}}

\begin{proof}[Proof of Corollary~\ref{corollary:power law}]
The end of the proof of Lemma~\ref{lemma:explicit formulas} verifies Assumption~\ref{assumption:test func} for $\xi$, with $\ell_N\equiv1$ and the single marked point $a_1=0$. Since $|\psi''(x)|\leq C|x|^{-1-\bar\alpha}$ near $0$, the bound~\eqref{eqdef:singularity psi''} holds with the exponent $\alpha$. By~\eqref{def:sigmaN}, the second assertion of Lemma~\ref{lemma:explicit formulas} gives
\[
\sigma_1^2(\xi)=(2\pi)^{-1+2\alpha-s}\frac{c_\alpha^2}{\beta c_s}\zeta(1+s-2\alpha).
\]
Theorem~\ref{theorem:CLT} proves the corollary.
\end{proof}

\appendix

\section{Well-posedness of the Helffer--Sjöstrand equation}\label{section:existence}

\subsection{Poincaré inequality}

\begin{lemma}
Let $\mu$ satisfy Assumption~\ref{assumption:gibbs measure}. For every $F\in H^1(\mu)$,
\begin{equation}\label{eq:P0}
\Var_\mu[F]\leq C_P\dE_\mu[|\nabla F|^2],\qquad
C_P=\max\left\{\frac{N}{4\pi^2},\frac{1}{2cN}\right\}.
\end{equation}
\end{lemma}

\begin{proof}
We use the orthogonal coordinates $(z,\omega)$ from the proof of Lemma~\ref{lemma:brascamp lieb inequality}.
The law is a product: $\omega$ is uniform on a circle of length $\sqrt N$,
and the potential in $z$ has curvature at least $2cN$.
Conditioning on $z$ and using the circle Poincaré inequality,
\eqref{eq:euclidean BL} and Jensen's inequality, we obtain
\[
\Var_\mu[F]
\leq\frac{N}{4\pi^2}\dE_\mu[|\partial_\omega F|^2]
+\frac1{2cN}\dE_\mu[|\nabla_zF|^2]
\leq C_P\dE_\mu[|\nabla F|^2].
\]
\end{proof}

\subsection{The collision boundary}

By the Riesz representation theorem~\cite[Theorem~5.5]{brezis2011functional}, for every $g\in L^2(\mu)$ there is a unique $Rg\in H^1(\mu)$ such that
\[
\dE_\mu[\nabla(Rg)\cdot\nabla h+(Rg)h]=\dE_\mu[gh],\qquad h\in H^1(\mu).
\]
The Neumann realization of $\mc L^\mu$ has domain $\operatorname{Ran}R$ and satisfies $\mc L^\mu Rg=g-Rg$. It is symmetric and nonnegative, and $1+\mc L^\mu$ is onto, so it is self-adjoint. Let $L_0$ be the restriction of $\mc L^\mu$ to smooth functions supported away from the collision boundary and periodic in the common rotation. Integration by parts shows that the Neumann realization extends $L_0$.

We prove the operator core property using~\eqref{eq:boundary growth}.
Write $d_\partial(X_N)=\operatorname{dist}(X_N,\partial D_N)$.
Multiplication by $e^{-H/2}$ transforms $L_0$ into
\[
S=-\Delta+W,\qquad
W=\frac{|\nabla H|^2}{4}-\frac{\Delta H}{2}\geq d_\partial^{-2}-C_N.
\]

Let $\lambda>C_N+1$ and $u\in\ker(S^*+\lambda)$.
Set $d_\ve=(d_\partial-\ve)^+$.
Local elliptic regularity permits the Lipschitz test $d_\ve^2u$ and gives
\[
\int|\nabla(d_\ve u)|^2+\int(W+\lambda)|d_\ve u|^2
=\int|\nabla d_\ve|^2|u|^2.
\]
Discarding the first term, and using $|\nabla d_\ve|\leq1$ and the lower bound on $W$, we obtain
\[
\int\Bigl[\Bigl(\frac{d_\ve}{d_\partial}\Bigr)^2+(\lambda-C_N)d_\ve^2\Bigr]|u|^2\leq\int|u|^2.
\]
As $\ve\to0$, dominated convergence gives $(\lambda-C_N)\int d_\partial^2|u|^2\leq0$.
Thus $u=0$, so $S+\lambda$, and hence $L_0+\lambda$, has dense range. Since $L_0$ is nonnegative, $\Vert(L_0+\lambda)v\Vert\geq\lambda\Vert v\Vert$, so the closure of $L_0+\lambda$ has closed, hence full, range. The operator $\mc L^\mu+\lambda$ extends this closure and is injective, so the two coincide. Hence the closure of $L_0$ is the Neumann realization $\mc L^\mu$.

If $h\in H^1(\mu)$ is orthogonal to $\operatorname{Ran}R$ in $H^1(\mu)$, then $\dE_\mu[gh]=0$ for every $g\in L^2(\mu)$, so $h=0$. Thus $\operatorname{Ran}R$ is dense in $H^1(\mu)$. Since convergence in the graph norm implies convergence in the form norm, the domain of $L_0$ is an operator core and a form core for the Neumann realization of $\mc L^\mu$.
Equation~\eqref{eq:IPPLmu} extends through the closed form.
Since $\nabla^2H(1,\ldots,1)=0$, scalar form approximation also gives
\[
h(1,\ldots,1)\in\mathcal V_\mu,\qquad h\in H^1(\mu).
\]

For the anchored potential, $\nabla H(0,y)=\bigl(-\sum_{i=2}^N\partial_iH^*(y),\nabla H^*(y)\bigr)$, and convexity gives $\Delta H^*\leq\Delta H$.
The anchored distance to the boundary is at least $d_\partial$, so
\[
\frac{|\nabla H^*|^2}{4}-\frac{\Delta H^*}{2}
\geq\frac{|\nabla H|^2}{4N}-\frac{\Delta H}{2}
\geq\operatorname{dist}(y,\partial D_N^*)^{-2}-C_N.
\]
The same argument therefore gives the anchored operator and form cores.

It remains to verify~\eqref{eq:boundary growth} for the Hamiltonians used
here. For $H=\beta\mathcal H_N$, choose an equally spaced lift $Y_N$ with
bounded $|X_N-Y_N|$. Then
\begin{equation*}
\nabla H(X_N)\cdot(X_N-Y_N)
=2\beta N^{-s}\sum_{i<j}g'(x_j-x_i)
       \bigl(x_j-x_i-(y_j-y_i)\bigr) 
\geq c_Nd_\partial^{-s-1}-C_N.
\end{equation*}
Each summand is bounded below, and a pair closest to collision supplies
the singular term. Summing the bounds on $g''$ gives, for small $d_\partial$,
\[
|\nabla H|\geq c_Nd_\partial^{-s-1},\qquad \Delta H\leq C_Nd_\partial^{-s-2}.
\]
Since $s>0$, after enlarging $C_N$ on the remaining compact quotient,
\[
\frac{|\nabla H|^2}{4N}-\frac{\Delta H}{2}
\geq c_Nd_\partial^{-2s-2}-C_Nd_\partial^{-s-2}
\geq d_\partial^{-2}-C_N.
\]
Smooth convex perturbations with bounded first and second derivatives
preserve these estimates, covering the penalties used in the paper.

\subsection{Well-posedness for gradient vector fields}

\begin{lemma}\label{lemma:exist gap func}
Let $\mu$ satisfy Assumption~\ref{assumption:gibbs measure} and let $F\in L^2(\mu)$. There exists a unique, modulo an additive constant, $\phi\in H^1(\mu)$ such that
\begin{equation}\label{eq:preli existence unique}
\mc L^\mu\phi=F-\dE_\mu[F].
\end{equation}
It is the unique, modulo an additive constant, minimizer of
\[
\phi\mapsto\dE_\mu[|\nabla\phi|^2-2\phi(F-\dE_\mu[F])].
\]
\end{lemma}

\begin{proof}
On the mean-zero subspace,~\eqref{eq:P0} makes the gradient norm equivalent to the $H^1(\mu)$ norm. The Riesz representation theorem~\cite[Theorem~5.5]{brezis2011functional}, applied to $h\mapsto\dE_\mu[(F-\dE_\mu[F])h]$, gives a unique mean-zero $\phi$ such that
\[
\dE_\mu[\nabla\phi\cdot\nabla h]
=\dE_\mu[(F-\dE_\mu[F])h],\qquad h\in H^1(\mu).
\]
This is~\eqref{eq:preli existence unique}. The functional at $h$ equals its value at $\phi$ plus $\dE_\mu[|\nabla(h-\phi)|^2]$, so by~\eqref{eq:P0} its minimizers are $\phi$ plus constants.
\end{proof}

\begin{proof}[Proof of Proposition~\ref{proposition:existence HS equation}]
We approximate in the operator domain.
For $F\in H^1(\mu)$, let $\phi$ be the mean-zero solution from
Lemma~\ref{lemma:exist gap func}. The operator core gives smooth $\phi_k$ supported away from the boundary with
\[
\phi_k\longrightarrow\phi,\qquad
\mc L^\mu\phi_k\longrightarrow F-\dE_\mu[F]\quad\text{in }L^2(\mu).
\]
For smooth $f$ with the same support property, the commutator and
integration by parts give
\begin{align*}
[\partial_i,\mc L^\mu]f
&=\sum_j(\partial_{ij}H)\partial_jf,\\
\|\nabla f\|_2^2&=\dE_\mu[f\mc L^\mu f],\qquad
\mathfrak q_\mu(\nabla f,\nabla f)=\|\mc L^\mu f\|_2^2.
\end{align*}
All norms here are in $L^2(\mu)$.
Applying these identities to $\phi_k-\phi_\ell$ gives
$\nabla\phi_k\to\nabla\phi$ in $\mathcal V_\mu$.

For smooth vector fields $h$ supported away from the boundary,
\[
\mathfrak q_\mu(\nabla\phi_k,h)
=\dE_\mu\left[\mc L^\mu\phi_k
 \sum_i\bigl(-\partial_i h_i+(\partial_iH)h_i\bigr)\right]
\longrightarrow\dE_\mu[\nabla F\cdot h].
\]
The limit uses the weak derivatives of $F$.
Density extends the identity to every $h\in\mathcal V_\mu$.
Periodicity gives $\dE_\mu[\sum_i\partial_i\phi]=0$.
Zero $\mathfrak q_\mu$-energy forces a vector field to be a constant
multiple of $(1,\ldots,1)$. The normalization gives uniqueness, and
expanding the quadratic functional proves~\eqref{eq:variatt}.

Testing the scalar equation with $G-\dE_\mu[G]$ gives
\[
\Cov_\mu[F,G]=\dE_\mu[\nabla\phi\cdot\nabla G],
\]
and~\eqref{eq:representation variance} follows for $G=F$.
If $F$ is rotation-invariant, uniqueness makes $\phi$ rotation-invariant.

For $\mu^*$, \eqref{eq:euclidean BL} and Remark~\ref{remark:strict conv}
give the scalar Poincaré inequality. The anchored operator core and the
same commutator argument give the anchored representations.
\end{proof}

\subsection{Well-posedness for non-gradient vector fields}

\begin{proof}[Proof of Proposition~\ref{prop:well posed non grad}]
We prove coercivity on the normalized subspace.
Let $e=N^{-1/2}(1,\ldots,1)$ and $f=u\cdot e$. The normalization gives $\dE_\mu[f]=0$.
Applying~\eqref{eq:P0} to each coordinate gives $\dE_\mu[|u-\dE_\mu[u]|^2]\leq C_P\dE_\mu[\sum_i|\nabla u_i|^2]$. Since $\sum_i\dE_\mu[u_i]=0$, Jensen's inequality and $\chi_{ij}''\geq c$ give
\[
2cN|\dE_\mu[u]|^2=2c\sum_{i<j}(\dE_\mu[u_j]-\dE_\mu[u_i])^2\leq\dE_\mu[u\cdot\nabla^2Hu].
\]
Adding these bounds and using $(2cN)^{-1}\leq C_P$, we obtain $\dE_\mu[|u|^2]\leq C_P\mathfrak q_\mu(u,u)$. Thus $\mathfrak q_\mu$ defines an equivalent Hilbert norm on the normalized subspace, and the Riesz representation theorem~\cite[Theorem~5.5]{brezis2011functional} gives a unique normalized $u$ satisfying~\eqref{eq:pose me} for normalized tests.
Since $e\in\mathcal V_\mu$ and
$\mathfrak q_\mu(u,e)=\dE_\mu[v\cdot e]=0$, the identity extends to all tests. For a normalized competitor $h$, the functional exceeds its value at $u$ by $\mathfrak q_\mu(h-u,h-u)$, so $u$ is its unique minimizer.

If $v$ is rotation-invariant, uniqueness gives the same property for $u$.
If also $v\cdot e=0$ almost everywhere, the admissible test $fe$ gives
\[
\dE_\mu[|\nabla f|^2]
=\mathfrak q_\mu(u,fe)=\dE_\mu[f\,v\cdot e]=0.
\]
Equation~\eqref{eq:P0} and $\dE_\mu[f]=0$ imply $f=0$.
\end{proof}

\section{Auxiliary estimates}\label{section:auxiliary}

\subsection{Proof of Lemma~\ref{lemma:gradient splitting}}\label{section:app weighted}

\begin{proof}[Proof of Lemma~\ref{lemma:gradient splitting}]
We separate the displacement errors from the lattice quadrature.
Recall that $K=N\ell_N$ and $\bar\alpha=1-s+\alpha\in[0,1-s/2)$.
Use cyclic lifts and label the anchor by $0$.
If $t_\gamma<1/2$, set $n_\gamma=\lfloor Nt_\gamma\rfloor$ and use
$k_{ij}=j-i$ for $-n_\gamma\leq i,j\leq n_\gamma$.
On the full circle, choose a shorter index path for each unordered pair,
fixing one choice for opposite labels. Reverse orientations use opposite
lifts. Write $k_{ij}$ for the directed index increment and $r_{ij}$ for
the corresponding separation.
On $\mathcal A_{\ve,\gamma}$ with $|Nx_0|\leq K^\ve$,
\[
 |Nx_i-i|\leq CK^\ve(1+|i|)^{s/2},
 \qquad N(x_{i+1}-x_i)\geq K^{-\ve}\quad(i,i+1\in I_\gamma).
\]
Thus $|r_{ij}|\leq3/4$ for large $K$. In a partial window,
the particles lie in $(-2t_\gamma,2t_\gamma)$.

\paragraph{\emph{Step 1: the decomposition}}
For the chosen lifts, set
\[
 \tilde g(r)=r g'(r),\qquad
 \zeta_0(x,y)=
 \frac{\ell_N(\psi(y/\ell_N)-\psi(x/\ell_N))}{y-x}.
\]
We extend $\psi$ periodically and $\zeta_0$ smoothly at $x=y$.
Direct differentiation gives
\[
 \partial_x\zeta_0=\frac{\zeta_0-\psi'(x/\ell_N)}{y-x},
 \qquad
 \partial_y\zeta_0=\frac{\psi'(y/\ell_N)-\zeta_0}{y-x},
 \qquad
 \partial_{xy}\zeta_0
 =\frac{\partial_y\zeta_0-\partial_x\zeta_0}{y-x}.
\]
Periodicity and the bound on $\psi'$ imply
\[
 |\psi(u)-\psi((2\ell_N)^{-1})|\leq C|u|^{s-1},\qquad |u|\geq1.
\]
The subtracted constant vanishes in a partial window.
Using this cancellation on arcs crossing the origin, $\bar\alpha<1$, and the
derivative bounds elsewhere, we obtain
\begin{align*}
 |\zeta_0(x,y)|&\leq C(\wo(x/\ell_N)+\wo(y/\ell_N)),\\
 |\partial_x\zeta_0(x,y)|+|\partial_y\zeta_0(x,y)|
 &\leq C\ell_N^{-1}(\wt(x/\ell_N)+\wt(y/\ell_N)),\\
 |\partial_{xy}\zeta_0(x,y)|
 &\leq \frac{C\ell_N^{-1}}{|y-x|}
             (\wt(x/\ell_N)+\wt(y/\ell_N)).
\end{align*}
Weights use circular distance to the origin. When both endpoints are
within $N^{-1}$ of the origin, we use the bound on $\psi''$ directly.
For $z=y+h$, the same identities give
\begin{align*}
 |\zeta_0(x,z)-\zeta_0(x,y)|
 &\leq C\ell_N^{-1}|h|
       (\wt(x/\ell_N)+\wt(y/\ell_N)+\wt(z/\ell_N)),\\
 |\partial_x\zeta_0(x,z)-\partial_x\zeta_0(x,y)|
 &\leq
 \frac{C\ell_N^{-1}|h|}{\min(|y-x|,|z-x|)}
       (\wt(x/\ell_N)+\wt(y/\ell_N)+\wt(z/\ell_N)).
\end{align*}
By the identity for $\partial_x\zeta_0$ and its bound, $|\zeta_0(u,v)-\psi'(u/\ell_N)|\leq C\ell_N^{-1}|v-u|(\wt(u/\ell_N)+\wt(v/\ell_N))$. If $x$ lies between $y$ and $z$, then this bound at $(x,y)$ and $(x,z)$ gives the first estimate. If $y$ lies between $x$ and $z$, then we use it at $y$ together with $(z-x)(\zeta_0(x,z)-\zeta_0(x,y))=(z-y)(\zeta_0(y,z)-\zeta_0(x,y))$, and the case of $z$ between $x$ and $y$ is symmetric. The second estimate follows from the first and the identity and bound for $\partial_x\zeta_0$.

Let $B_\gamma=[-t_\gamma,t_\gamma]$, with $B_\gamma=\dT$ when
$t_\gamma=1/2$. Full-circle integrals use $y\in(x_i-1/2,x_i+1/2)$.
Particle sums in Steps 1--2 run over $j\in I_\gamma\setminus\{i\}$. Define
\begin{align*}
 \dV_i={}&2N^{-s}\sum_{j\in I_\gamma}
     \partial_1\zeta_0(x_i,x_j)\tilde g(r_{ij})
 -2N^{1-s}\operatorname{PV}\int_{B_\gamma}
     \partial_1\zeta_0(x_i,y)\tilde g(y-x_i)\dd y\\
 &-2N^{-s}\sum_{j\in I_\gamma}
     \zeta_0(x_i,x_j)\tilde g'(k_{ij}/N)
 +2N^{1-s}\operatorname{PV}\int_{B_\gamma}
     \zeta_0(x_i,y)\tilde g'(y-x_i)\dd y,\\
 \dW_i={}&-2N^{-s}\sum_{j\in I_\gamma}\zeta_0(x_i,x_j)
       \bigl(\tilde g'(r_{ij})-\tilde g'(k_{ij}/N)\bigr).
\end{align*}
We set both fields to zero outside $I_\gamma$. Differentiation gives
\[
 \nabla\Loop_{\ell_N,\gamma}[\psi]=\dV+\dW.
\]
Both principal values remove $|y-x_i-m|<\delta$ for every integer $m$, with a common $\delta\downarrow0$. With $v=y-x_i$, we use $\tilde g(v)=vg'(v)$ and integrate by parts in the truncated integrals, once against $g'$ and twice against $g''$. Taylor expansion and periodicity bound the paired boundary terms at the cuts by $O(\delta^{1-s})$. The remaining integrands are $g$ times smooth coefficients. They are integrable since $s<1$, so both principal values exist. To differentiate the background, we integrate by parts once and complete the period by periodicity and the support condition. This leaves an endpoint term and
\[
-\int_{-1/2}^{1/2}\psi'((x_i+v)/\ell_N)g(v)\dd v.
\]
The endpoint term vanishes on the full circle and, by the support condition, is smooth in a partial window. Differentiation in $x_i$ under the integral sign and the same integration-by-parts formulas give the difference of the two principal values. On $\mathcal A_{\ve,\gamma}$ with $|Nx_0|\leq K^\ve$ and large $K$, only the image $m=0$ meets $B_\gamma-x_i$.

\paragraph{\emph{Step 2: the bound on $\dV$}}
Put $z_{ij}=x_i+k_{ij}/N$. Let $\dV_i^{(2)}$ be $\dV_i$ with
$x_j$ and $r_{ij}$ replaced by $z_{ij}$ and $k_{ij}/N$ in the discrete
terms. Then $\dV_i^{(1)}=\dV_i-\dV_i^{(2)}$ equals
\begin{align*}
 \dV_i^{(1)}
 ={}&2N^{-s}\sum_j\partial_1\zeta_0(x_i,x_j)
                 \bigl(\tilde g(r_{ij})-\tilde g(k_{ij}/N)\bigr)\\
 &+2N^{-s}\sum_j
       \bigl(\partial_1\zeta_0(x_i,x_j)-\partial_1\zeta_0(x_i,z_{ij})\bigr)
                     \tilde g(k_{ij}/N)\\
 &-2N^{-s}\sum_j
       \bigl(\zeta_0(x_i,x_j)-\zeta_0(x_i,z_{ij})\bigr)
                     \tilde g'(k_{ij}/N).
\end{align*}
The finite-difference bounds, the good event, and
$|\tilde g^{(m)}(r)|\leq C|r|^{-s-m}$ for $m=0,1,2$ give
\[
 |\dV_i^{(1)}|\leq C K^{\kappa\ve}\ell_N^{-1}
 \sum_{j\ne i}
 \frac{\wt(x_i/\ell_N)+\wt(x_j/\ell_N)+\wt(z_{ij}/\ell_N)}
      {(1+d(i,j))^{1+s/2}}.
\]
Uniformly over translated lattices,
\begin{equation}\label{eq:app wt convolution}
 \sum_{k\in\mathbb Z/N\mathbb Z}\frac{\wt((z+k)/K)}{(1+|k|)^{1+s/2}}
 \leq C(1+\log K)
 \left(\wt(z/K)+\frac{K^{1+\bar\alpha}}{(1+|z|)^{1+s/2}}\right),\qquad z\in N\dT,
\end{equation}
where $|k|$ and $|z|$ are circular distances. Splitting at half the
distance $|z|$ reduces this to
\[
 \sup_{z\in N\dT}\sum_{k\in\mathbb Z/N\mathbb Z}\wt((z+k)/K)
 \leq CK^{1+\bar\alpha}(1+\log K).
\]
On the good event, microscopic unit intervals contain at most $CK^\ve$
particles of $I_\gamma$, and
\[
 K^{-\ve}d(i,j)\leq N\operatorname{dist}(x_i,x_j)\leq K^\ve d(i,j).
\]
Partitioning into these intervals and using~\eqref{eq:app wt convolution}
for both particle and reference positions gives
\begin{equation}\label{eq:app V reference}
 |\dV_i^{(1)}|\leq C K^{\kappa\ve}
 \left(\ell_N^{-1}\wt(x_i/\ell_N)
       +\frac{N K^{\bar\alpha}}{(1+N\operatorname{dist}(x_i,\mathbb Z))^{1+s/2}}\right).
\end{equation}

For the quadrature in $\dV_i^{(2)}$, set
\begin{align*}
 F_x(v)&=N^{-s}\left(
 \partial_1\zeta_0(x,x+v/N)\tilde g(v/N)
 -\zeta_0(x,x+v/N)\tilde g'(v/N)\right),\\
 R_x(v)&=-N^{-s}(v/N)^2
                \partial_1\zeta_0(x,x+v/N)g''(v/N).
\end{align*}
The derivative identities give
\[
 F_x(v)=-N^{-s}\psi'(x/\ell_N)\tilde g'(v/N)+R_x(v).
\]
For $1/2\leq|v|\leq3N/4$, differentiation uses only $\psi''$ and gives
\begin{align*}
 |R_x(v)|&\leq C\ell_N^{-1}
 (\wt(x/\ell_N)+\wt((x+v/N)/\ell_N))|v|^{-s},\\
 |R_x'(v)|&\leq C\ell_N^{-1}
 (\wt(x/\ell_N)+\wt((x+v/N)/\ell_N))|v|^{-1-s}.
\end{align*}
For integers $k$ with $0<|k|\leq N/2$, the cell estimates are
\begin{align*}
 \int_{-1/2}^{1/2}|R_x(v)|\dd v
 &\leq C\ell_N^{-1}\wt(x/\ell_N),\\
 \left|R_x(k)-\int_{k-1/2}^{k+1/2}R_x(v)\dd v\right|
 &\leq\int_{k-1/2}^{k+1/2}|R_x'(v)|\dd v.
\end{align*}
Summation and~\eqref{eq:app wt convolution} bound these errors by the
right-hand side of~\eqref{eq:app V reference}.

In a partial window, the distance between the reference endpoints and the background
endpoints is at most $CK^\ve(1+|i|)^{s/2}/N$.
The boundary strips contribute zero if
$\psi(x_i/\ell_N)=\psi'(x_i/\ell_N)=0$.
Otherwise $|i|\leq n_\gamma/3$, the strips are at microscopic distance at least
$cn_\gamma$, and their contribution is at most
\[
 CK^\ve\ell_N^{-1}\wt(x_i/\ell_N)(1+|i|)^{s/2}n_\gamma^{-s}
 \leq C K^{\kappa\ve}\ell_N^{-1}\wt(x_i/\ell_N).
\]
The odd part of $F_x$ contributes at most $CN|\psi'(x_i/\ell_N)|n_\gamma^{-s}$.
Comparison with the two terms in~\eqref{eq:app V reference}, for
$|i|\leq K$ and $K<|i|\leq n_\gamma/3$, respectively, reduces to
\[
 Kn_\gamma^{-s}+K^{2-s-\bar\alpha}n_\gamma^{-s}
 \max\bigl(K^{-1+3s/2},n_\gamma^{-1+3s/2}\bigr)\leq C,
 \qquad \gamma\geq\max(4,2/s).
\]
On the full circle, we use midpoint cells for odd $N$.
For even $N$, we keep the single antipodal cell. The odd integral vanishes, and since $\tilde g'$ is odd, the sum of $\tilde g'(k/N)$ over the nonzero integers $-N/2<k\leq N/2$ reduces to $\tilde g'(1/2)$. The kernel term therefore contributes at most $N^{-s}|\psi'(x/\ell_N)\tilde g'(1/2)|$, and the error of the antipodal cell is at most this term plus the cell errors.
Thus $\dV_i^{(2)}$ satisfies~\eqref{eq:app V reference} as well,
which proves~\eqref{eq:V1sup}.

\paragraph{\emph{Step 3: the gap field}}
Orient each unordered pair from $\ell$ to $m$ along its chosen positive
path, of length $k$ and lifted separation $r$. Define
\[
 \dW_i^{\mathrm{gap}}=-2N^{-1-s}
 \sum_{\{\ell,m\}}
 \zeta_0(x_\ell,x_m)
           \bigl(\tilde g'(r)-\tilde g'(k/N)\bigr),
\]
where the sum runs over the unordered pairs $\{\ell,m\}\subset I_\gamma$ whose chosen path from $\ell$ to $m$ contains the edge $i$.
Each pair occurs once, and partial-window paths stay in $I_\gamma$.
For $U=(u_i)\in\dR^N$, telescoping gives
\[
 \dW\cdot U=-\sum_i \dW_i^{\mathrm{gap}}N(u_{i+1}-u_i),
 \qquad \sum_i\dW_i=0.
\]
The good event and the mean value theorem yield
\[
 N^{-1-s}\left|\tilde g'(r)-\tilde g'(k/N)\right|
       \leq C K^{\kappa\ve} k^{-2-s/2}.
\]
Summing over paths through edge $i$, with one endpoint fixed, gives
\[
 |\dW_i^{\mathrm{gap}}|
 \leq C K^{\kappa\ve}
     \sum_{j\in I_\gamma}
       \frac{\wo(x_j/\ell_N)}{(1+d(i,j))^{1+s/2}}.
\]
The kernel is summable. Young's inequality and microscopic packing give
\[
 \sum_i|\dW_i^{\mathrm{gap}}|^2
 \leq C K^{\kappa\ve}\sum_{j\in I_\gamma}\wo(x_j/\ell_N)^2
 \leq C K^{\kappa\ve}\bigl(K+K^{2\bar\alpha}\bigr).
\]
The logarithms at $2\bar\alpha=1$ and $\bar\alpha=0$ are absorbed by $K^{\kappa\ve}$.
All estimates are uniform under common translations with
$|Nx_0|\leq K^\ve$. Summing~\eqref{eq:V1sup} by microscopic packing gives
\[
 \left|\sum_i\partial_i\Loop_{\ell_N,\gamma}[\psi]\right|
       =\left|\sum_i\dV_i\right|
       \leq C N K^{\bar\alpha+\kappa\ve}.
\]
For bounded $K$, the gap bounds give the same estimates after increasing $C$.
\end{proof}

\subsection{\texorpdfstring{Two moment bounds and the proof of Lemma~\ref{lemma:auxiliary Loop}}{Two moment bounds and the proof of Lemma 5.8}}

\begin{lemma}\label{lemma:Bell}
Let $N\ell_N\geq1$ and $\psi\in\mathcal C^\infty(\ell_N^{-1}\dT)$, and let $\dB_{\ell_N}[\psi]$ be as in~\eqref{def:Bell}. There exists $C>0$ depending only on $s$ and $\beta$ such that
\begin{equation}\label{eq:Bell general}
\left|\dE_{\dGi}[\dB_{\ell_N}[\psi]]\right|
\leq CN\ell_N\Vert\psi'\Vert_{L^2(\ell_N^{-1}\dT)}^2.
\end{equation}
\end{lemma}

\begin{proof}
All norms are on $\ell_N^{-1}\dT$. For $0<y<N$, let
\begin{equation*}
q(y)=N^{-s}\ell_N^2g''(y/N)\int_\dT\bigl(\psi(\ell_N^{-1}(x+y/N))-\psi(\ell_N^{-1}x)\bigr)^2\dd x\,.
\end{equation*}
Then $q(N-y)=q(y)$. For $y\leq N/2$, the fundamental theorem of calculus and Cauchy--Schwarz bound the integral by $(y/N)^2\ell_N^{-1}\Vert\psi'\Vert_2^2$ and its derivative in $y$ by $2yN^{-2}\ell_N^{-1}\Vert\psi'\Vert_2^2$. Since $|g^{(m)}(r)|\leq C|r|^{-s-m}$ for $0<|r|\leq1/2$ and $m=2,3$, it follows that
\begin{equation}\label{eq:averaged kernel bounds}
|q(y)|\leq\frac{C\ell_N\Vert\psi'\Vert_2^2}{\min(y,N-y)^s}
\quad\text{and}\quad
|q'(y)|\leq\frac{C\ell_N\Vert\psi'\Vert_2^2}{\min(y,N-y)^{1+s}}\,.
\end{equation}

Let $S_k=N(x_{1+k}-x_1)$ for $1\leq k<N$, using cyclic lifts. We expand the definition of $\dB_{\ell_N}[\psi]$ and average the uniform rotation, which is independent of the gaps. By cyclic symmetry,
\begin{equation*}
\dE_{\dGi}[\dB_{\ell_N}[\psi]]
=N\sum_{k=1}^{N-1}\dE_{\dGi}[q(S_k)]-N\int_0^Nq(y)\dd y\,,
\end{equation*}
where absolute integrability follows from $s<1$ and Lemma~\ref{lemma:no explosion}.

Fix $0<\ve<s/2$ sufficiently small for Theorem~\ref{theorem:almost optimal rigidity}. For $1\leq k\leq N/2$ and every $p>0$, this theorem, Lemma~\ref{lemma:no explosion} and Jensen's inequality give
\begin{equation*}
\dE_{\dGi}[S_k^{-p}+(N-S_k)^{-p}]\leq C_pk^{-p}
\quad\text{and}\quad
\dGi\bigl(|S_k-k|>k^{s/2+\ve}\bigr)\leq Ce^{-ck^{\delta(\ve)}}\,.
\end{equation*}
We split at $|S_k-k|=k^{s/2+\ve}$ and use~\eqref{eq:averaged kernel bounds}, Cauchy--Schwarz and $s/2+\ve<1$. This gives
\begin{equation*}
\left|\dE_{\dGi}[q(S_k)]-q(k)\right|
\leq C\ell_N\Vert\psi'\Vert_2^2
\bigl(k^{-1-s/2+\ve}+k^{-s}e^{-ck^{\delta(\ve)}}\bigr)\,.
\end{equation*}
Since $\ve<s/2$, summation and the symmetry $k\mapsto N-k$ bound the sum of these differences over $1\leq k<N$ by $C\ell_N\Vert\psi'\Vert_2^2$. Finally, on the cells $[k-1/2,k+1/2]$,~\eqref{eq:averaged kernel bounds} gives
\begin{equation*}
\left|\sum_{k=1}^{N-1}q(k)-\int_0^Nq(y)\dd y\right|
\leq2\int_0^{1/2}|q(y)|\dd y+\int_{1/2}^{N-1/2}|q'(y)|\dd y
\leq C\ell_N\Vert\psi'\Vert_2^2\,.
\end{equation*}
With the identity for $\dE_{\dGi}[\dB_{\ell_N}[\psi]]$, these two bounds prove~\eqref{eq:Bell general}.
\end{proof}

\begin{lemma}\label{lemma:crude loop}
Let $p\geq1$. Then there exists $C=C(\beta,s,p)>0$ such that, for every integer $N\geq1$, every $\ell_N\in(0,1]$ and every $\psi\in\mathcal C^1(\ell_N^{-1}\dT)$,
\begin{equation}\label{eq:crude loop bound}
\Vert\Loop_{\ell_N}[\psi]\Vert_{L^p(\dGi)}
\leq C\sum_{m=0}^{N-1}\sup_{m\leq z\leq m+1}\Bigl|\psi'\Bigl(\frac{z}{N\ell_N}\Bigr)\Bigr|\,.
\end{equation}
\end{lemma}

\begin{proof}
We work in microscopic coordinates on $N\dT=\dR/N\mathbb Z$. Write $z_i=Nx_i$ and $\nu=\sum_i\delta_{z_i}-\dd z$, and set $g_N(r)=N^{-s}g(r/N)$ and $b(z)=N\ell_N\psi(z/(N\ell_N))$. Then~\eqref{def:Loop} reads
\begin{equation*}
\Loop_{\ell_N}[\psi]=\iint_{z\neq t}(b(z)-b(t))g_N'(z-t)\dd\nu(z)\dd\nu(t)\,,
\end{equation*}
and $b'(z)=\psi'(z/(N\ell_N))$.

\paragraph{\emph{Step 1: moment bounds}}
Let $n(I)$ count the particles in an interval $I$, so that $\nu(I)=n(I)-|I|$. Jensen's inequality applied to the gaps in $I$ gives
\begin{equation*}
(n(I)-1)_+^{p+1}\leq|I|^p\sum_{z_i\in I}\bigl(N(x_{i+1}-x_i)\bigr)^{-p}\,.
\end{equation*}
Rotation averaging and Lemma~\ref{lemma:no explosion} yield $\Vert n(I)\Vert_p\leq C(1+|I|)$. Fix $0<\delta<s/2$ sufficiently small for Theorem~\ref{theorem:almost optimal rigidity}. This theorem and Hölder's inequality give
\begin{equation}\label{eq:aux discrepancy moments}
\Vert\nu(I)\Vert_p\leq C|I|^{s/2+\delta}\qquad\text{for }1\leq|I|\leq N/2\,.
\end{equation}

We also need the pair energy at short distances. For $1\leq k<N$, using cyclic lifts, the inverse block moments give
\begin{equation*}
\bigl\Vert(N(x_{i+k}-x_i))^{-s}\mathds{1}_{N(x_{i+k}-x_i)\leq3}\bigr\Vert_p
\leq9\bigl\Vert(N(x_{i+k}-x_i))^{-s-2}\bigr\Vert_p\leq Ck^{-s-2}\,.
\end{equation*}
Summing in both orientations bounds every moment of
\begin{equation*}
\mathcal E_i=\sum_{j\ne i,\ \operatorname{dist}(z_i,z_j)\leq3}\operatorname{dist}(z_i,z_j)^{-s}\,.
\end{equation*}
If $|I|\leq3$ and $z_i\in I$, then $n(I)\leq1+3^s\mathcal E_i$. Rotation averaging therefore gives
\begin{equation*}
\dE_{\dGi}\Bigl[\Bigl(\sum_{z_i\in I}\mathcal E_i\Bigr)^p\Bigr]
\leq|I|\dE_{\dGi}[(1+3^s\mathcal E_1)^{p-1}\mathcal E_1^p]\leq C\qquad\text{for }|I|\leq3\,.
\end{equation*}

\paragraph{\emph{Step 2: an identity for the loop equation term}}
For $t\in N\dT$ outside the particle positions and $u,v>0$, let $A_t(u)=\nu((t-u,t))$ and $B_t(v)=\nu((t,t+v))$. For each pair, we integrate $b'$ along the shorter arc between the two points. An arc of length $d$ carries the factor $g_N'(d)$, and $d|g_N'(d)|\leq Cd^{-s}$, so Fubini's theorem applies since $s<1$. For $0<d<N/2$, integration by parts and $g_N'(N/2)=0$ give $-g_N'(d)=-\int_d^{N/2}(r-d)g_N'''(r)\dd r+(N/2-d)g_N''(N/2)$. If the arc runs from $z$ to $w$ through $t$, then $(r-d)_+$ is the length of the set of $u\in(0,r)$ with $z\in(t-u,t)$ and $w\in(t,t+r-u)$. Substituting $r=u+v$, we obtain
\begin{equation*}
\Loop_{\ell_N}[\psi]=-2\int_{N\dT}b'(t)\Lambda(t)\dd t\,,
\end{equation*}
where
\begin{equation*}
\Lambda(t)=-\iint_{\substack{u,v>0\\u+v<N/2}}g_N'''(u+v)A_t(u)B_t(v)\dd u\dd v
+g_N''(N/2)\int_0^{N/2}A_t(u)B_t(N/2-u)\dd u\,.
\end{equation*}
The last exchange of integrals is absolute. The terms containing a particle vanish for small $u+v$, since $t$ is not a particle position, and the remaining term is bounded by $Cuv(u+v)^{-s-3}\leq C(uv)^{-(1+s)/2}$.

\paragraph{\emph{Step 3: the bound on each cell}}
We show that $\Vert\int_m^{m+1}|\Lambda(t)|\dd t\Vert_p\leq C$ for $0\leq m<N$. By Step 2, this proves~\eqref{eq:crude loop bound}.

The count bound and~\eqref{eq:aux discrepancy moments} give $\Vert A_t(u)\Vert_{2p}\leq C(1+u)^{s/2+\delta}$, and likewise for $B_t(v)$. For $u+v\geq1$, the bound $|g_N'''(r)|\leq Cr^{-s-3}$ gives $|g_N'''(u+v)|\leq C((1+u)(1+v))^{-(s+3)/2}$. H\"older's inequality in $(u,v)$ with the integrable weight $((1+u)(1+v))^{-5/4+\delta/2}$ then bounds the $L^p$ norm of this part of $\Lambda(t)$ by
\begin{equation*}
C\Bigl(\int_0^\infty(1+u)^{-p(3/2-\delta)+(p-1)(5/4-\delta/2)}\dd u\Bigr)^{2/p}\leq C\,,
\end{equation*}
since the exponent equals $-1-(1/2-\delta)(p+1)/2$. H\"older's inequality in $u$ bounds that of the boundary term by $CN^{-1+2\delta}\leq C$.
For $u+v<1$, we integrate over $t\in[m,m+1]$ first and expand $A_t(u)B_t(v)$, where $A_t(u)+u$ and $B_t(v)+v$ are the numbers of particles in $(t-u,t)$ and $(t,t+v)$. Since $s<1$, the terms containing $u$ or $v$ contribute at most $C(1+n([m-1,m+2]))$. In the product of the two particle numbers, each pair of particles at distance $d<1$ contributes at most
\begin{equation*}
Cd\int_d^\infty(r-d)r^{-s-3}\dd r\leq Cd^{-s}\,.
\end{equation*}
The bounds of Step 1 on $n([m-1,m+2])$ and on the pair energy in $[m-1,m+2]$ conclude the proof.
\end{proof}

\begin{proof}[Proof of Lemma~\ref{lemma:auxiliary Loop}]
Recall that $K=N\ell_N$ and $\bar\alpha=1-s+\alpha$. We use the microscopic positions $z_i$, the measure $\nu$ and the kernel $g_N$ from the proof of Lemma~\ref{lemma:crude loop}. Fix $p\geq1$, and fix $0<\delta<s/4$ sufficiently small for Theorem~\ref{theorem:almost optimal rigidity}. We may assume that $K$ is large, and we let $M=K^\gamma$.

\paragraph{\emph{Step 1: cutoff of the transport}}
If $M\geq N/8$, then $\psi_\gamma=\psi$ and there is nothing to prove in this step. Suppose that $M<N/8$. The function $\psi-\psi_\gamma$ is constant for $|u|\leq M/(8K)$. For larger $|u|$, the bound on $\psi'$ and the bound preceding Lemma~\ref{lemma:auxiliary Loop} give
\[
|\psi'(u)-\psi_\gamma'(u)|\leq C|u|^{s-2}+\frac{CK}{M}|u|^{s-1}\mathds{1}_{|u|\leq M/(4K)}\,.
\]
Since $\Loop_{\ell_N}$ is linear and vanishes on constants, we apply Lemma~\ref{lemma:crude loop} to $\psi-\psi_\gamma$. Only the cells at distance at least $M/8$ from $0$ contribute, and summing the above bound over them gives
\begin{equation*}
\Vert\Loop_{\ell_N}[\psi]-\Loop_{\ell_N}[\psi_\gamma]\Vert_p\leq CK^{2-s}M^{s-1}\,.
\end{equation*}
As in the proof of Lemma~\ref{lemma:crude loop}, the cut-off transport becomes $b_0(z)=K\psi_\gamma(z/K)$ in microscopic coordinates. It satisfies
\[
\operatorname{supp}b_0\subset[-M/4,M/4],\qquad |b_0|\leq CK,\qquad |b_0'|\leq CK^{\bar\alpha}\,.
\]

\paragraph{\emph{Step 2: localization around the particle preceding $0$}}
Suppose that $M<N/8$. The remaining case is treated at the end of Step 4. Let $J$ be the label of the particle preceding $0$, and let $D=-z_J$. As in Step 6 of the proof of Lemma~\ref{lemma:a priori}, set $y_i=x_{J+i-1}-x_J$ for $2\leq i\leq N$ and $w(y)=Ny_2$. Uniform rotation and cyclic symmetry give the joint law
\begin{equation}\label{eq:aux size bias}
\mathds{1}_{0<D<w(y)}\dd D\dd\dGi^*(y)\quad\text{with}\quad \dE_{\dGi^*}[w]=1\,.
\end{equation}
Thus $y$ has law $w\dGi^*$. Hölder's inequality and~\eqref{eq:gap upper exponential} transfer the gap moment bounds and the rigidity estimates to this law.

Let $I$ be the set of labels $J+j$ with $|j|\leq M$. Let $\mathcal G$ be the event that $D\leq M/8$, $z_{J+\lfloor M\rfloor}\geq M/2$ and $z_{J-\lfloor M\rfloor}\leq-M/2$. On $\mathcal G$, the particles with labels outside $I$ lie outside $(-M/2,M/2)$, hence outside the support of $b_0$. The rigidity estimates give, for some $\lambda>0$,
\[
\dGi(\mathcal G^c)\leq Ce^{-cM^\lambda}\,.
\]

Let $L_I$ be the localized statistic~\eqref{def:Loop0} for $\psi_\gamma$, with labels in $I$ instead of $I_\gamma$. In microscopic coordinates, its background is $[-M,M]$. For $|z|\leq M/4$, let $E_{\mathrm{out}}(z)$ be the force exerted at $z$ by the particles outside $I$ and by the background outside $[-M,M]$:
\begin{equation*}
E_{\mathrm{out}}(z)=\sum_{i\notin I}g_N'(z-z_i)-\int_{N\dT\setminus[-M,M]}g_N'(z-t)\dd t\,.
\end{equation*}
On $\mathcal G$, the interaction between two points outside the support of $b_0$ vanishes. Therefore, on $\mathcal G$, the difference $\Loop_{\ell_N}[\psi_\gamma]-L_I$ is the sum of
\begin{equation*}
2\int_{|z|\leq M/4}b_0(z)E_{\mathrm{out}}(z)\dd\nu(z)
\end{equation*}
and of the interaction of the background on $[-M,M]$ with itself. This last term does not depend on the particles, so it does not affect variances. It is at most $CK^{\bar\alpha}M^2$ in absolute value.

\paragraph{\emph{Step 3: the exterior force}}
We bound $E_{\mathrm{out}}$ on $\mathcal G$ by comparing the exterior particles with the lattice. The labels outside $I$ are $J+j$ for integers $-N/2<j\leq N/2$ with $|j|>M$. For these $j$, the rigidity estimates give, with cyclic lifts,
\begin{equation*}
\Vert z_{J+j}-j\Vert_p\leq C|j|^{s/2+\delta}
\quad\text{and}\quad
\dGi(|z_{J+j}-j|>|j|/2)\leq Ce^{-c|j|^\lambda}\,.
\end{equation*}
Let $|z|\leq M/4$. If $|z_{J+j}-j|\leq|j|/2$, then the bound $|g_N''(r)|\leq C\operatorname{dist}(r,N\mathbb Z)^{-s-2}$ gives
\[
|g_N'(z-z_{J+j})-g_N'(z-j)|\leq C|z_{J+j}-j||j|^{-s-2}\,.
\]
Otherwise, on $\mathcal G$, both forces are at most $CM^{-1-s}$. Hence
\begin{equation*}
\Bigl\Vert\mathds{1}_{\mathcal G}\sup_{|z|\leq M/4}|g_N'(z-z_{J+j})-g_N'(z-j)|\Bigr\Vert_p
\leq C\bigl(|j|^{-2-s/2+\delta}+M^{-1-s}e^{-c|j|^\lambda}\bigr)\,.
\end{equation*}
We sum this bound over $j$. It remains to compare the sum of the lattice forces $g_N'(z-j)$ with the background integral. We do so cell by cell, and the fractional cells at $\pm M$ cost $CM^{-1-s}$. This gives
\begin{equation}\label{eq:aux exterior force}
\Bigl\Vert\mathds{1}_{\mathcal G}\sup_{|z|\leq M/4}|E_{\mathrm{out}}(z)|\Bigr\Vert_p
\leq CM^{-1-s/2+\delta}\,.
\end{equation}

\paragraph{\emph{Step 4: the localization error}}
We now show that the variance of $\Loop_{\ell_N}[\psi]-L_I$ is at most $CK$. On $\mathcal G$, the particles in $[-M/4,M/4]$ have labels in $I$, so there are at most $2\lfloor M\rfloor+1$ of them. Since $|b_0|\leq CK$, the estimate~\eqref{eq:aux exterior force} gives
\[
\Bigl\Vert\mathds{1}_{\mathcal G}\int_{|z|\leq M/4}b_0(z)E_{\mathrm{out}}(z)\dd\nu(z)\Bigr\Vert_2\leq CKM^{-s/2+\delta}\,.
\]
The inverse gap moments, $|b_0'|\leq CK^{\bar\alpha}$ and $s<1$ give $\Vert L_I\Vert_p\leq CK^{\bar\alpha}M^2$. Lemma~\ref{lemma:crude loop} bounds $\Vert\Loop_{\ell_N}[\psi_\gamma]\Vert_p$ by a fixed power of $K$. With the bound on $\dGi(\mathcal G^c)$ of Step 2, Hölder's inequality bounds the contribution of $\mathcal G^c$ by $CK^{-10}$. Combining these bounds with Steps 1 and 2, we conclude that, for $\gamma$ sufficiently large,
\begin{equation*}
\Var_{\dGi}\bigl[\Loop_{\ell_N}[\psi]-L_I\bigr]^{1/2}
\leq C\bigl(K^{2-s}M^{s-1}+KM^{-s/2+\delta}+K^{-10}\bigr)\leq CK^{1/2}\,.
\end{equation*}

If $M\geq N/8$, then we take all labels for $I$ and the whole circle for the background. In this case, $\Loop_{\ell_N}[\psi]-L_I$ does not depend on the particles, and $N\leq8M$ gives the same polynomial moment bounds.

\paragraph{\emph{Step 5: anchoring the localized term}}
Under~\eqref{eq:aux size bias}, we translate all particles by $D/N$, keeping the background and $\psi_\gamma$ fixed. After relabeling $J$ as $1$, this maps $L_I$ to $\Loop_{\ell_N,\gamma}[\psi_\gamma](0,y)$. It remains to compare this function with $F_\gamma^*(y)$.

Step 1 in the proof of Lemma~\ref{lemma:a priori}, with $2\gamma$ in place of $\gamma$, gives, for every $B>0$,
\begin{equation}\label{eq:aux cutoff probability}
\dGi(\eta_{\ve,2\gamma}\ne1)\leq C_BK^{-B}\,.
\end{equation}
This bound also holds under $w\dGi^*$ by Hölder's inequality. On $\{\eta_{\ve,2\gamma}(0,y)>0\}$, we have $D\leq w(y)\leq K^\ve$. On this event, the estimate at the end of Appendix~\ref{section:app weighted} gives, along the translation,
\begin{equation*}
\Bigl|\sum_i\partial_i\Loop_{\ell_N,\gamma}[\psi_\gamma]\Bigr|\leq CNK^{\bar\alpha+\kappa\ve}\,.
\end{equation*}
On $\{\eta_{\ve,2\gamma}(0,y)=1\}$, we integrate this bound along the translation, whose length is $D/N\leq K^\ve/N$. On the complementary event, we use the polynomial moment bounds and~\eqref{eq:aux cutoff probability} instead. After increasing $\kappa$, this gives
\begin{equation*}
\Vert L_I-F_\gamma^*(y)\Vert_2\leq CK^{\bar\alpha+\kappa\ve}+CK^{-10}\,.
\end{equation*}

\paragraph{\emph{Step 6: conclusion}}
By~\eqref{eq:aux size bias}, the random variable $F_\gamma^*(y)$ has the law of $F_\gamma^*$ under $w\dGi^*$. The variance triangle inequality therefore gives
\[
\Var_{\dGi}[\Loop_{\ell_N}[\psi]]^{1/2}
\leq\Var_{\dGi}\bigl[\Loop_{\ell_N}[\psi]-L_I\bigr]^{1/2}
+\Vert L_I-F_\gamma^*(y)\Vert_2
+\Var_{w\dGi^*}[F_\gamma^*]^{1/2}\,.
\]
The first two terms are bounded in Steps 4 and 5. For the last term, note that $F_\gamma^*$ vanishes unless $w(y)\leq K^\ve$, and that it is bounded by a fixed power of $K$. Since $w(y)\leq K^\ve$ wherever $\eta_{\ve,2\gamma}(0,y)>0$, the bound~\eqref{eq:aux cutoff probability} under $w\dGi^*$ yields
\begin{equation*}
\Var_{w\dGi^*}[F_\gamma^*]
\leq\dE_{\dGi^*}\bigl[w\bigl(F_\gamma^*-\dE_{\dGi^*}[F_\gamma^*]\bigr)^2\bigr]
\leq K^\ve\Var_{\dGi^*}[F_\gamma^*]+CK^{-10}\,.
\end{equation*}
Inserting these three bounds proves~\eqref{eq:aux step1}.
\end{proof}

\bibliographystyle{alpha}
\bibliography{bib.bib}

\end{document}